\documentclass[12pt]{article}
\usepackage{preprint-layout} %
\usepackage{preprint-notation} %

\usepackage{booktabs}
\usepackage{hyperref}
\usepackage{adjustbox}

\title{Stabilizing and Solving Unique Continuation Problems by Parameterizing Data and Learning Finite Element Solution Operators}
\date{\today}

\author{
Erik Burman,
Mats G. Larson, 
Karl Larsson,
Carl Lundholm
}

\begin{document}

\maketitle

\begin{abstract}
We consider an inverse problem involving the reconstruction of the solution to a nonlinear partial differential equation (PDE) with unknown boundary conditions. Instead of direct boundary data, we are provided with a large dataset of boundary observations for typical solutions (collective data) and a bulk measurement of a specific realization. To leverage this collective data, we first compress the boundary data using proper orthogonal decomposition (POD) in a linear expansion. Next, we identify a possible nonlinear low-dimensional structure in the expansion coefficients using an autoencoder, which provides a parametrization of the dataset in a lower-dimensional latent space. We then train an operator network to map the expansion coefficients representing the boundary data to the finite element (FE) solution of the PDE. Finally, we connect the autoencoder's decoder to the operator network which enables us to solve the inverse problem by optimizing a data-fitting term over the latent space. We analyze the underlying stabilized finite element method (FEM) in the linear setting and establish an optimal error estimate in the $H^1$-norm. The nonlinear problem is then studied numerically, demonstrating the effectiveness of our approach.
\end{abstract}

\section{Introduction}

Technological advances have led to measurement resolution and precision improvements, shifting the paradigm from data scarcity to abundance. While these data can potentially improve the reliability of computational predictions, it still needs to be determined how to consistently merge the data with physical models in the form of partial differential equations (PDE). In particular, if the PDE problem is ill-posed, as is typical for data assimilation problems, a delicate balancing problem of data accuracy and regularization strength has to be solved. If the data is inaccurate, the PDE problem requires strong regularization; however, if the data is accurate, such a strong regularization will destroy the accuracy of the approximation of the PDE. Another question is how to use different types of data. Some large data sets, consisting of historical data of events similar to the one under study, can be available.  In contrast, a small set of measurements characterizes the particular realization we want to model computationally.  In this case, the former data set measures the ``experience" of the physical phenomenon, while the latter gives information on the current event to be predicted.

This is the situation that we wish to address in the present work. The objective is to construct a computational method that combines machine learning techniques for the data handling parts and hybrid network/finite element methods (FEMs) for the approximation in physical space. First, the large data set is mapped to a lower-dimensional manifold using an autoencoder or some other technique for finding low-dimensional structures, such as singular value decomposition or manifold learning. Then, we train a network to reproduce the solution map from the lower-dimensional set to the finite element space. Finally, this reduced order model solves a nonlinear inverse problem under the a priori assumption that the solution resides in a neighborhood of the lower-dimensional manifold. 

To ensure an underpinning of the developed methods, we consider the case of a unique continuation problem for a nonlinear elliptic operator. That is, given some interior measurement (or measurements on the part of the boundary), a solution is reconstructed despite lacking boundary data on the part of the boundary. Such problems are notoriously ill-posed, and using only the event data set, it is known that the accuracy of any approximation in the whole domain cannot be guaranteed due to the poor global stability \cite{ARRV09}. Indeed, in general, stability is no better than logarithmic. This means that for perturbations of order $\epsilon$, the error must be expected to be of order $|\log(\epsilon)|^{-\alpha}$ with $\alpha \in (0,1)$. In interior subdomains stability is of H\"older type, meaning that the same perturbation gives rise to an $O(\epsilon^\alpha)$ error. Computational methods can have, at best, rates that reflect this stability of the continuous problem \cite{BNO24}. To improve on these estimates additional assumptions on the solution are needed. A convenient a priori assumption is that the missing data of the approximate solution is in a $\delta$-neighbourhood of a finite $N$-dimensional space, $\mcG$, where $\delta$ is the smallest distance from the solution to $\mcG$ in some suitable topology. In this case, it is known that the stability is Lipschitz; that is, the problem has similar stability properties to a well-posed problem, and finite element methods can be designed with optimal convergence up to the data approximation error $\delta$. For linear model problems discretized using piecewise affine finite element methods with mesh parameter $h$, one can prove the error bound \cite{BO24},
\[
\|u - u_h\|_{H^1(\Omega)} \leq C_N ( h + \delta)
\]
Here, $C_N$ is a constant that depends on the dimension $N$ of the data set $\mcG$, the geometry of the available event data, and the smoothness of the exact solution. In particular, $C_N$ typically grows exponentially in $N$.

Since the size of $N$ must be kept down, there is a clear disadvantage in using the full large dataset. Indeed, for $N$ sufficiently large, the experience data will have no effect. Instead, we wish to identify a lower-dimensional structure in the high-dimensional dataset, a lower-dimensional manifold such that the data resides in a $\delta$-neighbourhood of the manifold. For this task, one may use proper orthogonal decomposition in the linear case or neural network autoencoders in the general case.

In the linear case, the data fitting problem reduces to a linear system; however, an ill-conditioned optimization problem has to be solved in the nonlinear case, leading to repeated solutions of linearized finite element systems. To improve the efficiency of this step, we propose to train a network to encode the data to an FE map, giving fast evaluation of finite element approximations without solving the finite element system in the optimization.

The approach is analyzed in the linear case with error estimates for a stabilized FEM using the reduced order model.

\pagebreak
\paragraph{Contributions.}
\begin{itemize}
\item We prove that the inverse problem with boundary data in a finite-dimensional set $\mcG$ is stable and design a method that reconstructs the solution using the reduced order basis with the same dimension as $\mcG$. We prove optimal error bounds in the $H^1$-norm for this method, where the constant of the error bound grows exponentially with the dimension of $\mcG$.
\item In the situation where a large set of perturbed random data, $\mcG_S$, from the set $\mcG$ is available, we develop a practical method for the solution of the severely ill-posed inverse problem of unique continuation, leveraging the large dataset to improve the stability properties. In order to handle nonlinearity in the PDE operator and data efficiently we adopt machine learning algorithms. The machine learning techniques are used for the following two subproblems:
\begin{enumerate}
\item Identification of a potential latent space of $\mcG$ from $\mcG_S$ to find the smallest possible space for the inverse identification.
\item Construction of a discrete approximation of the solution map 
\begin{align}
\phi_{u} : \mcG \rightarrow H^1(\Omega)
\end{align}
that gives an approximation of the finite element solution to 
\begin{align}
\mcP(u) = 0 \quad \text{in $\Omega$}, \qquad u |_{\partial \Omega} \in \mcG
\end{align}
where $\mcP$ is the nonlinear PDE operator in question. The construction is done in a way that is a special case of the approach presented in \cite{larson2024nonlinoplearn} which in turn is a special case of an even more general approach presented in \cite{sharp_data-free_2023}.
\end{enumerate}
\item The performance of the combined finite element/machine learning approach is assessed against some academic data assimilation problems.
\end{itemize}

\paragraph{Previous Works.} 
The inverse problem we consider herein is of unique continuation type. There are many types of methods for this type of problem. In the framework we consider the earliest works considered quasi-reversibility \cite{Bour05}. The stabilized method we consider for unique continuation was first proposed in \cite{Bu13,Bu14} and \cite{BHL18}. More recent works use residual minimization in dual norm \cite{CIY22,DMS23,BHLL23}. The optimal error estimates for unique continuation with a trace in a finite-dimensional space was first considered for Dirichlet trace in \cite{BO24} and for Neumann trace in \cite{BOZ24}. The idea of combining unique continuation in finite-dimensional space with collective data was first proposed in \cite{Cor22, BJL24} using linear algebra methods for the compression and direct solution of the linear unique continuation problem. Low rank solvers for the solution of inverse problems have also been designed in \cite{RFL24} using proper orthogonal decomposition.

In recent years, significant advancements have been made in utilizing machine learning for solving PDEs \cite{MR4582511,MR4467422,bookDLnCP}. One important aspect is how to suitably and efficiently represent the learned solution \cite{MR4645137,MR4756922,PATEL2024116536,bachmayr2024}. An application that comes very natural in the context of neural networks is the derivation of reduced order models \cite{MR4619374,MR4549101}.

These developments are very useful in the context of inverse problems, where they have been utilized in both data- and model-driven inverse problems. In \cite{MR4260002} a combination of networks and traditional methods is considered to recover the diffusion coefficient in Poisson's and Burgers' equations. In general, the same is done in \cite{drad073} with the traditional method being FEM and the equations being elliptic and parabolic. Yet more examples of applying deep learning to this type of problem are given by \cite{CEN2023112427,BERG2021100008}. Anyone interested in the application of deep learning for PDE-solving has undoubtedly encountered Physics-Informed Neural Networks (PINNs) \cite{MR3881695} which are also used for inverse problems. Works not involving deep learning but still relevant are \cite{MR2684726, MR3341244} where projection-based reduced order models for inverse problems are presented. Taking the step to also include machine learning, some of the authors from the previous works give an overview of this mix in \cite{MR4298221}. Another overview of using machine learning for inverse problems is given by \cite{Arridge_2019}. In \cite{MR4832359}, an approach to reduce the error introduced by using operator learning for inverse problems is studied. As a contrast, \cite{M1338460} instead uses machine learning to reduce the error introduced by approximate forward models. Focusing instead on the other side of the computational spectrum, i.e., speed, \cite{MR4597397} presents a physics-based deep learning methodology with applications to optimal control. The work \cite{DASGUPTA2024116682} presents a modular machine learning framework for solving inverse problems in a latent space. Although using different techniques and approaches, this general description also holds for what we present here.

\paragraph{Comparison between this Work and Others.}
To the best of our knowledge, there are no previous works in the literature with a similar theoretical foundation addressing this type of data assimilation problem. Notably, the importance of the finite dimensionality of boundary data for stability, and thus the necessity to reduce the dimension of measured population data as much as possible, has only been considered in \cite{BJL24} using classical methods. Here, we apply autoencoders and operator learning to this problem for the first time. To provide context on how other approaches might perform compared to ours, we note that the stabilized method proposed here yields optimally converging approximations, contingent on the properties of the finite element (FE) space and the stability of the inverse problem. This is not the case for Tikhonov regularized approaches, where discretization is typically applied without further consideration of numerical stability. Bayesian inference methods usually share a similar shortcoming, depending on the choice of prior. We note that in our computational examples, stabilization was not necessary, indicating that the space discretization was sufficiently well-resolved. The use of PINNs in this context leads to a formulation where the strong form of the PDE is minimized. This presents complications, as boundary conditions are generally difficult to impose in network approximations, particularly on the finite-dimensional subspace. Additionally, there appears to be no way to eliminate spurious local minima in the PDE approximation when using PINNs. In our case, since we minimize a convex functional over the finite element space for all parameter values, the space discretization part does not suffer from this defect. Nonetheless, the optimization could converge to local minima when networks approximate the operator, a common shortcoming with network approximation methods.

Concerning the approach to learning the physical model, the method we use is presented in detail in \cite{larson2024nonlinoplearn}, where the focus lies on the method itself as opposed to here, where the focus is on applying it to inverse problems. In \cite{larson2024nonlinoplearn}, a comparison with other machine learning approaches is made so we refer to this work for details and only give a brief characterization here:
\begin{itemize}
\item The core concept is to learn a finite element solution operator. The output is thus an approximation of a finite element solution. An advantage of this is that the method can be combined with standard FEM for support and enhancement in both theory and practice.
\item A multilayer perceptron (MLP) is used to approximate the solution operator.
\item The finite element part enters by using a mesh and basing the loss function on an energy functional that when minimized gives the FE-solution. An alternative is to use the weak residual, which although is more general seems to be more computationally costly.
\item The input to the network is a parameterization of problem data, e.g., right-hand side functions and boundary values. The network thus learns a parameterized family of PDE-problems as opposed to only a single problem.
\item The method is by default data-free, meaning no input-output data pairs. Instead input is sampled from probability distributions. However, the method allows for the incorporation of data sets as demonstrated here.
\end {itemize}
None of these individual features is new in physics-based deep learning, but to the best of our knowledge, this specific combination of them has not been studied outside of this work and \cite{larson2024nonlinoplearn}.

Looking through the literature for other works employing deep learning for unique continuation problems, we find two different types. The first type presented in \cite{ivagnes_mlpipeline_romips_2023} is a data-driven approach for parameterizing both boundary conditions and solutions for flow problems in a cylinder. The velocity distribution is observed in a downstream cross section and the objective is to find a matching inlet profile. Although numerical PDE simulations are used to generate data, this learning approach is physics-free in the contextual sense. The second more common type uses PINNs and seeks the full solution given pointwise observations in a subdomain of the solution domain. In \cite{mishra_pinnip_2021} four standard linear problems (Poisson's equation, the heat and wave equations, and Stokes flow) are considered. For a nonlinear problem, see \cite{escapil_hanalysis_2023} where the 2D Navier--Stokes equations are studied. The work \cite{nechita_helmholtzpinn_2023} again considers a linear problem, the Helmholtz equation. A drawback of these PINNs-based works is that the neural networks only learn a \emph{single} PDE solution during training. Comparing these works with ours, we see some differences with our approach: First, it is physics-based in contrast to \cite{ivagnes_mlpipeline_romips_2023}. Second, it allows for learning an entire class of related PDE solutions as opposed to a single one as in the PINNs-based works. With these points in mind, we think that our deep learning approach to unique continuation problems provides a novelty that can further the field.

\paragraph{Outline.} In Section 2, we introduce the model problem and the finite element discretization; in Section 3, we present and prove stability and error estimates for a linear model problem; in Section 4, we develop a machine learning-based approach for solving the inverse problem; in Section 5, we present several numerical examples illustrating the performance of the method for various complexity of the given 
set of boundary data; and in Section 6 we summarize our findings and discuss future research directions.

\paragraph{Notation.}
\begin{itemize}
\item We use $\lesssim$ to mean that there is a positive constant in the inequality (typically on the right-hand side).
\item For a bounded domain or a set of mesh features $D$, we denote by $\| \cdot \|_D$ and $(\cdot, \cdot)_D$ the standard $L^2(D)$-norm and inner product, respectively. Some common instances in the text are $D = \omega, \Omega, \partial \Omega, \mcF_h$. 
\item We denote by $\| \cdot \|_{\IR^N}$ the standard absolute value for vectors in $\IR^N$. We note that it should not be confused with the $L^2$-notation in the previous point. The reason for using the notation $\| \cdot \|_{\IR^N}$ is because we use it on the expansion coefficients of functions, thus making expressions involving norms on both functions and their coefficients more general and consistent.
\item For a positive-definite bilinear form $B$ we denote the corresponding norm by $\| \cdot \|_B$, i.e., $\| v \|_B^2 := B(v, v)$. An example from the text is $B = m_h$.
\end{itemize}

\section{Inverse Problem and Finite Element Method}
\subsection{Inverse Problem}

Let $\Omega$ be a domain in $\IR^d$, $\omega \subset \Omega$ a subdomain, and consider the minimization problem
\begin{align}\label{eq:inv-prob}
\boxed{ \inf_{v \in V} \frac12 \| u_0 - v \|_{L^2(\omega)}^2 \quad \text{subject to} \quad \mcP(v) = 0 \quad \text{in $\Omega$} }
\end{align}
where $\mcP(\cdot)$ is a nonlinear second order differential operator and 
$u_0$ is an observation of the solution in the subdomain $\omega$.
Note that we do not have access to boundary conditions for the partial differential equation; we only know that $\mcP(u) = 0$ in $\Omega$, and 
thus, the problem is, in general, ill-posed.

Assume that we have access to a dataset 
\begin{align}
\mcG \subset H^{1/2} (\partial \Omega)
\end{align}
of observed Dirichlet data at the boundary. The dataset 
$\mcG$ may have different properties, but here we will assume 
that it is of the form
\begin{equation}\label{def:G}
\mcG = \Bigl\{ g \in H^{1/2}(\partial \Omega) \,|\, 
g = \sum_{i=1}^N a_i \varphi_i, \ a_i \in I_i \Bigr\}
\end{equation}
where $I_i$ are bounded intervals and $\varphi_i \in H^{1/2}(\partial \Omega)$. Below we will also consider access to a finite set $\mcG_S \subset \mcG$ of samples from $\mcG$,
\begin{equation}\label{def:GS}
\mcG_S = \{ g_i \,|\, i \in I_S\}
\end{equation} 
where $I_S$ is some index set.

Including $v|_{\partial \Omega} \in \mcG$ as a constraint leads to
\begin{align}\label{eq:inverse}
\boxed{ \inf_{v \in V} \frac12 \| u_0 - v \|_{L^2(\omega)}^2 \quad \text{subject to} \quad \mcP(v) = 0 \quad \text{in $\Omega$}, \quad v|_{\partial \Omega} \in \mcG}
\end{align}
A schematic illustration of a problem of form \eqref{eq:inverse} is given in Figure~\ref{fig:schematic-problem-setting}.

\begin{figure}
  \centering
  \includegraphics[width=0.5\linewidth]{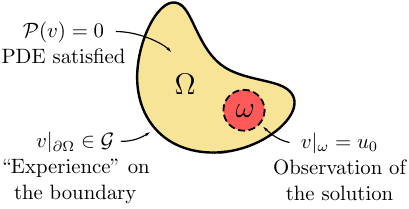}
  \caption{Schematic view of the minimization problem setup where we seek the $v\in V$ that minimizes the error in the observation of the solution, while under a PDE constraint with boundary conditions according to experience.}
  \label{fig:schematic-problem-setting}
\end{figure}

\subsection{Finite Element Method}

Let $V_h$ be a finite element space on a quasi-uniform partition $\mcT_h$ of $\Omega$ into shape regular elements with mesh 
parameter $h\in (0,h_0]$ and assume that there is an interpolation operator $\pi_h : H^1(\Omega) \rightarrow V_h$ and a constant such that for all $T \in \mcT_h$,  
\begin{equation}\label{eq:interpol}
\| v - \pi_h v \|_{H^m(T)} \lesssim h^{k-m}\| v \|_{H^k(N(T))}
\end{equation}
for $0\leq m \leq k \leq p+1$. Here $N(T)$ is the union of all elements that share a node with $T$. 

The finite element discretization of (\ref{eq:inverse}) takes the form
\begin{align}\label{eq:inverse-fem}
\boxed{ \inf_{v \in V_h} \frac12 \| u_0 - v \|_{L^2(\omega)}^2 \quad \text{subject to} \quad (\mcP(v),w)_{H^{-1}(\Omega), H^1(\Omega)} = 0  \quad  \forall w \in V_{h,0},
\quad v|_{\partial \Omega} \in \pi_h \mcG  }
\end{align}
where $ V_{h,0} = V_h \cap H^1_0(\Omega)$.

\section{Analysis for a Linear Model Problem}

In this section, we present theoretical results for a linear model problem.  We show that the finite dimensionality leads to a well-posed continuous problem, which may, however, have insufficient stability that may cause problems in the corresponding discrete problem.  We, therefore, introduce a stabilized formulation that retains the stability properties from the continuous problem, and then we prove error estimates.  

\subsection{The Continuous Problem}
Consider the linear model problem 
\begin{align}
\mcP u = 0\quad \text{in $\Omega$},\qquad
u = g \quad \text{on $\partial \Omega$}
\end{align}
where $\mcP = - \Delta$  and 
\begin{align}
g \in \mcG = \Bigl\{\sum_{n=1}^N a_n g_n \,|\, a_n \in \IR \Bigr\}
\end{align}
where the functions $\{g_n\}_{n=1}^N$ are linearly independent 
on $\partial \Omega$. Then with 
\begin{align}\label{eq:basis}
\mcP \varphi_n = 0\quad \text{in $\Omega$},\qquad
\varphi_n = g_n \quad \text{on $\partial \Omega$}
\end{align}
we may express $u$ as the linear combination
\begin{align}
u = \sum_{n=1}^N \hatu_n \varphi_n
\end{align}
where $\hatu \in \IR^N$ is the coefficient vector. The inverse problem (\ref{eq:inv-prob}) is then equivalent to 
computing the $L^2(\omega)$-projection of $u_0$ on 
$\mcU_N = \text{span}\{\varphi_n\}_{n=1}^N$,
\begin{equation}\label{eq:projexactmodes}
\boxed{u_N \in \mcU_N:\qquad    (u_N,w)_{\omega} = (u_0,w)_{\omega} \qquad \forall w \in \mcU_N}
\end{equation}
This is a finite-dimensional problem, and therefore, existence follows from uniqueness.  
To prove uniqueness, consider two solutions $u_1$ and $u_2$, we then have
\begin{equation}
(u_1-u_2,w)_{\omega} =0 \qquad \forall w \in \mcU_N
\end{equation}
and taking $w= u_1 - u_2$ gives
\begin{equation}
\|u_1-u_2\|_{\omega} =0
\end{equation}
By unique continuation for harmonic functions, we conclude that $u_1 - u_2$ is zero on the boundary and 
therefore $u_1 = u_2$ since the set $\{g_n\}_{n=1}^N$ is linearly independent on 
$\partial \Omega$.   It follows that  $\{\varphi_n\}_{n=1}^N$ is linearly independent on $\omega$ and by 
finite dimensionality, there is a constant (it is $\lambda_{min}^{-1/2}$), such that 
\begin{align}\label{eq:stab-exact}
\boxed{ \| \hatv \|_{\IR^N} \lesssim \| v \|_{\omega} \qquad v \in \mcU_N }
\end{align}
Note, however, that the constant may be huge, reflecting the often near 
ill-posed nature of an inverse problem.  

\subsection{The Discrete Problem}

In practice, only an approximation of the basis $\{\varphi_n\}_{n=1}^N$ is available,  
since we observe data on the boundary and must solve for an approximate basis.  
Assuming that we compute an approximate basis $\{\varphi_{n,h}\}_{n=1}^N$ using 
Nitsche's method with continuous piecewise linears $V_h$,  defined on a triangulation 
$\mcT_h$ of $\Omega$, 
\begin{align}\label{eq:nitsche}
\boxed{\varphi_{n,h} \in V_h: \qquad a_h(\varphi_{n,h},v) = l_{h, \varphi_n}(v)\qquad \forall v \in V_h}
\end{align}
where the forms are defined by
\begin{align}
a_h(v,w) &= (\nabla v, \nabla w)_\Omega - (\nabla_n v,  w)_{\partial \Omega} - (v, \nabla_n w)_{\partial\Omega} + \beta h^{-1} (v,w)_{\partial \Omega} 
\\
l_{h,g} (v) &= - (g,\nabla_n v)_{\partial\Omega} + \beta h^{-1} (g,  v)_{\partial \Omega} 
\end{align}
with $g$ the given Dirichlet data on $\partial \Omega$,  we have the error estimates 
\begin{align}\label{eq:basis-error-est}
\|\varphi_n - \varphi_{n,h} \|_\Omega 
+ h \|\nabla(\varphi_n - \varphi_{n,h})\|_\Omega 
\lesssim h^2 \|\varphi_n\|_{H^2(\Omega)} \lesssim h^2 \| g_n \|_{H^{3/2}(\partial \Omega)}
\end{align}
provided the regularity estimate $\| \varphi_n \|_{H^2(\Omega)} \lesssim \| g_n \|_{H^{3/2}(\partial \Omega)}$ 
holds, which is the case for convex or smooth domains.  

Next, we define the operators
\begin{align}\label{eq:I}
I: \IR^N \ni \hatv &\mapsto \sum_{n=1}^N \hatv_{n} \varphi_{n} \in \mcU_{N}
\\ \label{eq:Ih}
I_h: \IR^N \ni \hatv &\mapsto \sum_{n=1}^N \hatv_{n} \varphi_{n,h} \in \mcU_{N,h}
\end{align}
to represent linear combinations given coefficient vectors, where $\mcU_{N,h} = \text{span}\{\varphi_{n,h}\}_{n=1}^N$. By composing $I$ and $I_h$ 
with the coefficient extraction operator $\widehat{\cdot}$, we note that $v = I \hatv$ for $v \in \mcU_{N}$ 
and $v = I_h \hatv$ for $\mcU_{N,h}$. We also note that $I_h \hatv$ is the Galerkin 
approximation  defined by (\ref{eq:nitsche}) of $v = I \hatv$,  since $\varphi_{n,h}$ is the Galerkin approximation 
of $\varphi_n$ for $n=1,\dots,N$,  and we have the error estimate
\begin{align}\label{eq:Ih-error}
\| (I  - I_h) \hatv\|_\Omega + h \| \nabla (I  - I_h) \hatv \|_\Omega + h^{1/2}  \| (I- I_h) \hatv \|_{\partial \Omega}
&\lesssim h^2 \| \hatv \|_{\IR^N}  
\end{align}
The estimate (\ref{eq:Ih-error}) follows directly using the Cauchy--Schwarz inequality and 
the error estimates (\ref{eq:basis-error-est}) for the approximate basis 
\begin{align}
\| \nabla^m (v - I_h \hatv) \|^2_\Omega  &=  \Big( \sum_{n=1}^N v_n^2 \Big) \Big( \sum_{n=1}^N \|\nabla^m( \varphi_n - \varphi_{n,h}) \|^2_\Omega \Big)
\\
&\lesssim h^{2(2-m)} \Big( \sum_{n=1}^N v_n^2 \Big) \Big(\sum_{n=1}^N \| g_n \|^2_{H^{3/2}(\partial \Omega)} \Big)
\end{align}
with $m=0,1$.

Now if we proceed as in (\ref{eq:projexactmodes}) with the modes $\varphi_n$ replaced by the approximate 
modes $\varphi_{n,h}$, we can not directly use the same argument as in the continuous case to show that there 
is a unique solution since the discrete method does not possess the unique continuation property, and it doesn't 
appear easy to quantify how small the mesh size must be to guarantee that the bound (\ref{eq:stab-exact}) holds 
on $\mcU_{N,h}$. %

To quantify the discrete stability, note that the constant in (\ref{eq:stab-exact}) is  
characterized by the Rayleigh quotient 
\begin{align}
\lambda_{\text{min}} = \min_{\hatv \in \IR^N} \frac{\|I \hatv\|^2_\omega}{\| \hatv \|
^2_{\IR^N}}
\end{align}
and for the corresponding discrete estimate 
\begin{align}\label{eq:stab-discrete}
\| \hatv \|^2_{\IR^N} \lesssim  \| v \|^2_\omega
\end{align}
we instead have the constant 
\begin{align}
\lambda_{h,\text{min}} = \min_{\hatv \in \IR^N} \frac{\|I_h \hatv\|^2_\omega}{\| \hatv \|
^2_{\IR^N}}
\end{align}
Using the triangle inequality and the error estimate (\ref{eq:basis-error-est})  we have 
\begin{align}
\| I_h \hatv \|_{\omega} \geq \| I \hatv \|_{\omega} - \| (I_h - I) \hatv \|_{\omega} \geq \| I \hatv \|_{\omega} - c h^2 \| \hatv \|_{\IR^N}
\end{align}
and thus we may conclude that 
\begin{align}
\lambda_{\text{min},h} \geq (\lambda_{\text{min}}^{1/2} - c h^2)^2 \geq c \lambda_{\text{min}}
\end{align}
for $h< h_0$ with $h_0$ small enough. Thus for $h$ small enough the discrete bound (\ref{eq:stab-discrete}) holds but we note that the 
precise characterization of how small $h$ has to be appears difficult.

To handle this difficulty, let us instead consider the stabilized form
\begin{align}\label{eq:mh}
\boxed{m_h(v,w) = (v,w)_{\omega} + s_{h,\partial \Omega}(v,w)+ s_h(v,w)}
\end{align}
Here
\begin{align}
s_{h,\partial \Omega}(v,w) = h^{-1} (v- I \hatv,  w - I \hatw)_{\partial \Omega} + h (\nabla_T (v - I_h \hatv) , \nabla_T (w - I \hatw))_{\partial \Omega} 
\end{align}
where $\nabla_T = (I- n\otimes n) \nabla$ is the tangential derivative on $\partial \Omega$ with $n$ denoting the unit normal to $\partial \Omega$.
 The form $s_h$ is the standard 
normal gradient jump penalty term
\begin{equation}\label{eq:sh}
s_h(v,w) = \sum_{F \in \mcF_h} h([\nabla v], [\nabla w])_F 
\end{equation}
where $\mcF_h$ is the interior faces in the mesh $\mcT_h$.  %
The role of the form $s_{h,\partial \Omega}$ is to give control of the distance of the approximation to the finite dimensional set $\mathcal{G}$ in the $H^{1/2}(\partial \Omega)$-norm. In principle
the form 
\begin{equation}
s_{\partial \Omega}(v,w) = (v - I \hatv,   w - I \hatw)_{H^{1/2}(\partial \Omega)}
\end{equation}
could be used directly, but
to obtain a stabilization term that is easier to handle in practice we note that by the Galigliardo-Nirenberg inequality, $\|v\|_{H^{1/2}(\partial \Omega)} \lesssim \|v\|_{L^2(\partial \Omega)}^{1/2} \|v\|_{H^{1}(\partial \Omega)}^{1/2}$, there holds $s_{\partial \Omega}(v,v) \lesssim s_{h,\partial \Omega}(v,v)$, which is sufficient for stability.

\subsection{Error Estimates}

Our first result is that the additional stabilization terms in $m_h$ ensure that we have stability for the discrete problem 
similar to (\ref{eq:stab-exact}) that holds for the exact problem.
\begin{lem} Let $m_h$ be defined by (\ref{eq:mh}).  Then, there is a constant, depending on $N$ but not $h$, such that
\begin{align}\label{eq:stab-mh}
\boxed{\| \hatv \|_{\IR^N} \lesssim \| v \|_{m_h} \qquad v \in \mcU_{N,h}}
\end{align}
\end{lem}
\begin{proof} For $v \in \mcU_{N,h}$ we get by using the stability (\ref{eq:stab-exact}) on $\mcU_N$,  adding 
and subtracting $I_h \hatv$,  and employing the triangle inequality,  
\begin{align}\label{eq:stab-h-prf-a}
\| \hatv \|_{\IR^N} \lesssim \| I \hatv \|^2_\omega \lesssim  \| I_h \hatv \|^2_\omega +   \| ( I - I_h) \hatv \|^2_\omega
=   \| v \|^2_\omega +   \| (I-I_h) \hatv \|^2_\omega
\end{align}
where we finally  used the identity $I_h \hatv = v$,  which holds since $v \in \mcU_{N,h}$.  Next, we bound the second term using the stabilizing terms in $m_h$.  To that end, we observe that we have the orthogonality
\begin{align}
a_h((I - I_h) \hatv,  w) = 0\qquad \forall w \in V_h
\end{align}
since the discrete basis is, a Galerkin projection (\ref{eq:nitsche}) of the exact basis with respect to the Nitsche form $a_h$.  
Using the dual problem 
\begin{align}\label{eq:dual}
\mcP \phi = \psi \quad \text{in $\Omega$}, 
\qquad 
\phi = 0 \quad \text{on $\partial \Omega$}
\end{align} 
we obtain by partial integration followed by Galerkin orthogonality
\begin{align}
( ( I - I_h) \hatv, \psi)_\Omega = ( ( I - I_h) \hatv, \mcP \phi)_\Omega = a_h ( ( I - I_h) \hatv, \phi) =  a_h ( ( I - I_h) \hatv, \phi - \pi_h \phi)  
\end{align}
where $\pi_h: H^1(\Omega) \rightarrow V_h$ is the interpolation operator. Performing another partial integration, we get
\begin{align}
&a_h ( ( I - I_h) \hatv, \phi - \pi_h \phi)
\\
&= ([\nabla_n(I - I_h)\hatv],\phi - \pi_h \phi)_{\mcF_h} - ((I - I_h)\hatv,  \nabla_n( \phi - \pi_h \phi) )_{\partial \Omega}
\\
&\lesssim (  h^{3/2} \|[\nabla_n(I - I_h)\hatv]\|_{\mcF_h}  + h^{1/2} \|(I - I_h)\hatv\|_{\partial \Omega}  ) \| \phi \|_{H^2(\Omega)}
\end{align}
where we used the standard trace inequality $\|w\|^2_{\partial T} \lesssim h^{-1} \| w \|_T^2 + h \| \nabla w \|^2_T$ for $w \in H^1(T)$  on an element $T\in \mcT_h$. Finally, using the elliptic regularity $ \| \phi \|_{H^2(\Omega)} \lesssim \| \psi \|_{\Omega}$, combining the results, and taking $\psi = ( I - I_h) \hatv$, we get
\begin{align}
\|( I - I_h) \hatv \|_\Omega &\lesssim  h^{3/2} \|[\nabla_n(I - I_h)\hatv]\|_{\mcF_h}  + h^{1/2} \|(I - I_h)\hatv\|_{\partial \Omega}
\\
&\lesssim h (\| v \|_{s_h} +  \| v \|_{s_{h,\partial \Omega} } ) 
\end{align}
which combined with (\ref{eq:stab-h-prf-a}) directly gives the desired estimate. 
\end{proof}

We define the stabilized projection,
\begin{align}\label{eq:proj-stab}
\boxed{u_{N,h} \in \mcU_{N,h}: \qquad m_h(u_{N,h},v) = (u_0,v)_{\omega} \qquad \forall v \in \mcU_{N,h}} 
\end{align}

We then have the following error estimate for the stabilized projection with approximate basis functions.
 
\begin{prop}\label{prop:energy-est} Let $u_N\in \mcU_N$ be defined by (\ref{eq:projexactmodes}) and $u_{N,h} \in \mcU_{N,h}$ be defined by 
(\ref{eq:proj-stab}). Then, there is a constant such that,
\begin{equation}\label{eq:err-est}
\boxed{
\| u_N - u_{N,h}\|_{m_h} \lesssim h \|u_0\|_\omega
}
\end{equation}
\end{prop}
\begin{proof}[Proof of Proposition \ref{prop:energy-est}] Using the triangle inequality 
\begin{align}\label{eq:err-est-prf-aa}
\| u_N - u_{N,h} \|_{m_h} 
\leq 
\| u_N - I_h \hatu_N \|_{m_h} + \| I_h \hatu_N - u_{N,h} \|_{m_h}
\end{align}
Here the first term can be directly estimated using (\ref{eq:Ih-error}), 
\begin{equation}\label{eq:err-est-prf-bb}
\| u_N - I_h \hatu_N \|_{m_h} =  \| (I -  I_h) \hatu_N \|_{m_h} \lesssim h \| \hatu_N \|_{\IR^N}  \lesssim h \| u_0 \|_\omega
\end{equation}
since for $v \in \mcU_N$ we have $v = I \hatv$ and using the stability estimate (\ref{eq:stab-exact}) followed by 
(\ref{eq:projexactmodes}) we get 
\begin{align}
\| \hatu_N \|_{\IR^N} \lesssim \| u_N \|_\omega \lesssim \| u_0 \|_\omega
\end{align}
For the second term,  we first note that the stabilization terms 
$s_h$ and $s_{h,\partial \Omega}$ vanish on $\mcU_N$ so that  
\begin{align} \label{eq:err-est-prf-aaa}
m_h(u_{N},v) = (u_N,v)_\omega \qquad \forall v \in \mcU_{N,h}
\end{align}
Then by subtracting and adding $u_N$ in the first argument to $m_h$, we have for any $v\in \mcU_{N,h}$,
\begin{align}
&m_h( I_h \hatu_N - u_{N,h}, v )
\\ \label{eq:err-est-prf-c}
&= m_h( I_h \hatu_N - u_N, v ) + m_h( u_N, v ) - m_h( u_{N,h}, v )
\\  \label{eq:err-est-prf-a}
&= m_h( I_h \hatu_N - u_N, v ) + ( u_N,v)_{\omega} - (u_0, v )_{\omega}
\\  \label{eq:err-est-prf-b}
&= m_h( I_h \hatu_N - u_N, v ) + ( u_N,v - I \hatv )_{\omega} - (u_0, v - I \hatv )_{\omega}
\end{align}
where we used (\ref{eq:err-est-prf-aaa}) and (\ref{eq:proj-stab}) on the second and third terms in (\ref{eq:err-est-prf-c}), respectivley, and the definition (\ref{eq:projexactmodes}) of $u_N$ to subtract $I\hatv \in \mcU_N$ in (\ref{eq:err-est-prf-a}). Employing continuity of the involved forms, we get
\begin{align}
&m_h( I_h \hatu_N - u_{N,h}, v )
\\
&\leq \|I_h \hatu_N - u_N\|_{m_h} \|v\|_{m_h} + \| u_N\|_{\omega} \|v - I \hatv \|_{\omega} + \|u_0\|_\omega \|v - I \hatv\|_{\omega}
\\
&\lesssim h \|\hatu_N\|_{\IR^N} \|v\|_{m_h} + \| u_N\|_{\omega} h^2 \| \hatv \|_{\IR^N} + \|u_0\|_\omega h^2  \| \hatv \|_{\IR^N}
\\
&\lesssim h \underbrace{( \|\hatu_N\|_{\IR^N}  + h \| u_N\|_{\omega} + h \|u_0\|_\omega )}_{\lesssim   \|u_0\|_\omega}  \| v \|_{m_h}
\end{align}
where we used the stability (\ref{eq:stab-mh}) and the bounds $ \|\hatu_N\|_{\IR^N}  \lesssim \| u_N \|_\omega$ and  $\| u_N \|_{\omega} \lesssim \| u_0 \|_\omega$.
Thus by taking $v = I_h u_N - u_{N,h}$, we conclude that
\begin{align}
\| I_h u_N - u_{N,h} \|_{m_h} \lesssim h  \| u_0\|_{\omega} 
\end{align}
which combined with (\ref{eq:err-est-prf-aa}) and (\ref{eq:err-est-prf-bb}) concludes the proof.
\end{proof}

We finally prove the following global result,
\begin{prop}\label{prop:global-est} Let $u_N\in \mcU_N$ be defined by (\ref{eq:projexactmodes}) and $u_{N,h} \in \mcU_{N,h}$ be defined by 
(\ref{eq:proj-stab}). Then, there is a constant depending on higher order Sobolev spaces of $g$ such that,
\begin{equation}\label{eq:err-est-glob-H1}
\boxed{
\| u_N - u_{N,h}\|_{H^1(\Omega)} \lesssim h \| u_0 \|_\omega
}
\end{equation}
\end{prop}
\begin{proof} With $e = u_N - u_{N,h}$, we have
\begin{equation}
\|e\|_{H^1(\Omega) }\lesssim \|e - I \hate\|_{H^1(\Omega)} + \|I \hate\|_{H^1(\Omega)}
\end{equation}
By norm equivalence on discrete spaces we have
\begin{equation}
\| I \hate\|_{H^1(\Omega)} \lesssim \|\hate\|_{\mathbb{R}^N} 
\end{equation}
Since $I \hate \in \mathcal{U}_{N}$ there holds using \eqref{eq:stab-exact},
\begin{equation}
\|I \hate\|_{H^1(\Omega)} \lesssim \|I \hate\|_\omega \leq  \|e\|_\omega + \|u_{N,h} - I \hatu_{N,h}\|_\omega
\end{equation}
By Proposition \ref{prop:energy-est} there holds
\begin{equation}
\|e\|_\omega \lesssim h \|u_0\|_\omega
\end{equation}
For the second term we have using \eqref{eq:Ih-error}, 
\begin{equation}
\|u_{N,h} - I \hatu_{N,h}\|_\omega \leq \|(I_h- I) \hatu_{N,h}\|_\Omega \lesssim h^2 \|\hatu_{N,h}\|_{\mathbb{R}^N} \lesssim  h^2 \|u_{N,h}\|_{m_h}
\end{equation}
Similarly we have 
\begin{equation}
 \|e - I \hate \|_{H^1(\Omega)} = \|u_{N,h} - I \hatu_{N,h}\|_{H^1(\Omega)} \lesssim h \|\hatu_{N,h}\|_{\mathbb{R}^N} \lesssim  h \|u_{N,h}\|_{m_h}
\end{equation}
We conclude the proof by using the bound
\begin{equation}
\|u_{N,h}\|_{m_h} \lesssim \|u_0\|_\omega
\end{equation}
\end{proof}

\begin{rem} Observe that the stabilization is never explicitly used in order to obtain error estimates. Indeed its only role is to ensure the bound $\|\hatu_{N,h}\|_{\mathbb{R}^N} \lesssim \|u_{N,h}\|_{m_h}$ without condition on the mesh. %
\end{rem}

\paragraph{Example (Exponential growth of the stability constant).} Let $\Omega$ be the unit disc. Then the solutions to $-\Delta u = 0$ are of the form 
\begin{align} \label{eq:ex-analytical-expansion}
u(r,\theta) =  \sum_{n=0}^\infty a_{2n} \underbrace{r^n \cos (n \theta)}_{=\varphi_{2n}} +  a_{2n+1} \underbrace{r^n \sin (n \theta)}_{=\varphi_{2n+1}}
\end{align}
where $(r,\theta)$ are the standard polar coordinates.  Let $\omega$ be the disc centered at the origin with radius 
$r_\omega$.  We note that when $n$ becomes large, the modes become small in the disc $\omega$, and therefore,  
the inverse problem becomes increasingly ill-posed, see Figure~\ref{fig:disc-ex-a}. For instance, the constant in an estimate of the type 
\begin{align} \label{eq:ex-analytical-constant}
\| \varphi_{2n+m} \|_\Omega \leq C_n \| \varphi_{2n+m}\|_\omega, \qquad m=0,1 
\end{align}
scales like 
\begin{equation}
C_n = r_{\omega}^{-(n+1)}
\end{equation}
and thus becomes arbitrarily large when $n$ becomes large.  But, if we, from observations, can conclude that only modes with $n < n_g$ for some $n_g$ are present, then the stability is controlled, see Figure~\ref{fig:disc-ex-b}. Note also that the stability is directly related to where the disc $\omega$ is placed.  If it is located close to the boundary, the stability improves.

\begin{figure}\centering
  \begin{subfigure}{0.59\linewidth}\centering
  \includegraphics[width=0.31\linewidth]{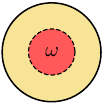}
  \includegraphics[width=0.31\linewidth]{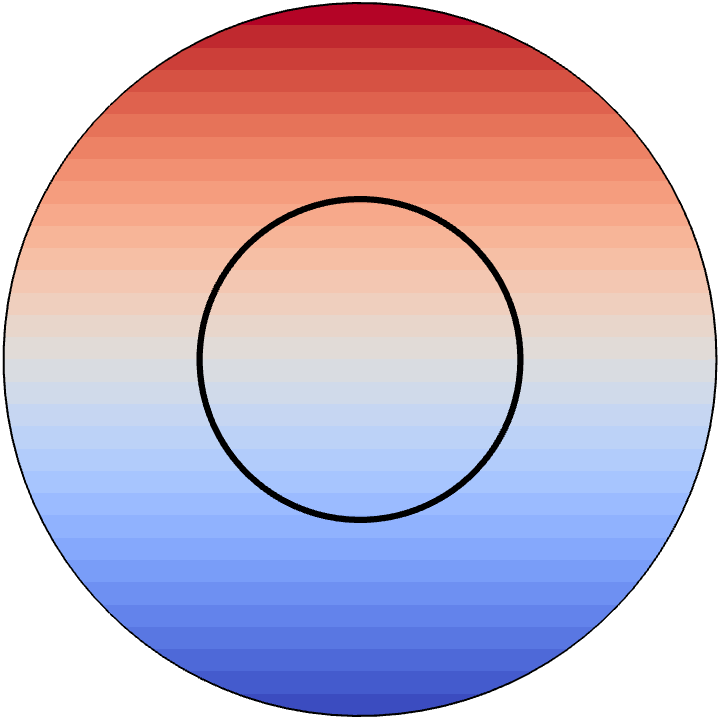}
  \includegraphics[width=0.31\linewidth]{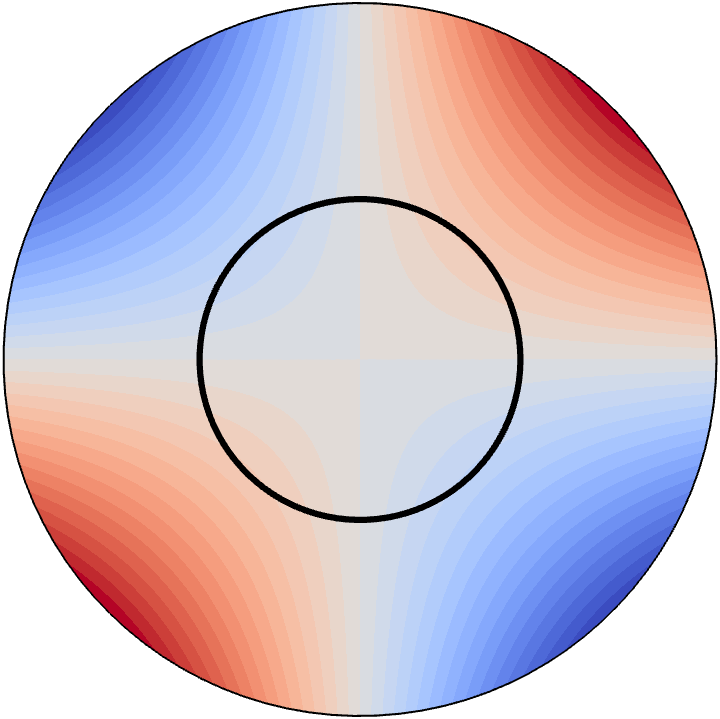}
  
  \includegraphics[width=0.31\linewidth]{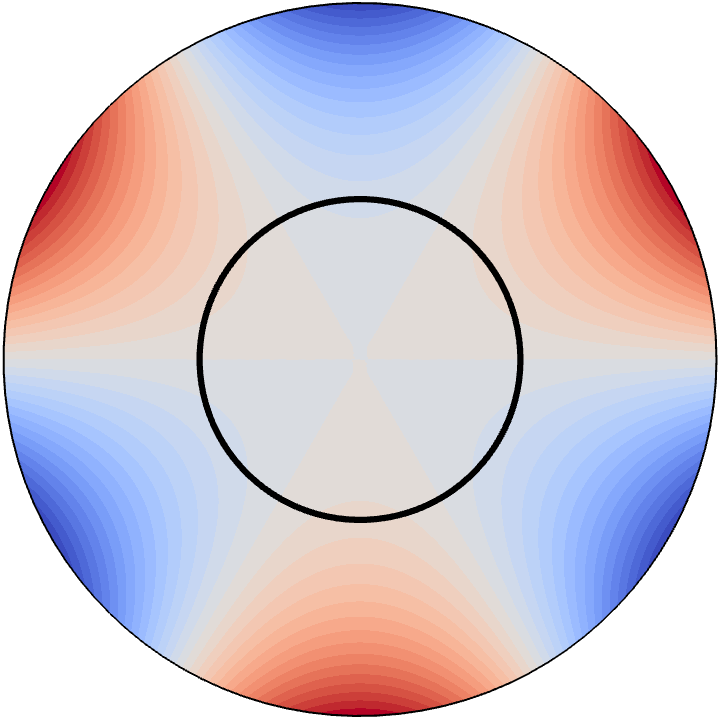}
  \includegraphics[width=0.31\linewidth]{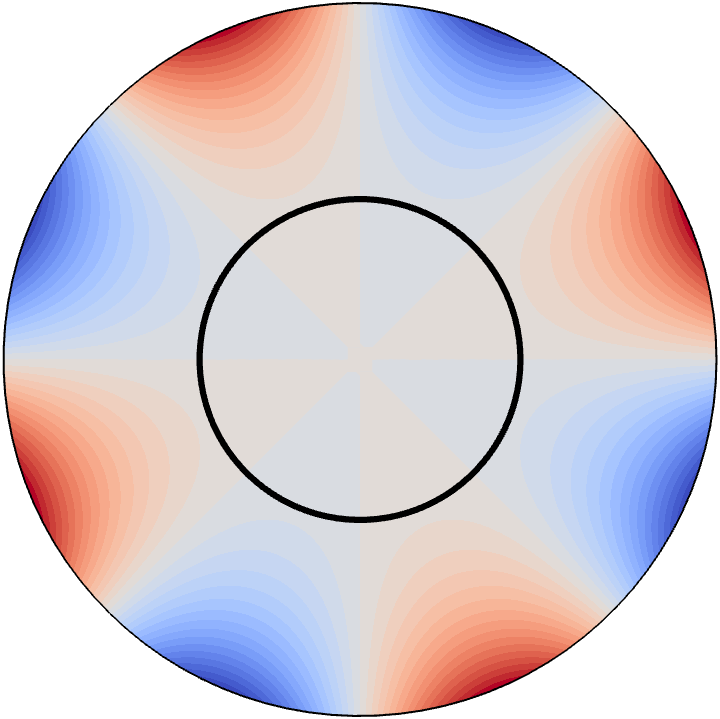}
  \includegraphics[width=0.31\linewidth]{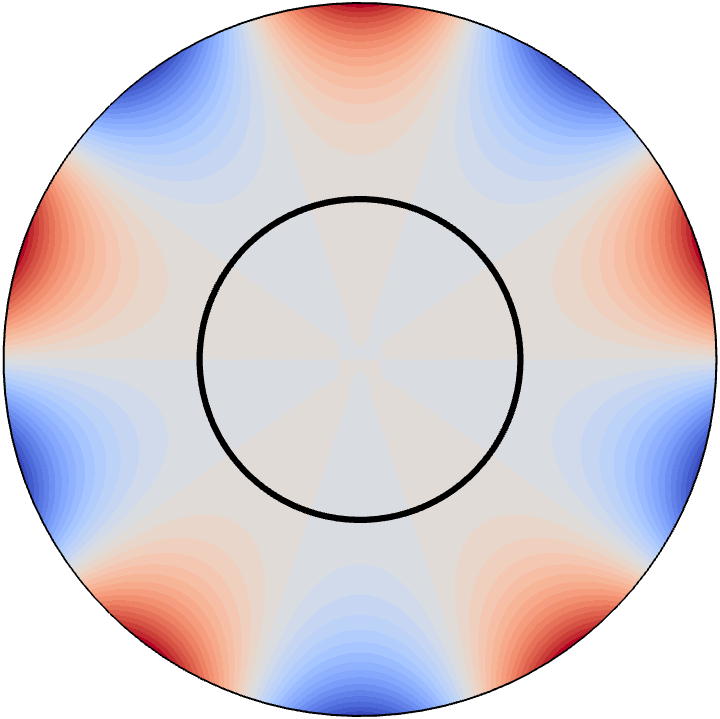}
  \subcaption{Domain and first five modes of $\varphi_{2n+1}=r^n\sin(n\theta)$}
  \label{fig:disc-ex-a}
\end{subfigure}
\begin{subfigure}{0.40\linewidth}\centering
  \includegraphics[width=0.9\linewidth]{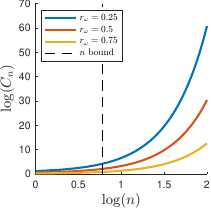}
  \subcaption{Scaling of stability constants}
  \label{fig:disc-ex-b}
\end{subfigure}
\caption{Illustrations for the analytical example with exponential growth of the stability constant. In (a), we show the unit disc domain containing a subdomain $\omega$, in the form of a centered disc of radius $r_\omega$. Looking at the first five non-zero modes $\varphi_{2n+1}=r^n\sin(n\theta)$ from the expansion \eqref{eq:ex-analytical-expansion} we see that these modes rapidly becomes very small within $\omega$, making the problem of retrieving the coefficient values in the expansion based on the solution within $\omega$ increasingly ill-posed for larger $n$. In (b), we illustrate how the constants in the stability estimate \eqref{eq:ex-analytical-constant} scales exponentially with $n$ for different values of the radius $r_\omega$. By utilizing observations, we may conclude an upper bound on $n$, which in turn puts an upper bound on the size of the stability constant.}
\label{fig:disc-ex}
\end{figure}

\section{Methods Based on Machine Learning}

\paragraph{Overview.} We develop a method for efficiently solving the inverse problem (\ref{eq:inverse}) with access to sampled data $\mcG_S$ using machine learning techniques. The main approach is:
\begin{itemize}
\item Construct a parametrization of the data set by first approximately expanding the samples in a finite series of functions, for instance, using Proper Orthogonal Decomposition, and secondly using an autoencoder to find a possible nonlinear low-dimensional structure in the expansion coefficients. 
\item Use operator learning to construct an approximation of the finite element solution operator that maps the expansion coefficients to the finite element solution.
\item Composing the decoder, which maps the latent space to expansion coefficients, with the solution network, we obtain a differentiable mapping that can be used to solve the inverse problem in a lower-dimensional space.
\end{itemize}

\subsection{Processing the Boundary Data}

We combine linear and nonlinear dimensionality reduction techniques by first using PCA on the data to get a POD-basis and then using autoencoders on the POD-coefficients for further reduction. Such combinations are not uncommon, see, e.g., \cite{fulton_latentspace_2019}, and there are several reasons why we do this: In general, an initial linear reduction may function as a relatively cheap preprocessing step to aid a subsequent nonlinear reduction that typically is more expensive. More specifically, we consider methodology for progressing from a fully linear problem (linear PDE, linear data) to a fully nonlinear one (nonlinear PDE, nonlinear data), hence both linear and nonlinear techniques that can be combined are required. The reason for using POD is that it very cheaply and naturally gives a basis that can be used for both the linear and nonlinear PDE-solving techniques considered here. The reason for using autoencoders is simply because they have the same general network architecture already used for the operator networks which also makes combinations with them seem natural.

\paragraph{Proper Orthogonal Decomposition.} To assimilate the data set $\mcG$ in a method for solving the extension problem, we seek to construct a differentiable parametrization of~$\mcG$. To that end, we first use Proper  Orthogonal Decomposition (POD) to represent the data in a POD 
basis $\{\varphi_n\}_{n=1}^N$, 
\begin{align}
g = \sum_{n=1}^N \hatg_n \varphi_n
\end{align}
where $\hatg_n = (g,\varphi_n)_{\IR^N}$. We introduce the mapping 
\begin{align}
\phi_{\mathrm{POD},N}:\mcG \ni g \mapsto \hatg \in G_N \subset \IR^{N} 
\end{align}
where $G_N = \phi_{\mathrm{POD},N}(\mcG)$. We also need the reconstruction operator 
\begin{align}
\phi_{\mathrm{POD},N}^\dagger: \IR^{N} \ni a \mapsto \sum_{n=1}^N a_n \varphi_n \in \mcG
\end{align}
Letting $I_{N}$ denote the identity operator on $\IR^{N}$, we have
\begin{equation}
\phi_{\mathrm{POD},N} \circ \phi_{\mathrm{POD},N}^\dagger = I_{N}
\end{equation}
and we note that the operator $\phi_{\mathrm{POD},N} $ is invertible 
and differentiable.

\paragraph{Autoencoder.} Next, we seek to find a possible nonlinear low-dimensional structure in the POD coefficients using an autoencoder $\phi_{\text{de}} \circ \phi_{\text{en}}$
\begin{align}
\boxed{ G_N \overset{\phi_{\text{en}}}{\longrightarrow} Z  \overset{\phi_{\text{de}}}{\longrightarrow} G_N}
\end{align}
where $\phi_{\text{en}}$ denotes the encoder map and $\phi_{\text{de}}$ the decoder map. Letting $\mathbb{E}$ denote the expectation operator and $P$ an arbitrary probability distribution, the autoencoder is trained to minimize the loss
\begin{align}
\mathbb{E}_{\hatg \sim P} \left[ \| \hatg - (\phi_{\text{de}} \circ \phi_{\text{en}})(\hatg) \|^2_{\IR^{N}} \right] 
\end{align}
See Figure~\ref{fig:networks-ae} for a schematic illustration.
Here $Z \sim \IR^{n_Z}$ is the latent space with dimension $n_Z < N$. If there is a low-dimensional structure, we may often take $n_Z$ significantly lower than $N$.

\subsection{Operator Learning}

The operator learning approach taken here is the same as in \cite{larson2024nonlinoplearn} which is a special case of a more general method presented in \cite{sharp_data-free_2023}. We discretize the PDE problem using finite elements and train a network 
\begin{align}
\phi_{u,N,h} : G_N \rightarrow V_h \subset H^1(\Omega)
\end{align}
which approximates the finite element solution to 
\begin{align}
\mcP(u) = 0 \quad \text{in $\Omega$}, \qquad u = \phi_{\mathrm{POD},N}^\dagger (\hatg) \quad \text{on $\partial \Omega$}
\end{align}
see Figure~\ref{fig:networks-operator}. The output of the network is the finite element degrees of freedom (DoFs). For the training of the network we use the energy functional $E$ corresponding to the differential operator $\mcP$ as the foundation for the loss function. Again, letting $\mathbb{E}$ denote the expectation operator and $P$ an arbitrary probability distribution, the loss function that we minimize during training is
\begin{align}
\mathbb{E}_{\hatg \sim P} \left[ E(\phi_{u,N,h}(\hatg)) \right]
\label{eq_abstract_loss}
\end{align}
If there is no corresponding energy functional, one can instead minimize the residual of the finite element problem. It should be noted though, that assembling the residual instead of the energy has a greater computational cost and that the residual is not as easily and naturally decomposed into its local contributions as the energy. For technical details about network architecture and training used in this work, we refer to Section~\ref{sec_exs_nold}.

\subsection{Inverse Problem}

Finally, composing the maps, we get a solution operator 
\begin{align}
\boxed{Z \overset{\phi_{\text{de}}}{\longrightarrow} G_N 
 \overset{\phi_{u,N,h}}{\longrightarrow}  V_h}
\end{align}
that maps the latent space into approximate finite element solutions to the partial differential equation
\begin{align}
\mcP((\phi_{u,N,h} \circ \phi_g)(z)) = 0
\end{align}
see Figure~\ref{fig:networks-operator-ae}.

This mapping is differentiable and can be directly used to 
rewrite the optimization problem as an unconstrained problem 
in the form
\begin{align}
\boxed{ \inf_{z \in Z} \frac12 \| u_0 - (\phi_{u,N,h} \circ \phi_{\text{de}})(z) \|_{L^2(\omega)}^2 }
\label{eq_invprob_Z}
\end{align}
where we note that the constraint is fulfilled by construction.

\begin{figure}\centering
\begin{subfigure}[t]{\linewidth}\centering
\includegraphics[width=0.8\linewidth]{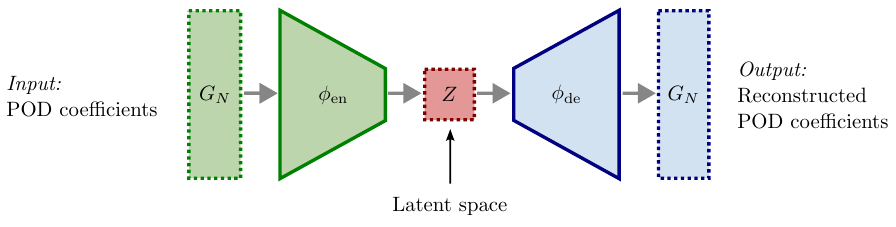}
  \subcaption{Autoencoder network $\phi_{{\text{de}}}  \circ \phi_{{\text{en}}}$}
  \label{fig:networks-ae}
\end{subfigure}

\begin{subfigure}[t]{\linewidth}\centering
  \includegraphics[width=0.7\linewidth]{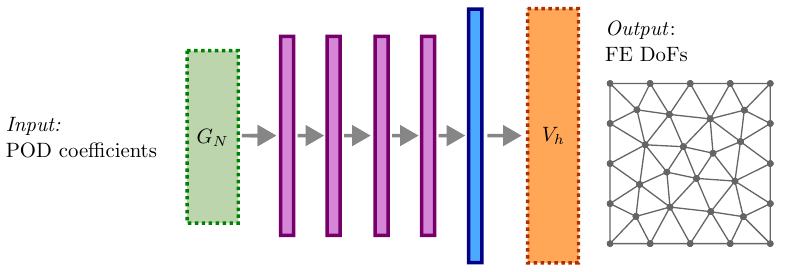}
    \subcaption{Operator network $\phi_{u,N,h}$ with four nonlinear layers (purple) and one linear output layer (blue)}
    \label{fig:networks-operator}
  \end{subfigure}

  \begin{subfigure}[t]{\linewidth}\centering
    \includegraphics[width=0.90\linewidth]{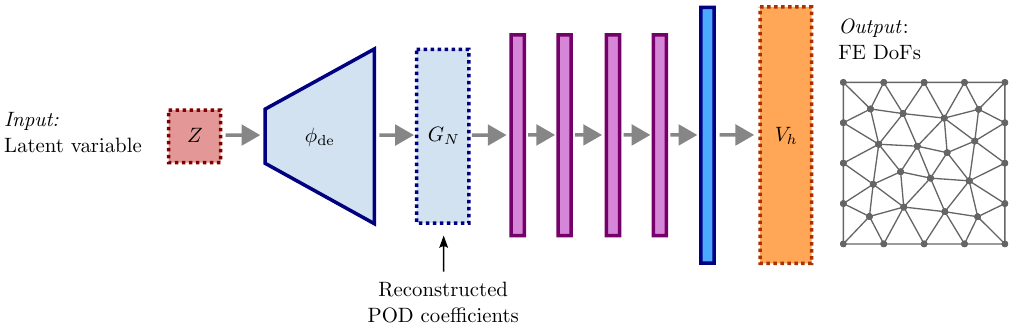}
      \subcaption{Composition of decoder and operator network $\phi_{u,N,h} \circ \phi_{{\text{de}}}$}
      \label{fig:networks-operator-ae}
    \end{subfigure}
\caption{Overview of networks utilized in methods based on machine learning.
The autoencoder network in (a) is used for identifying a low-dimensional structure in the dataset $\mcG$. The operator network in (b) is
trained to approximate the solution to the PDE, given input boundary data.
The composition of the decoder part of the autoencoder and the operator network in (c) is used for solving the inverse problem in the low-dimensional latent space.}
\label{fig:networks}
\end{figure}

\section{Examples}\label{sec_examples}

We consider three examples of the inverse minimization problem ordered in increased nonlinearity. The first is a fully linear case with a linear differential operator and linear boundary data. In the second example, we consider a nonlinear operator with linear data. The final example is a fully nonlinear case with both operator and data being nonlinear. The examples demonstrate how each introduced nonlinearity may be treated with machine learning methods.

The geometry is the same in all the examples. We take the solution domain $\Omega := (-0.5, 0.5)^2 \subset \IR^2$ and the reference domain $\omega \subset \Omega$ to be the u-shaped domain defined by
\begin{equation}
\omega := \{ (x, y) \in \IR^2 \,|\, x < 0.25 \land (x < -0.25 \lor y < -0.25 \lor y > 0.25) \}
\end{equation}
see Figure~\ref{fig_domainwithomega}.

\begin{figure}
\centering
\includegraphics[width=0.2\linewidth]{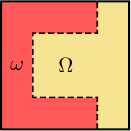}
\caption{Domain used in all numerical examples with the subdomain $\omega$ indicated.}
\label{fig_domainwithomega}
\end{figure}

When solving the inverse problems in practice, we use data $u_0 \in V_h$. We also minimize the mean squared error (MSE) over the DoFs belonging to $\omega$ instead of the squared $L^2(\omega)$ norm of the error, which is valid since they are equivalent on $V_h$ from the Rayleigh quotient. The only ``stabilization'' we use for the inverse problem is that the boundary data is finite-dimensional and that this dimension together with the mesh size $h$ both are small enough. We point out that no additional stabilization, such as including penalty terms, is used. The criterion for when a minimization process is considered to have converged is based on the change of significant digits of the MSE. For the fully linear problem we consider it converged when at least three significant digits remain constant, and for all the nonlinear problems when at least two significant digits remain constant. This is in turn based on when both the optimization variables and the visual representation of the output do not seem to change anymore and has been obtained by testing.

The implementation used for the examples is based on the code presented in \cite{sharp_data-free_2023} which is publicly available at \url{https://github.com/nmwsharp/neural-physics-subspaces}. All inverse problem minimizations have been performed with the Adam optimizer with a step size = 0.1 on an Apple M1 CPU. The GPU computations were performed on the Alvis cluster provided by NAISS (See Acknowledgements).

\subsection{Linear Operator with Linear Data}

We start with the fully linear case which we will build upon in the later examples. To construct a linear synthetic data set $\mcG$, we may pick a 
set of functions $\{\varphi_j\}_{j\in J} \subset H^{1/2}(\partial \Omega)$ where $J$ is some index set, and consider 
\begin{align}
\mcG = \Bigl\{g_i = \sum_{j \in J } \xi_j \varphi_j \Bigr\}
\end{align}
where $\xi\in [s_i,t_i] \subset \IR$. Note that we require the boundary data to be bounded.  Alternatively,  we can also consider taking the 
convex hull of the basis functions $\{\varphi_j\}_{j\in J}$, which corresponds to requiring that 
\begin{align}
\sum_j \xi_j = 1, \quad \xi_j \geq 0
\end{align}
Given nodal samples of such functions, we may apply principal component analysis (PCA) to estimate a set of basis functions and use them to 
parametrize the data set. More precisely, assume we observe the boundary 
data in the nodal points at the boundary. Let $X$ be the matrix where each observation forms a row. Then, computing the eigenvectors to the symmetric 
matrix $X^T X$ provides estimates of the basis.

Here, we consider two-dimensional examples. We let $\Omega$ be the unit square centered at the origin and generate four structured uniform triangular meshes of varying sizes: 10x10, 28x28, 82x82, and 244x244. The synthetic data set of boundary nodal values is in turn generated from the perturbed truncated Fourier series 
\begin{align}
g(x) = (\hatg_0 + \delta_0) + \sum_{n=1}^{(N-1)/2} (\hatg_{2n-1} + \delta_{2n-1}) \sin(2n\pi x/l) 
+ (\hatg_{2n} + \delta_{2n}) \cos(2n\pi x/l)
\label{eq_fourierboundaryfunc}
\end{align}
where $l$ is the circumference of $\Omega$ and $x$ is the counter-clockwise distance along the boundary starting from the point where the boundary crosses the first coordinate axis. We sample unperturbed coefficients $\hatg_{j} \sim \mcU(-1, 1)$ and perturbations $\delta_j \sim \mcN(0, 0.0225)$. For each of the four meshes, we consider two values of the number of coefficients used to describe the boundary conditions; $N=9$ and $N=21$. We generate 1000 functions of the type \eqref{eq_fourierboundaryfunc} for each of the eight cases. Then, for every case, we compute a POD basis $\{\varphi_{\mathrm{POD},j}\}_{j \in J}$ for the boundary using PCA on the data set. Unsurprisingly, the number of significant singular values turns out to be the number $N$ used in each case. 

We use the truncated POD boundary basis corresponding to significant singular values to compute and interior basis. We do this by solving Laplace's equation with FEM. We take the discrete space $V_h$ to simply be the space of piecewise linear finite elements on the triangle mesh considered. The FEM interior basis $\{\varphi_{\mathrm{FEM},n}\}_{n = 1}^N$ is computed by: For each $n = 1, ..., N$, find $\varphi_{\mathrm{FEM},n} \in V_h$ such that $\varphi_{\mathrm{FEM},n} |_{\partial \Omega} = \varphi_{\mathrm{POD},n}$ and 
\begin{align}
( \nabla \varphi_{\mathrm{FEM},n}, \nabla v)_{\Omega} = 0 \quad \forall v \in V_h
\end{align}
In Figure~\ref{fig_lindata_podbfs}, the significant POD boundary basis functions, together with their corresponding FEM interior basis functions, are presented for the case with $N=9$ and the 82x82 mesh.
\begin{figure}
\centering
\adjincludegraphics[width=0.3\linewidth,Clip=0 {.23\height} 0 {.23\height}]{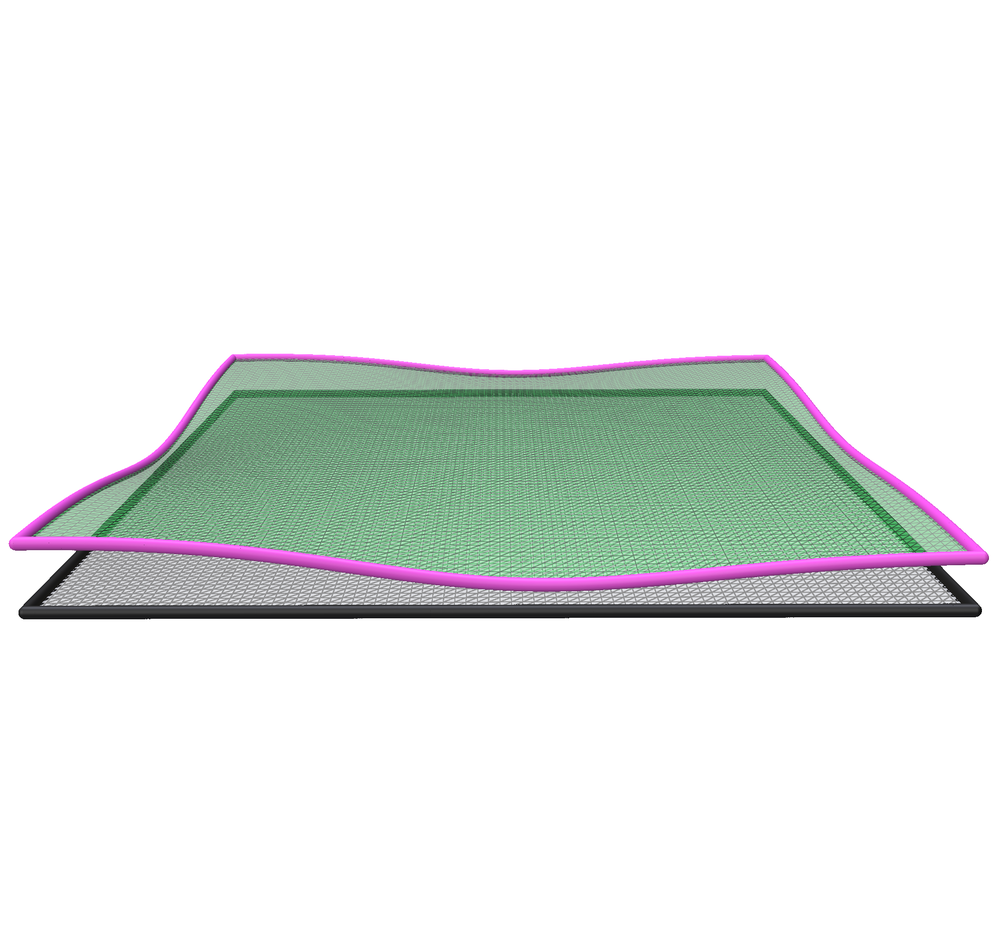}
\adjincludegraphics[width=0.3\linewidth,Clip=0 {.23\height} 0 {.23\height}]{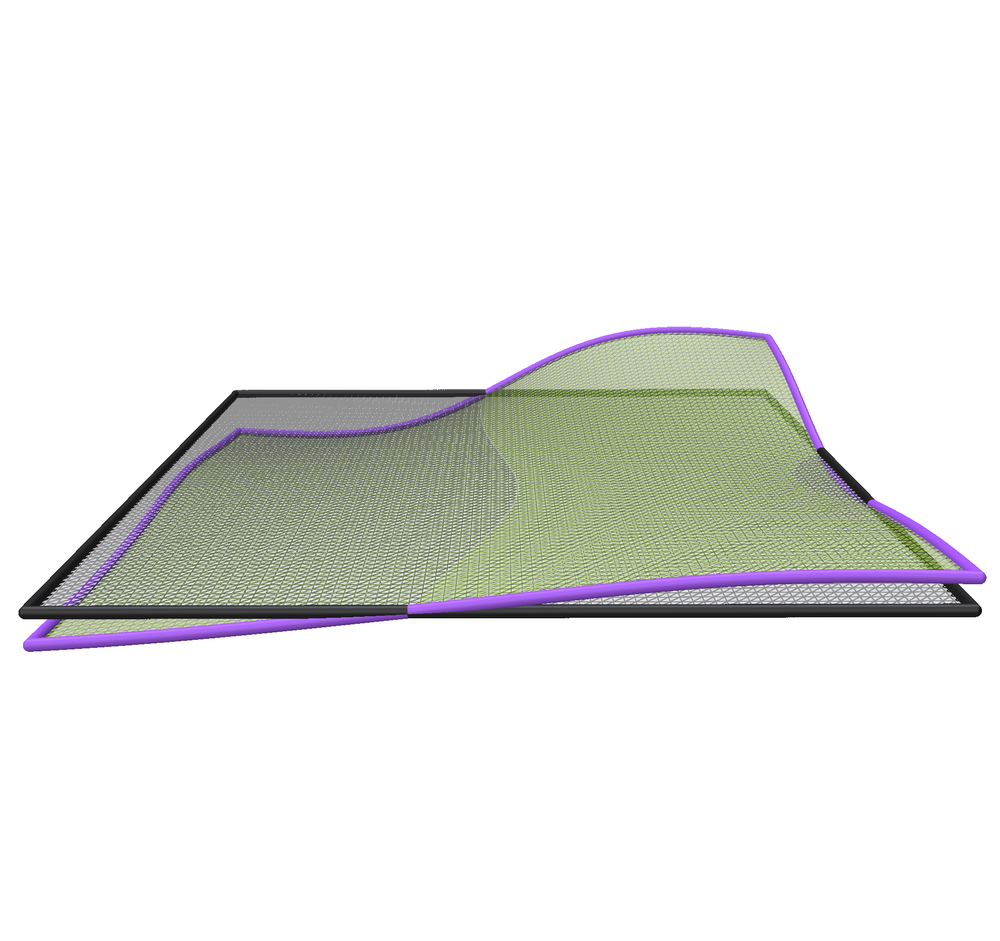}
\adjincludegraphics[width=0.3\linewidth,Clip=0 {.23\height} 0 {.23\height}]{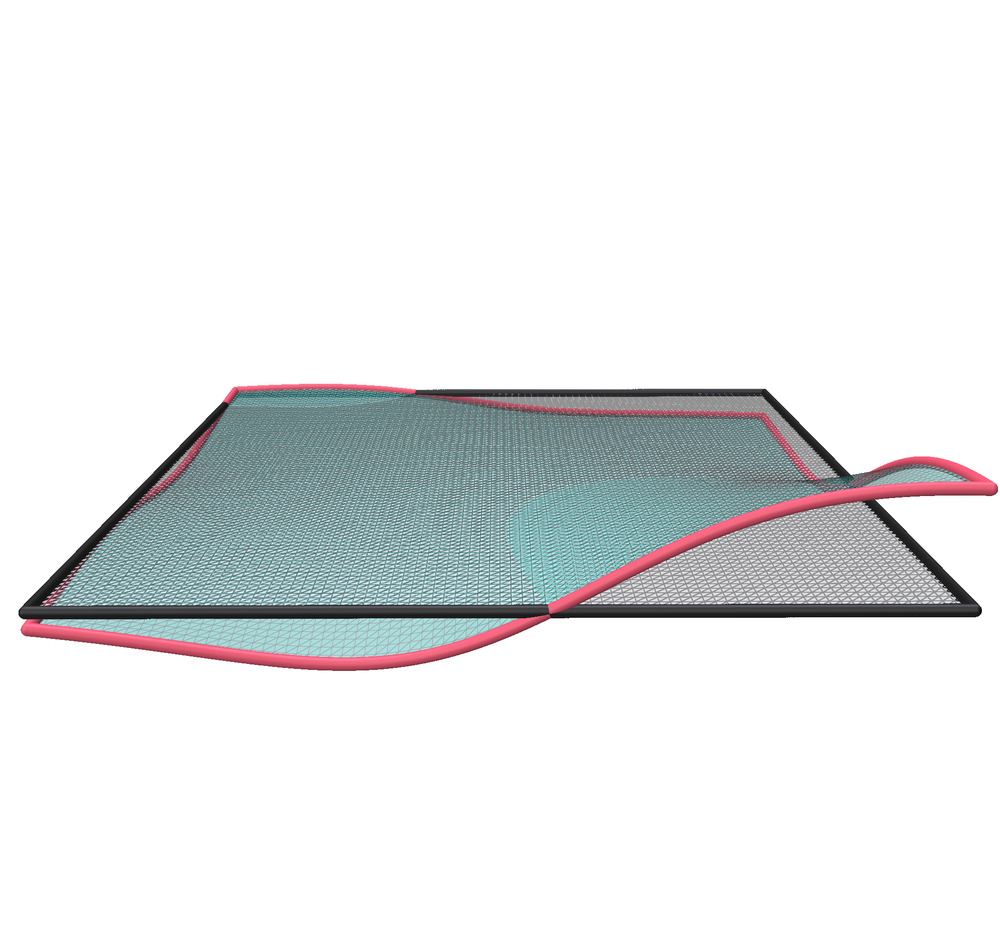}

\adjincludegraphics[width=0.3\linewidth,Clip=0 {.23\height} 0 {.23\height}]{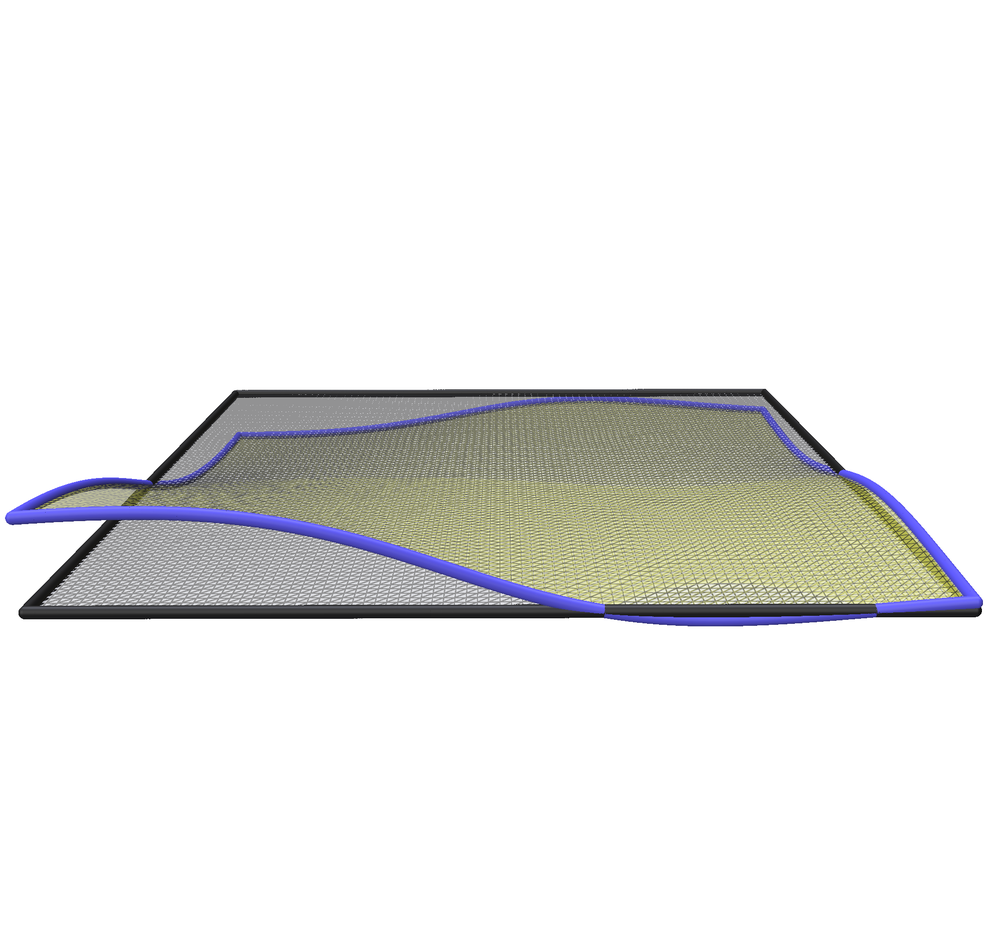}
\adjincludegraphics[width=0.3\linewidth,Clip=0 {.23\height} 0 {.23\height}]{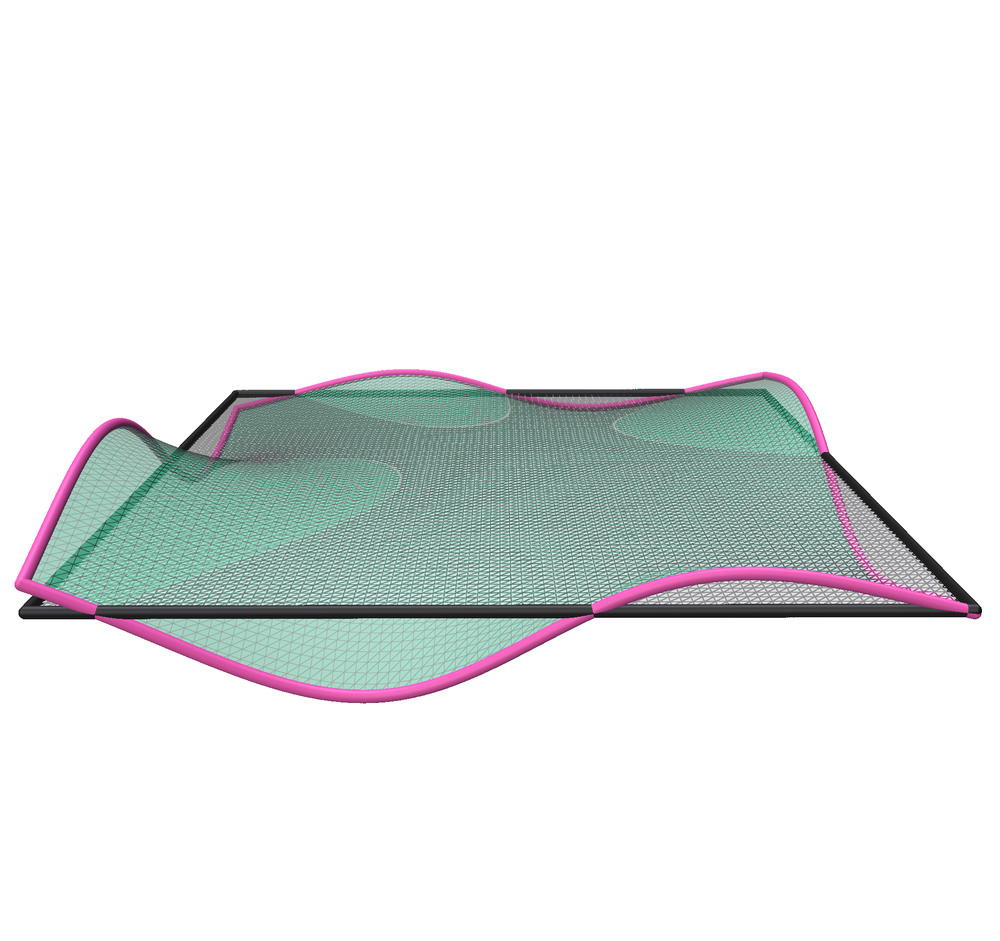}
\adjincludegraphics[width=0.3\linewidth,Clip=0 {.23\height} 0 {.23\height}]{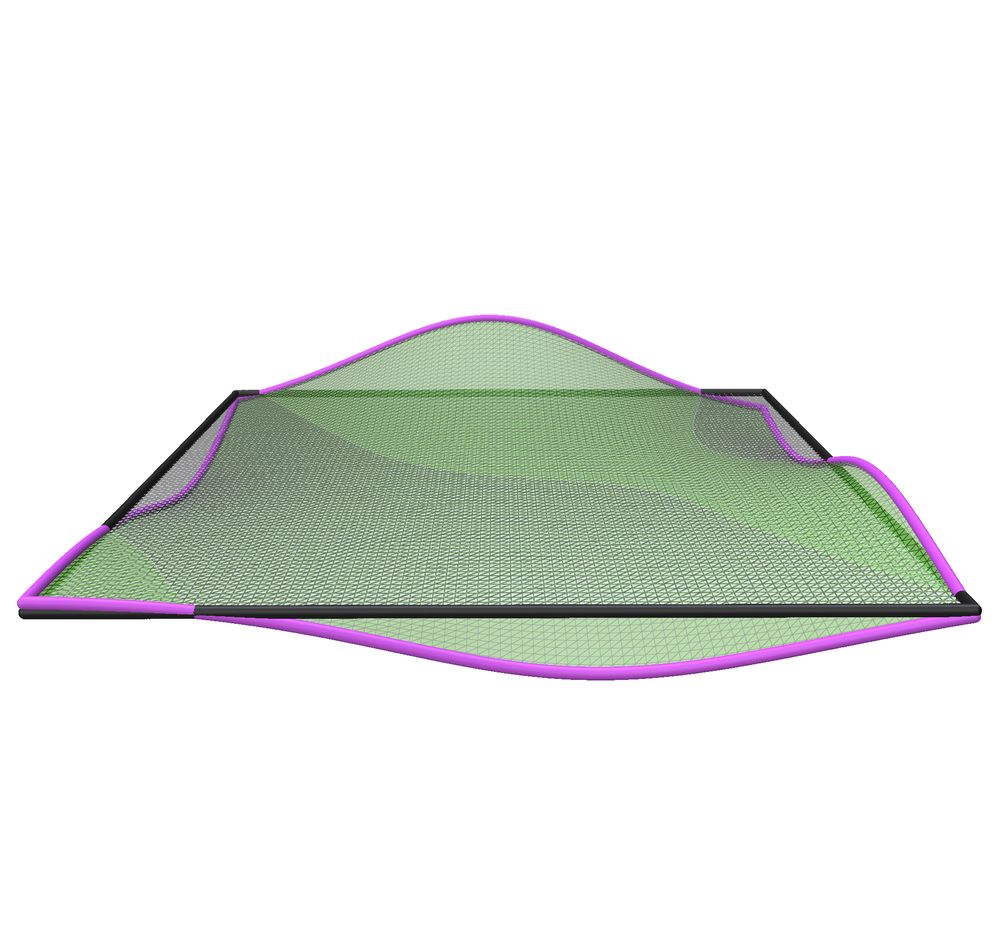}

\adjincludegraphics[width=0.3\linewidth,Clip=0 {.23\height} 0 {.23\height}]{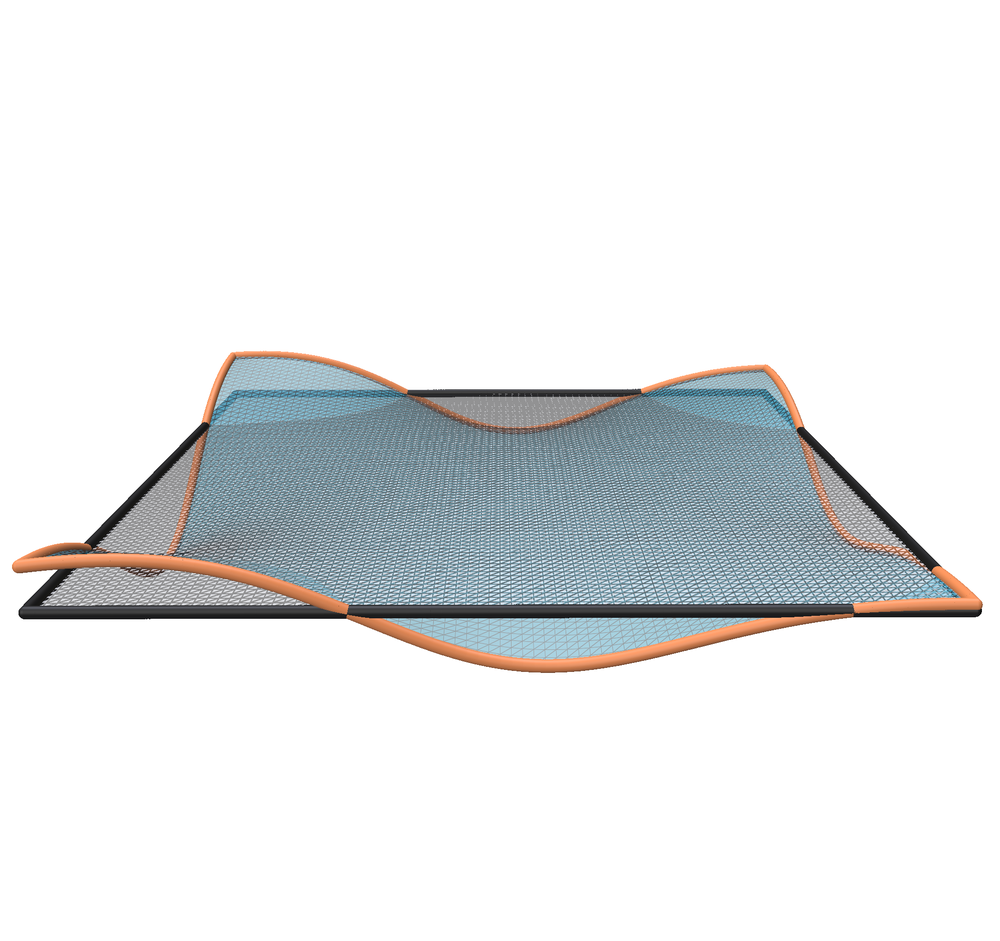}
\adjincludegraphics[width=0.3\linewidth,Clip=0 {.23\height} 0 {.23\height}]{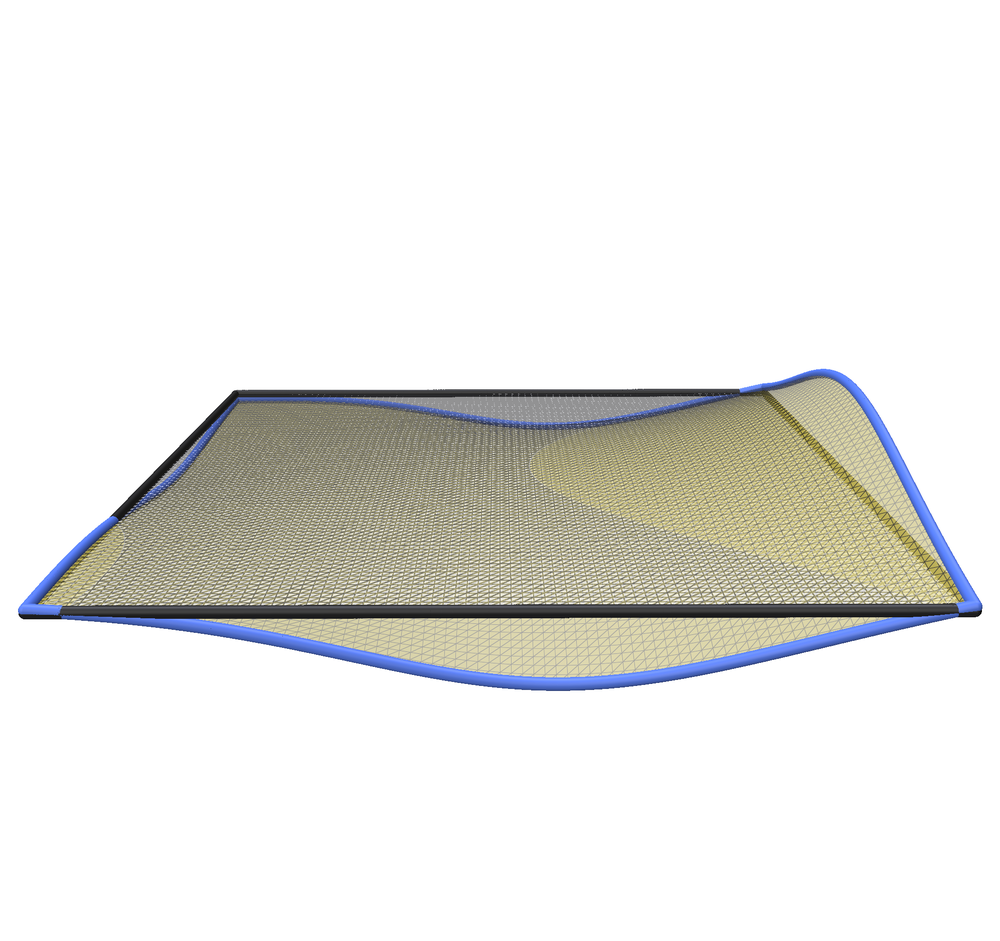}
\adjincludegraphics[width=0.3\linewidth,Clip=0 {.23\height} 0 {.23\height}]{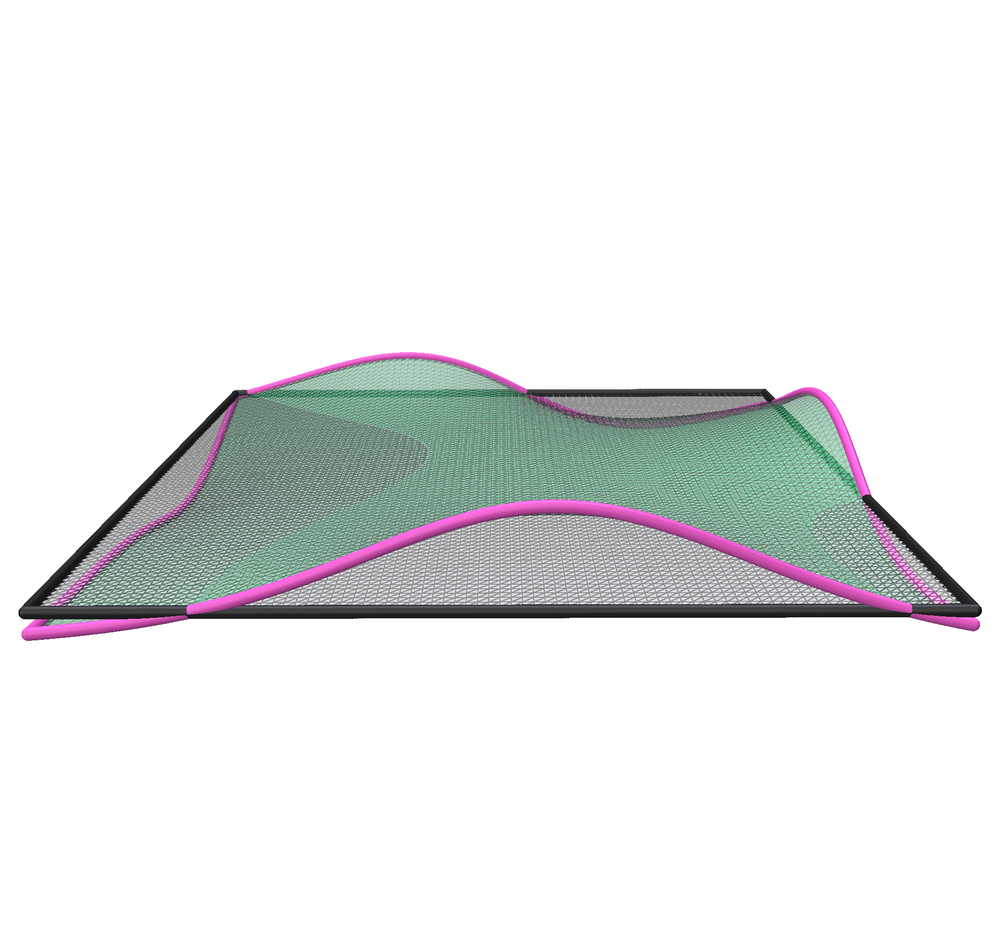}
\caption{FEM interior basis functions with their corresponding POD boundary basis functions on a structured uniform 82x82 triangular mesh of the unit square.}
\label{fig_lindata_podbfs}
\end{figure}

We may now use the fact that Laplace's equation is linear to superpose the FEM interior basis functions in a linear combination.
\begin{align}
u_{h, \text{lin}, N} = \sum_{n = 1}^N c_n \varphi_{\mathrm{FEM},n}
\label{eq_linlin_lc}
\end{align}
This enables us to solve a linear inverse minimization problem over the coefficients $(c_1, ..., c_N) \in \IR^{N}$ in the linear combination. We present a demonstration of this process for the case with $N = 9$ and the 82x82 mesh in Figure~\ref{fig_linlin_optprocess}. There it can clearly be observed that the finite element solution given by the linear combination approaches the noisy data and the reference solution as the optimization progresses.
\begin{figure}
\centering
\includegraphics[width=0.3\linewidth]{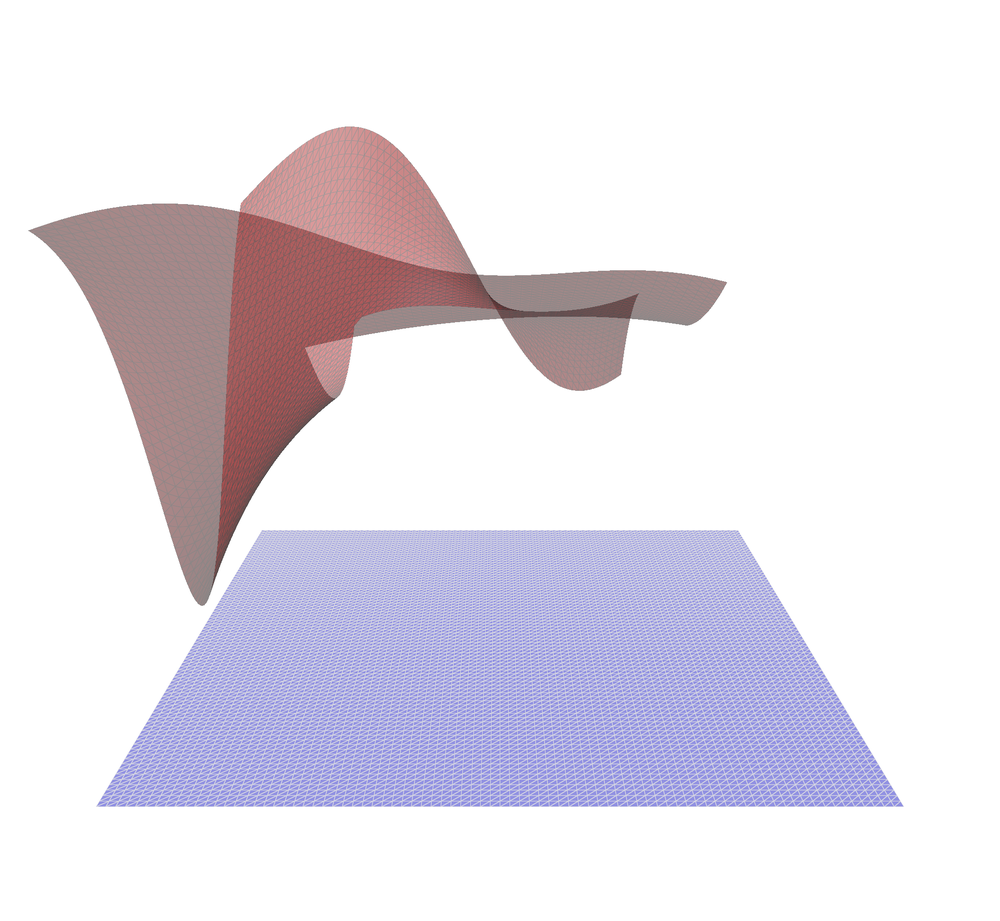}
\includegraphics[width=0.3\linewidth]{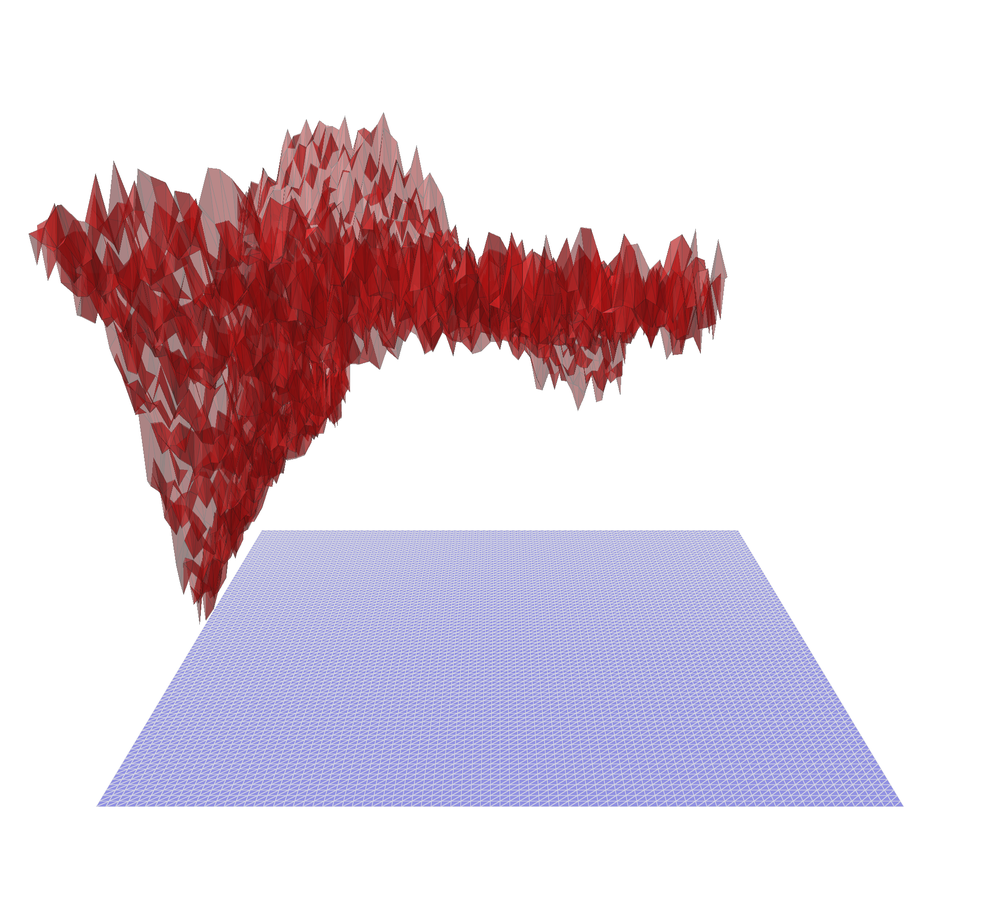}
\includegraphics[width=0.3\linewidth]{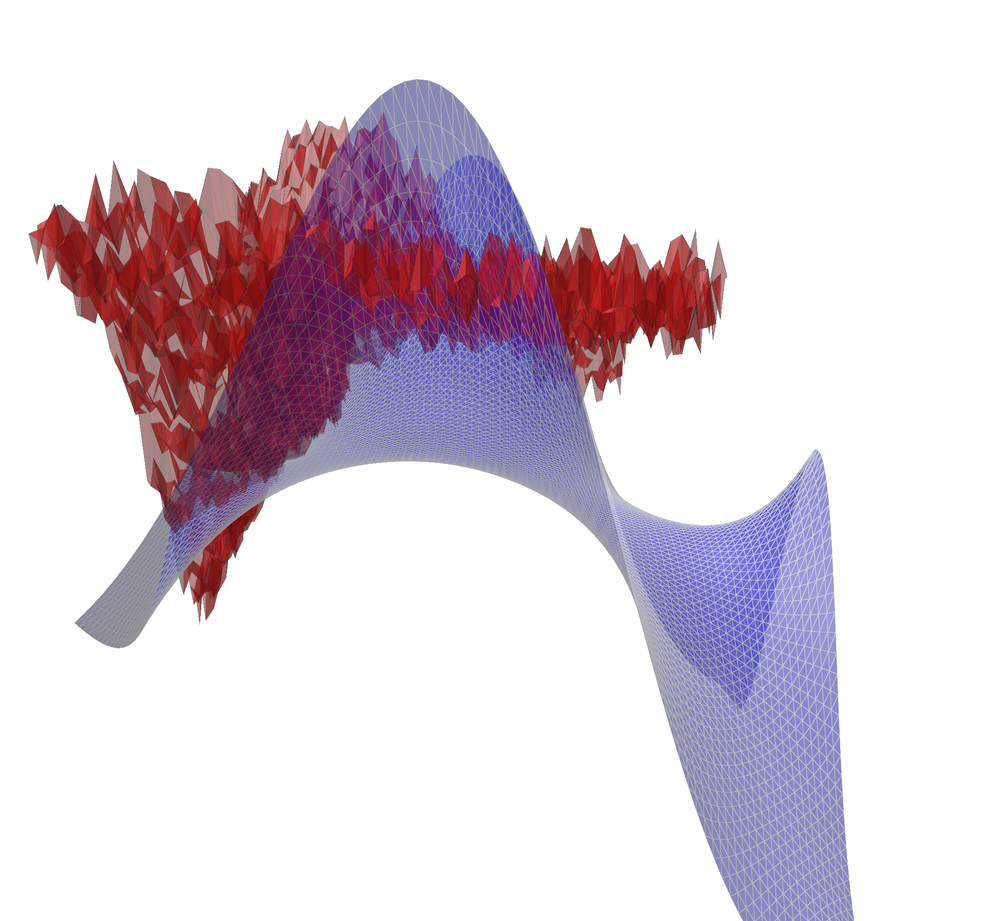}
\includegraphics[width=0.3\linewidth]{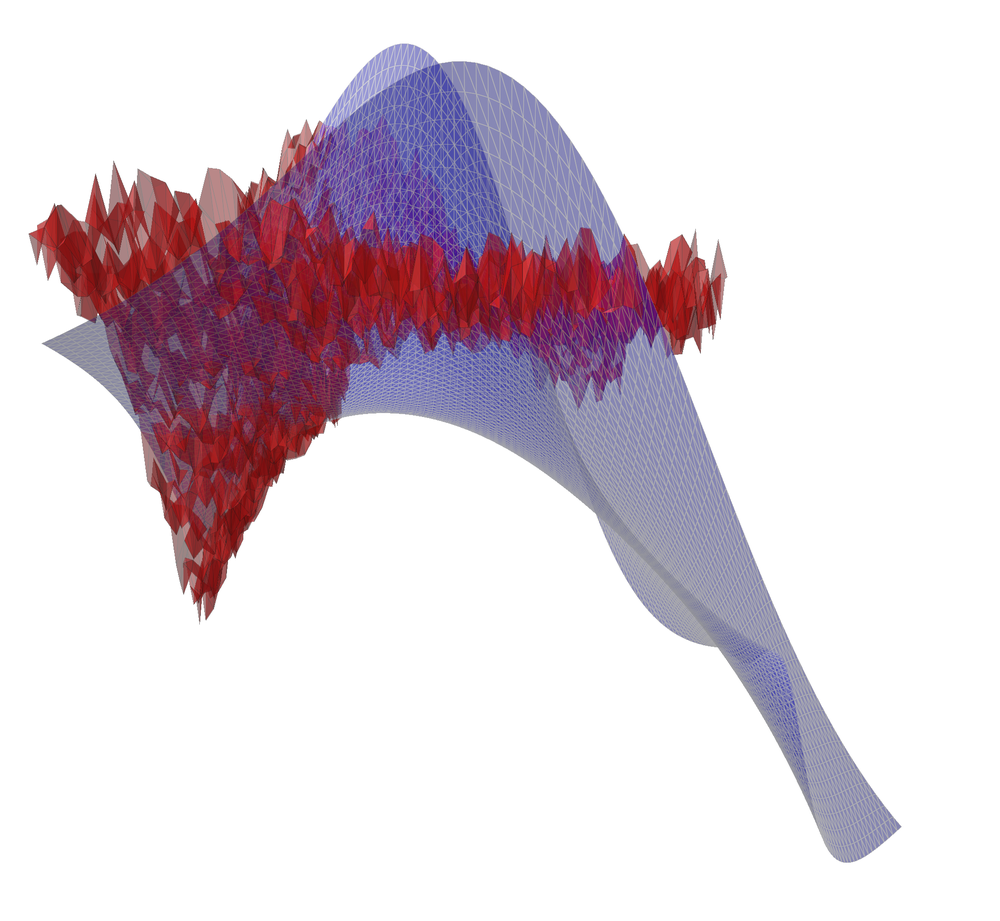}
\includegraphics[width=0.3\linewidth]{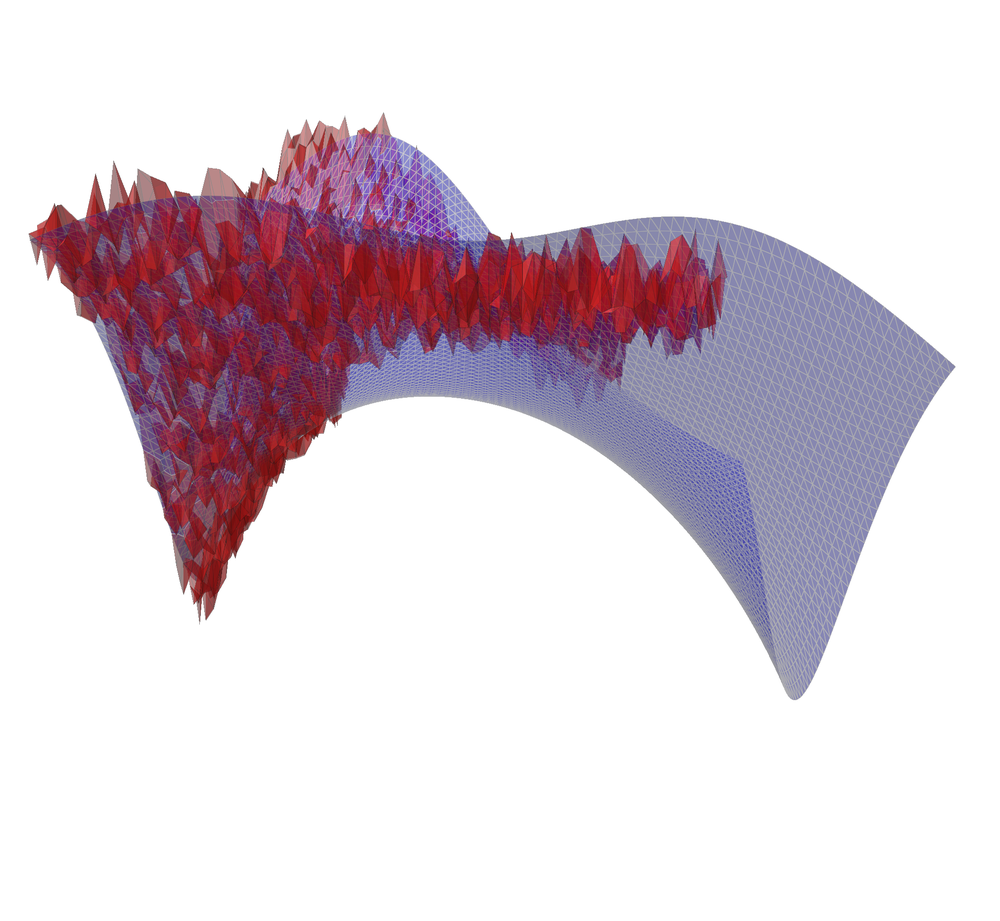}
\includegraphics[width=0.3\linewidth]{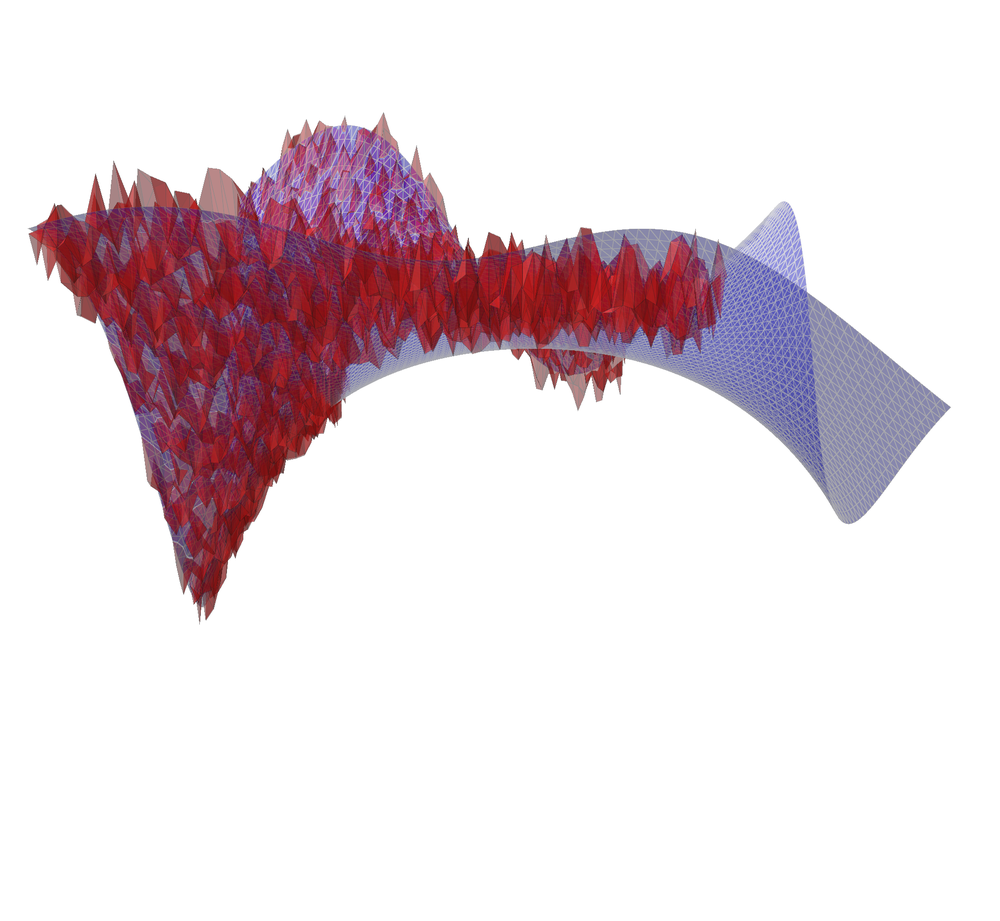}
\includegraphics[width=0.3\linewidth]{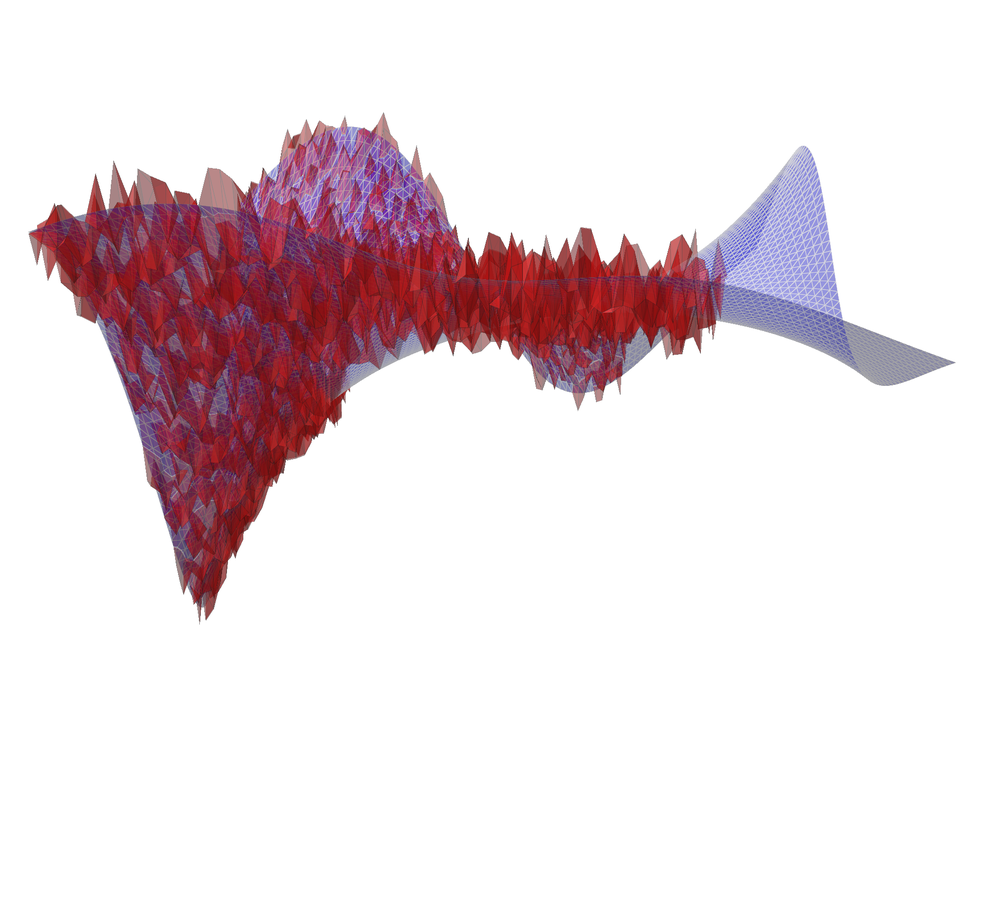}
\includegraphics[width=0.3\linewidth]{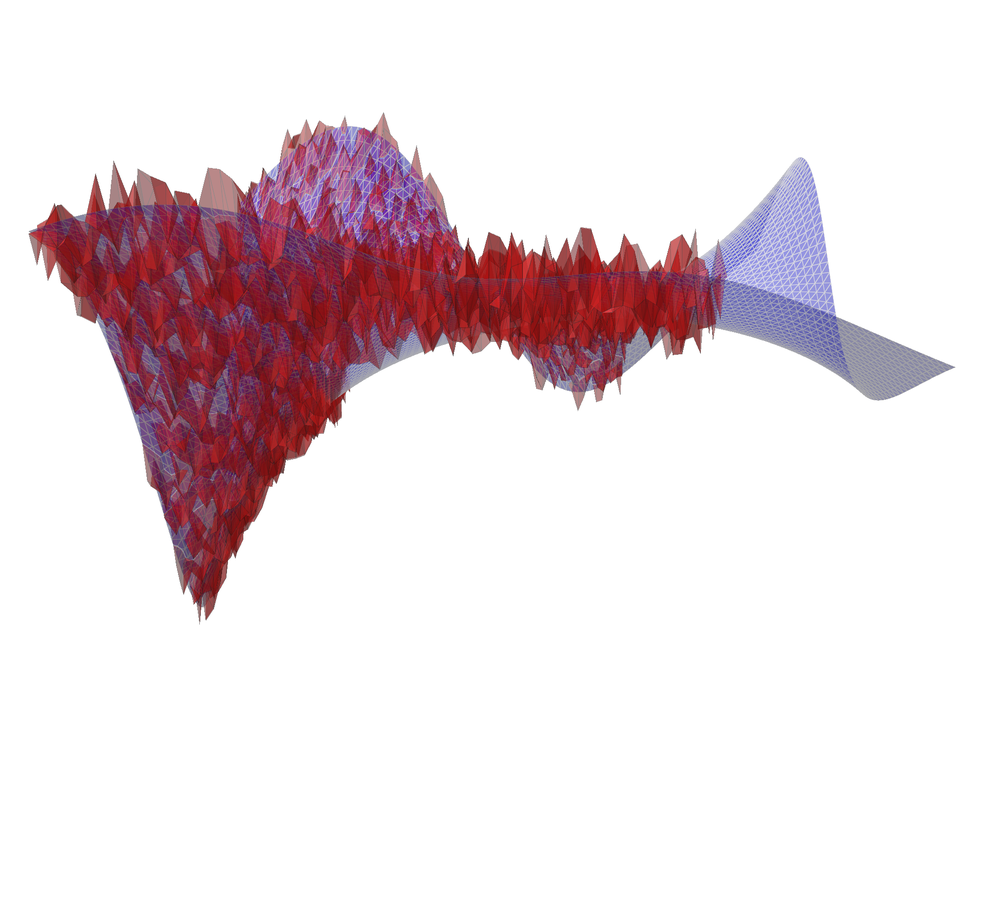}
\includegraphics[width=0.3\linewidth]{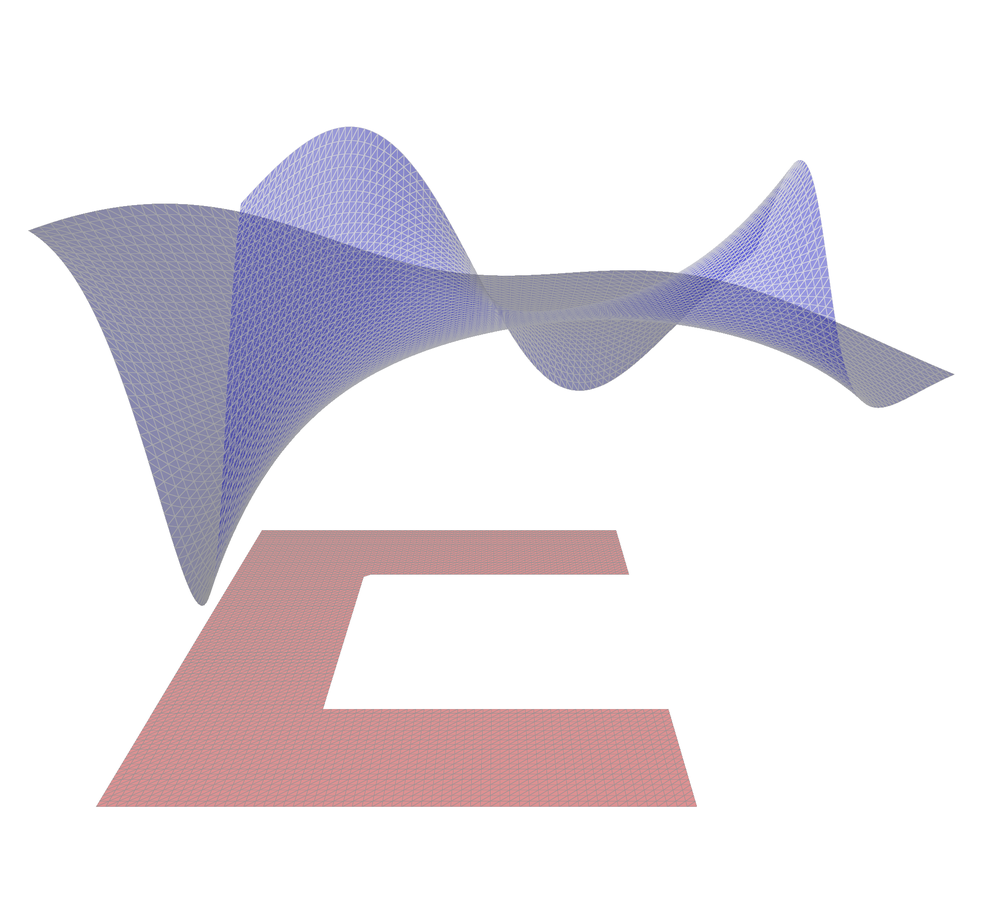}
\caption{Optimization process over a 9-dimensional \emph{coefficient space} for a \emph{linear} inverse problem with noisy data. Here, the FEM interior basis functions in Figure~\ref{fig_lindata_podbfs} are used. The unperturbed data is shown in the first frame. The second frame is the same as the first but with added noise sampled from $\mcU(-0.05, 0.05)$. The last frame shows $\omega$ and the reference solution used for the data which was obtained by taking $c_1 = 10, c_9 = 3.023,$ and all other $c_n$'s = 0 in \eqref{eq_linlin_lc}. The penultimate frame shows the optimization's MSE-converged reconstruction of the reference solution. The MSE converged after 861 iterations with the Adam optimizer with a step size = 0.1. This took 10.3 s on an Apple M1 CPU. The MSE's between the reference solution and the converged reconstruction are: on $\omega$ (used in optimization), $\text{MSE}_{\omega}$ = 8.28e-4; on the convex hull of $\omega$, $\text{MSE}_{\text{co}(\omega)}$ = 9.12e-4; and on its complement, $\text{MSE}_{\text{co}(\omega)^c}$ = 2.66e-3.}
\label{fig_linlin_optprocess}
\end{figure}

\subsection{Nonlinear Operator with Linear Data} \label{sec_exs_nold}

We again consider the linear data sets from the previous section, but here together with a \emph{nonlinear} differential operator. Because of the nonlinearity, we cannot use the FEM interior basis and the superposition principle as in the fully linear case. Instead, we use a neural network to approximate the solution operator, i.e., the inverse of the nonlinear differential operator. The solution is still in the form of a finite element function, so the output of the network gives an approximation of the finite element solution. The input to the network is the POD coefficients $(p_1, ..., p_N)$ corresponding to the same significant POD boundary basis functions as in the fully linear case. We use the following nonlinear energy functional as the foundation for the loss function during training of the network.
\begin{equation}
E(v) = \int_{\Omega} \frac{1}{2}(1 + v^2)|\nabla v|^2 \mathop{}\!\mathrm{d} x
\end{equation} 
This functional corresponds to the nonlinear differential operator whose inverse (the solution operator) we want to approximate with the neural network. We use a simple multilayer perceptron network architecture with 4 hidden layers of the same width X and an output layer of width O representing the finite element DoFs. For standard P1 elements considered here it is simply the finite element function's nodal values. We use the exponential linear unit (ELU) as the activation function in the 4 hidden layers and no activation function in the last layer. A schematic illustration of this network is provided in Figure~\ref{fig:networks-operator}.

In each iteration during the training, we pick a fixed number (referred to as the batch size) of randomly selected coefficient vectors and use them to compute an average loss. The coefficient values are picked from $\mcN(0, 0.09)$. The optimization is performed with the Adam optimizer where we perform $10^6$ iterations with a decreasing learning rate. The learning rate starts at 1e-4, and after every 250k iterations, it is decreased by a factor of~0.5.

To measure the well-trainedness of the network, we, as an initial guiding measure, use the zero energy $E(\phi_{u,N,h}(0))$, i.e., the value of the computed energy using the output from the network when an all zero vector is given as input. This, of course, corresponds to homogeneous Dirichlet boundary conditions and gives that the solution $u = 0$ and thus that $E(0) = 0$. We also perform more rigorous studies of well-trainedness by computing the actual finite element solution with FEniCS \cite{LoggEtal2012} and comparing it to the network approximation. This is done by computing their average norm difference over 1000 problems, where for each problem we randomly select a coefficient vector with values from $\mcN(0, 0.09)$. The difference is computed in both the $H_0^1$-norm ($H^1$-seminorm) and the $L^2$-norm. We also compute both the absolute and the relative norm differences, where the relative norm difference is the absolute difference divided by the norm of the finite element solution.

For the numerical examples we have again considered the two different coefficient vector lengths (9 and 21) and the four meshes from the linear case in the previous section. The network architectures and batch sizes used during training are given in Table~\ref{table:architectures}.

\begin{table}
  \caption{Network architectures and batch sizes used for the various mesh sizes. DoFs refers to the number of DoFs in the finite element space $V_h$, which is the dimension of the MLP's output vector. Width refers to the width of the four hidden layers in the MLP.}
  \label{table:architectures}
    \centering
    \begin{tabular}{lccc}
    \toprule
     Mesh & DoFs (O) & Width (X) & Batch size \\
    \midrule
    10x10 & 81 & 64 & 32 \\
    28x28 & 729 & 256 & 64 \\ 
    82x82 & 6561 & 512 & 64 \\ 
    244x244 & 59049 & 1024 & 96 \\ 
    \bottomrule
  \end{tabular}%
\end{table}

The hyperparameters of the network (number of layers, width, activation function) and training settings (optimizer, number of iterations, decreasing learning rate, batch size, etc.) have been obtained by trial and error, where we have looked at the zero energy. An intuition for the size of the network (number of layers and widths) is that it needs to be large enough to provide a good approximation of the solution operator, but sufficiently small so that training is feasible and economical. An intuition for using ELU is that it is smoother than many other activations, e.g., the commonly used ReLU, which is to some degree in accordance with PDE- theory, where the solution is expected to depend smoothly on the problem data, e.g., boundary values.

In Table~\ref{table:training-info}, we present training info for the four mesh sizes for coefficient vector length $N=9$ and $N=21$. The training has been performed on a single A100 GPU. For the largest mesh case (244x244), we have not been able to train with all elements present in the energy functional loss function (It has resulted in a NaN loss function value). To make it work, we have employed the trick of randomly selecting a fixed number of elements for every input vector during training, and only considering the energy functional contribution from those elements. The number of elements used is denoted ``Els'' in Table~\ref{table:training-info}. It can be observed from the zero energies and norm errors in Table~\ref{table:training-info} that the operator networks generally become more accurate with finer meshes \emph{if} all elements are used in the energy computation. In both cases with the 244x244 mesh, the trend in higher accuracy is broken. This is reasonable since only a few elements, instead of all, are used in the energy computation for the loss function.

\begin{table}
  \caption{Training info from using an A100 GPU.}
  \label{table:training-info}
  \begin{subtable}{\linewidth}
    \centering
    \caption{Input data size $N = 9$}
    \label{tab_ONtraininginfo_coeffs9}  
    \begin{small}%
    \begin{tabular}{lccccccc}
    \toprule
     Mesh & Els & \shortstack{Training \\ time} & \shortstack{GPU \\ Util} & \shortstack{Inference \\ time} & $E(\phi_{u,N,h}(0))$ & \shortstack{$H_0^1$-error \\ 1k-avg (rel)} & \shortstack{$L^2$-error \\ 1k-avg (rel)} \\
    \midrule
    10x10 & All & 358 s & 45\% & 0.8 ms & 3.9e-5 & 9.4e-3 (1.87\%) & 4.3e-4 (0.47\%) \\
    28x28 & All & 337 s & 73\% & 0.8 ms & 1.2e-6 & 2.1e-3 (0.7\%) & 9.7e-5 (0.18\%) \\ 
    82x82 & All & 615 s & 100\% & 0.8 ms & 3.2e-7 & 1.1e-3 (0.6\%) & 3.0e-5 (0.1\%) \\ 
    244x244 & 3k & 2673 s & 100\% & 0.7 ms & 5.2e-5 & 1.4e-2 (14.1\%) & 9.6e-5 (0.5\%) \\ 
    \bottomrule
  \end{tabular}%
  \end{small}%
\end{subtable}

\vspace{2ex}

\begin{subtable}{\linewidth}
  \centering
  \caption{Input data size $N = 21$}
\label{tab_ONtraininginfo_coeffs21}
\begin{small}%
\begin{tabular}{lccccccc}
  \toprule
   Mesh & Els & \shortstack{Training \\ time} & \shortstack{GPU \\ Util} & \shortstack{Inference \\ time} & $E(\phi_{u,N,h}(0))$ & \shortstack{$H_0^1$-error \\ 1k-avg (rel)} & \shortstack{$L^2$-error \\ 1k-avg (rel)} \\
  \midrule
  10x10 & All & 339 s & 47\% & 0.8 ms & 8.9e-5 & 1.6e-2 (1.23\%) & 1.2e-3 (1.07\%) \\
  28x28 & All & 354 s & 69\% & 0.8 ms & 5.1e-6 & 5.5e-3 (0.73\%) & 2.9e-4 (0.44\%) \\ 
  82x82 & All & 617 s & 100\% & 0.8 ms & 3.7e-6 & 3.6e-3 (0.82\%) & 8.4e-5 (0.22\%) \\ 
  244x244 & 4k & 2733 s & 100\% & 0.8 ms & 1.0e-4 & 2.5e-2 (9.8\%) & 1.6e-4 (0.72\%) \\
  \bottomrule
\end{tabular}%
\end{small}
\end{subtable}
  \end{table}

With these neural networks we may solve the inverse minimization problem over the coefficient space. In Figure~\ref{fig_nonlinlin_optprocess}, a demonstration of this process is presented for the case of 21 input coefficients and the 244x244 mesh, i.e., the neural network whose training info is presented in the last row of Table~\ref{tab_ONtraininginfo_coeffs21}. In Figure~\ref{fig_nonlinlin_optprocess} it can clearly be observed that the approximate finite element solution given by the operator network approaches the noisy data and the reference solution as the optimization progresses.
\begin{figure}
\centering
\includegraphics[width=0.3\linewidth]{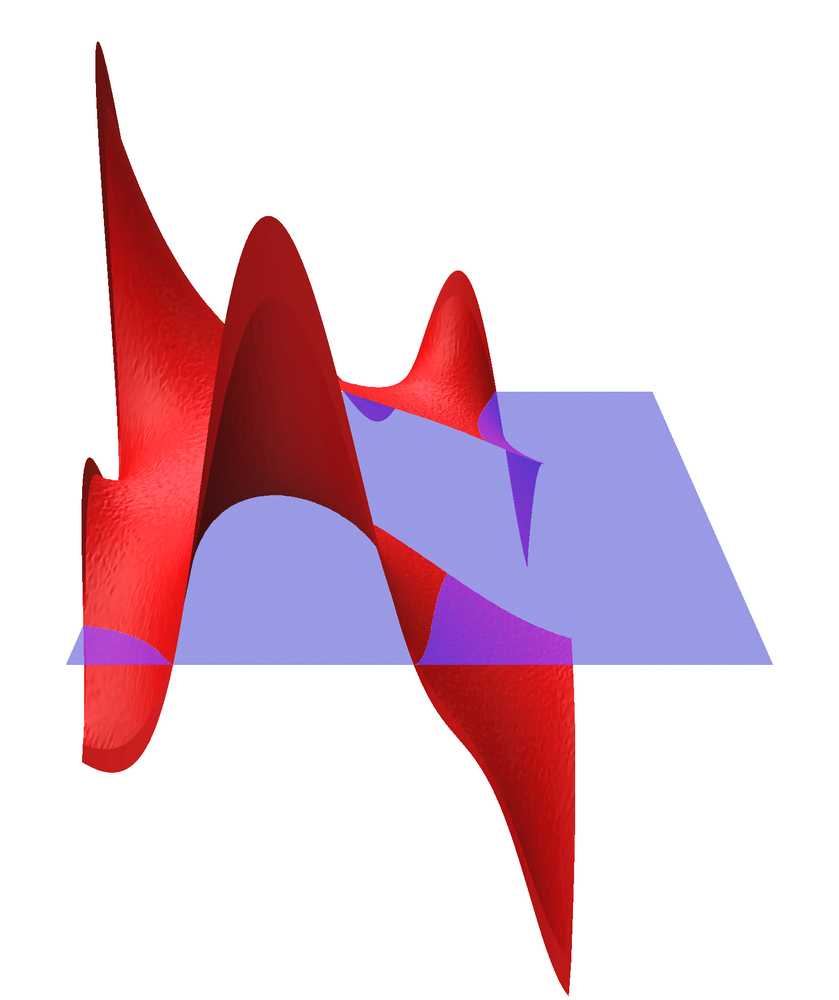}
\includegraphics[width=0.3\linewidth]{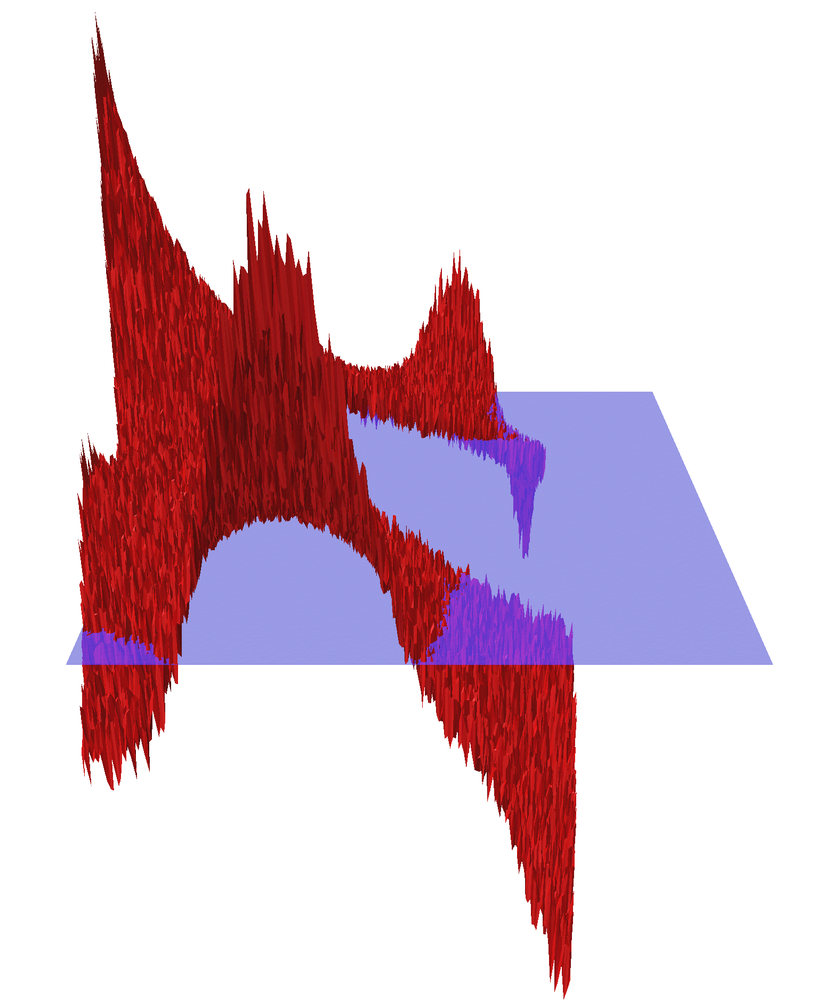}
\includegraphics[width=0.3\linewidth]{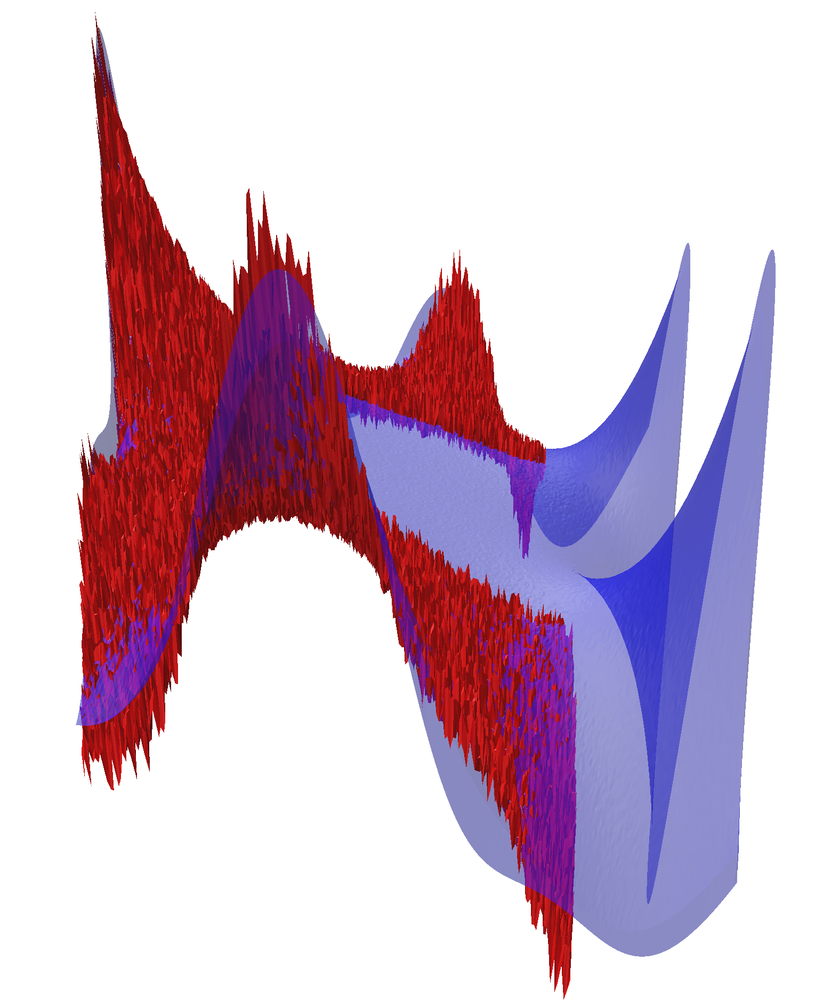}
\includegraphics[width=0.3\linewidth]{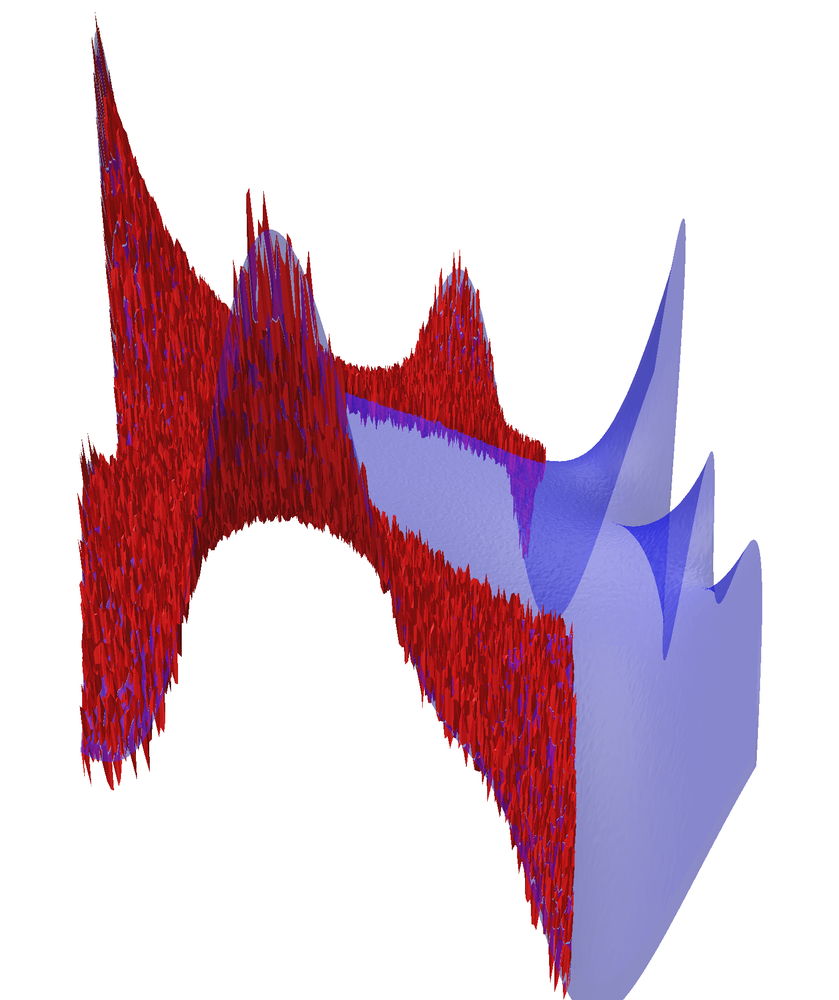}
\includegraphics[width=0.3\linewidth]{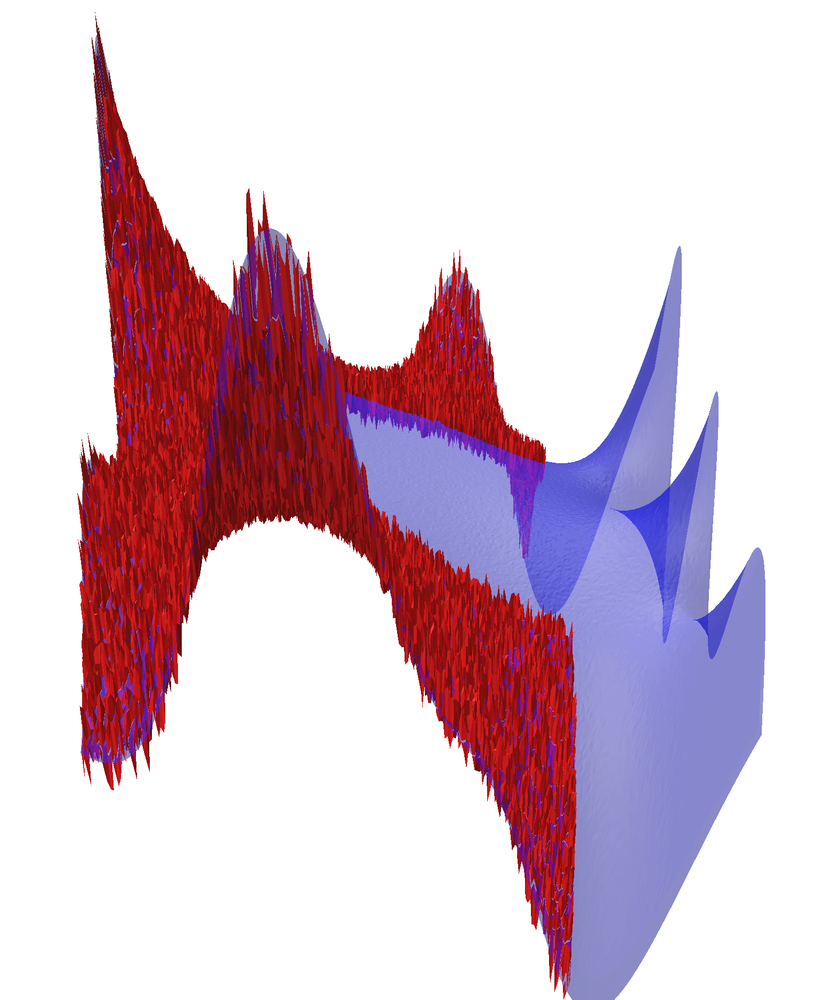}
\includegraphics[width=0.3\linewidth]{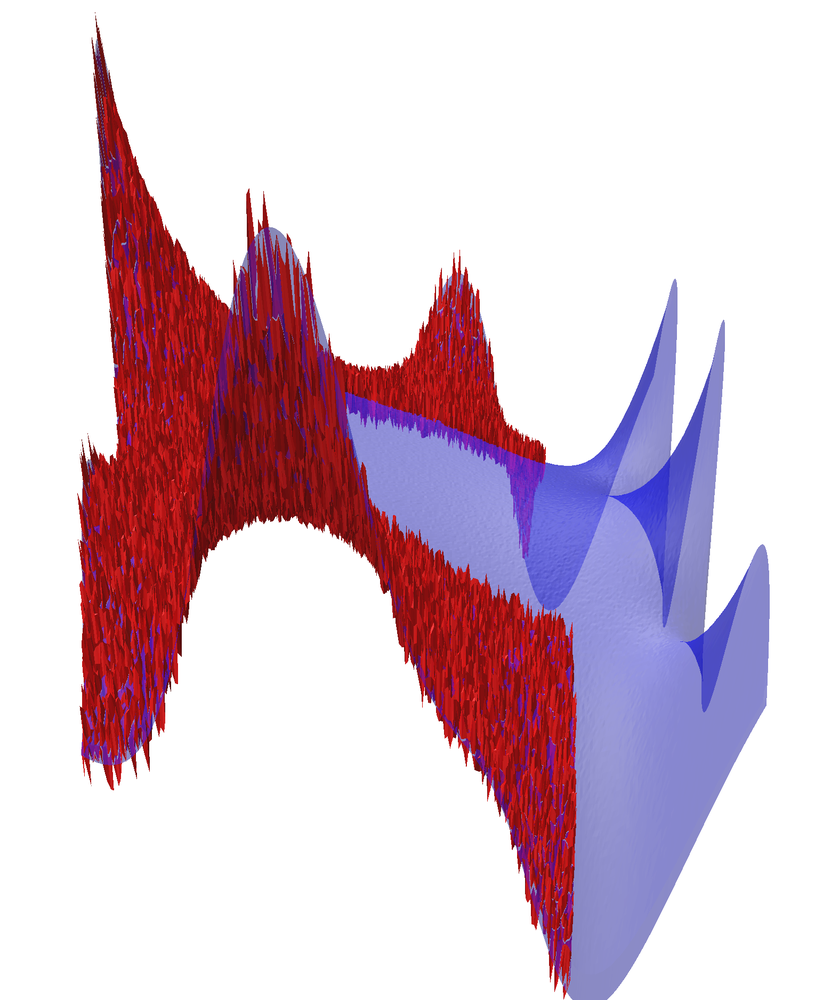}
\includegraphics[width=0.3\linewidth]{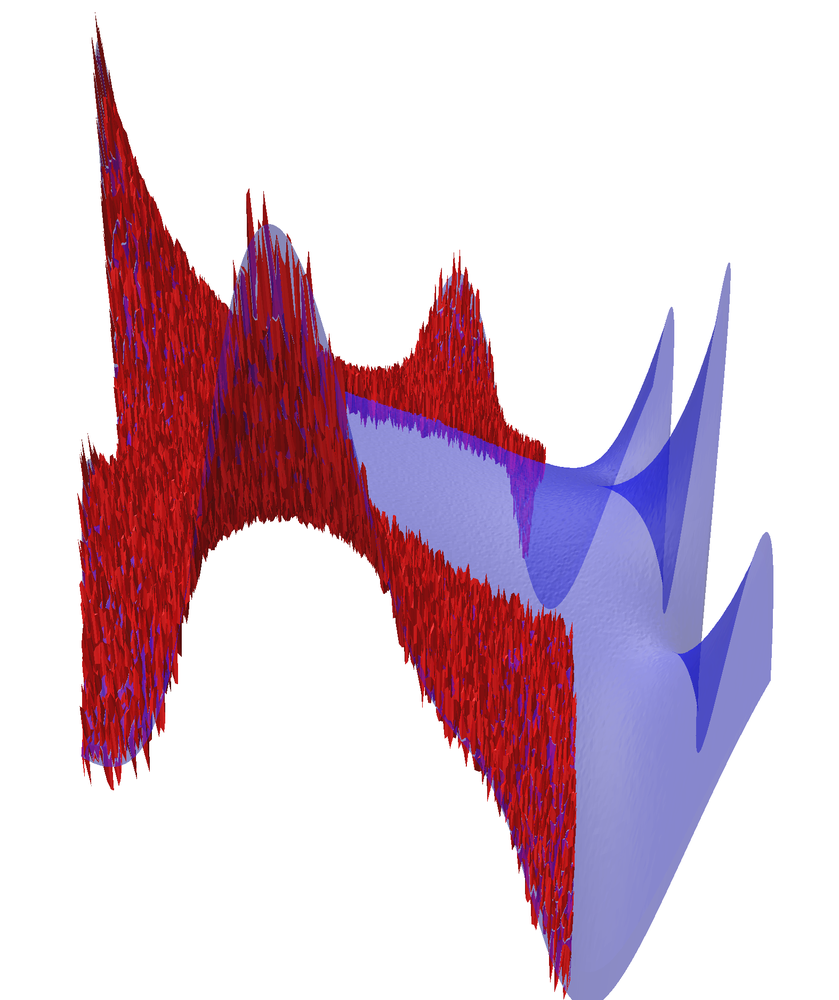}
\includegraphics[width=0.3\linewidth]{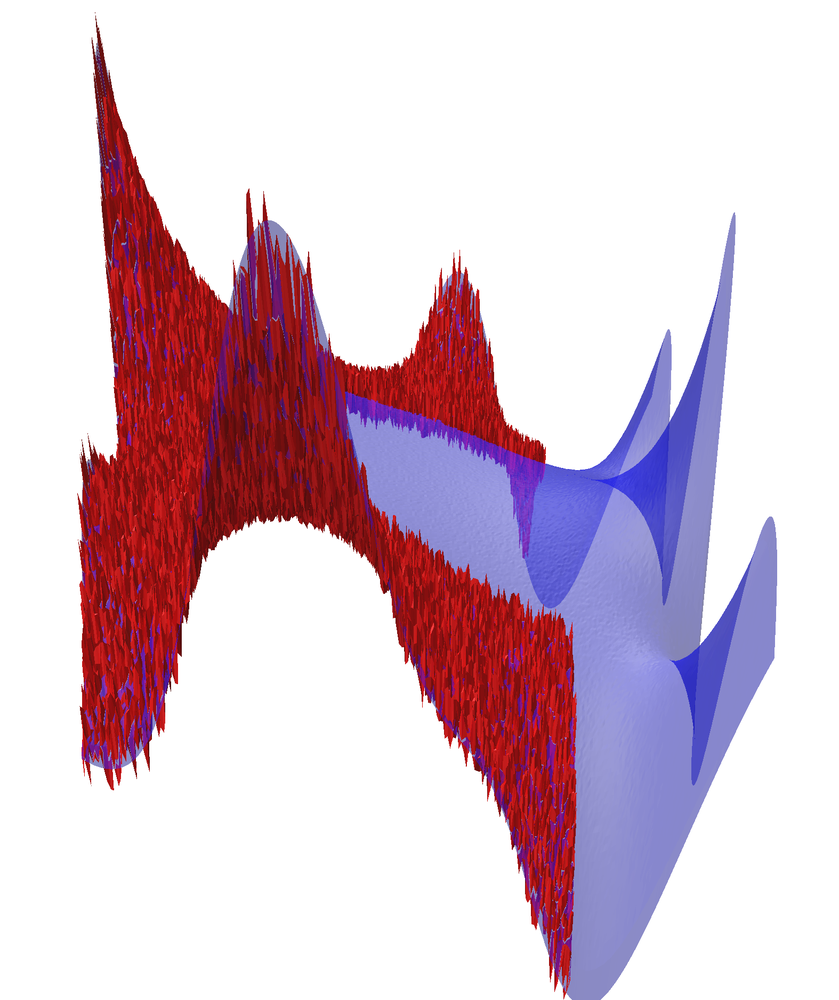}
\includegraphics[width=0.3\linewidth]{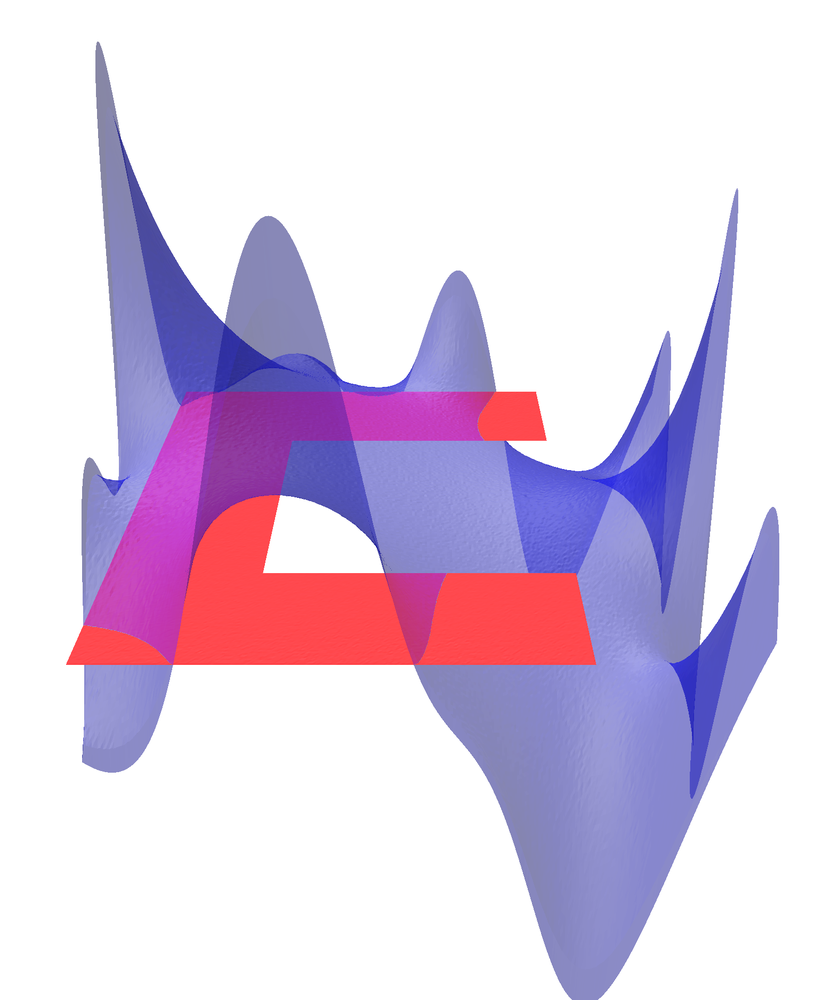}
\caption{Optimization process over a 21-dimensional \emph{coefficient space} for a \emph{nonlinear} inverse problem with noisy data. Here, the operator network in the last row of Table~\ref{tab_ONtraininginfo_coeffs21} (21 input coefficients, 59049 output DoFs) is used. The unperturbed data is shown in the first frame. The second frame is the same as the first but with added noise sampled from $\mcU(-0.05, 0.05)$. The last frame shows $\omega$ and the reference solution used for the data which was obtained by taking $p_{14} = 10$ and all other $p_n$'s = 0. The penultimate frame shows the optimization's MSE-converged reconstruction of the reference solution. The MSE converged after 2843 iterations with the Adam optimizer with a step size = 0.1. This took 140.2 s on an Apple M1 CPU. The MSE's between the reference solution and the converged reconstruction are: on $\omega$ (used in optimization), $\text{MSE}_{\omega}$ = 8.36e-4; on the convex hull of $\omega$, $\text{MSE}_{\text{co}(\omega)}$ = 8.38e-4; and on its complement, $\text{MSE}_{\text{co}(\omega)^c}$ = 1.91e-3.}
\label{fig_nonlinlin_optprocess}
\end{figure}

\subsection{Nonlinear Operator with Nonlinear Data}

We consider the same nonlinear differential operator with the same neural networks as in the previous section but here we add complexity by introducing an underlying nonlinear dependence on the input coefficients to the network. To construct such a nonlinear dependence we may pick a smooth function $a: X \rightarrow \IR^{|J|}$, where $X$ is a parameter domain in $\IR^{n_X}$, and for some index set $I$, consider boundary data of the form
\begin{align}
\mcG_a = \Bigl\{g_i = \sum_{j \in J } (a_j(x_i) + \delta_j) \varphi_{\mathrm{POD},j} \,|
\, i \in I \Bigr\}
\end{align}
where $\delta_j$ is some small probabilistic noise and $\{x_i\in X \,|\, i \in I\}$ is a set of samples from the parameter space $X$ 
equipped with a probability measure. In this case, we expect an autoencoder with a latent space $Z$ of at least the same dimension as $X$ to perform well.

\paragraph{Polynomial Data}
We consider a simple polynomial example where the coefficients $a \in \IR^{|J|}$ depend on the parameter variables $x \in \IR^{n_X}$ as 
\begin{align}
a = a(x) = A x + B x^2 + \delta
\end{align}
Here the matrices $A, B \in \IR^{|J| \times n_X}$ and their entries are randomly sampled from a uniform distribution. The perturbations $\delta \in \IR^{|J|}$ are sampled from a normal distribution.

For the numerical results we take ${|J|} = 9$, $n_X = 3$ and sample matrix entries from $\mcU(-1, 1)$ which are then held fixed. To generate coefficient vectors, we sample parameter variables $x_k \sim \mcU(-2, 2)$ and perturbations $\delta_j \sim \mcN(0, 1)$. We consider two cases: the linear case with $B = 0$ and the quadratic case with $B \neq 0$. To get a sense of what the data look like, we plot the coefficients as functions of the parameters for both cases in Figure~\ref{fig_nonlindata_poly_coeffs}.
\begin{figure}
\centering
\includegraphics[width=0.49\linewidth]{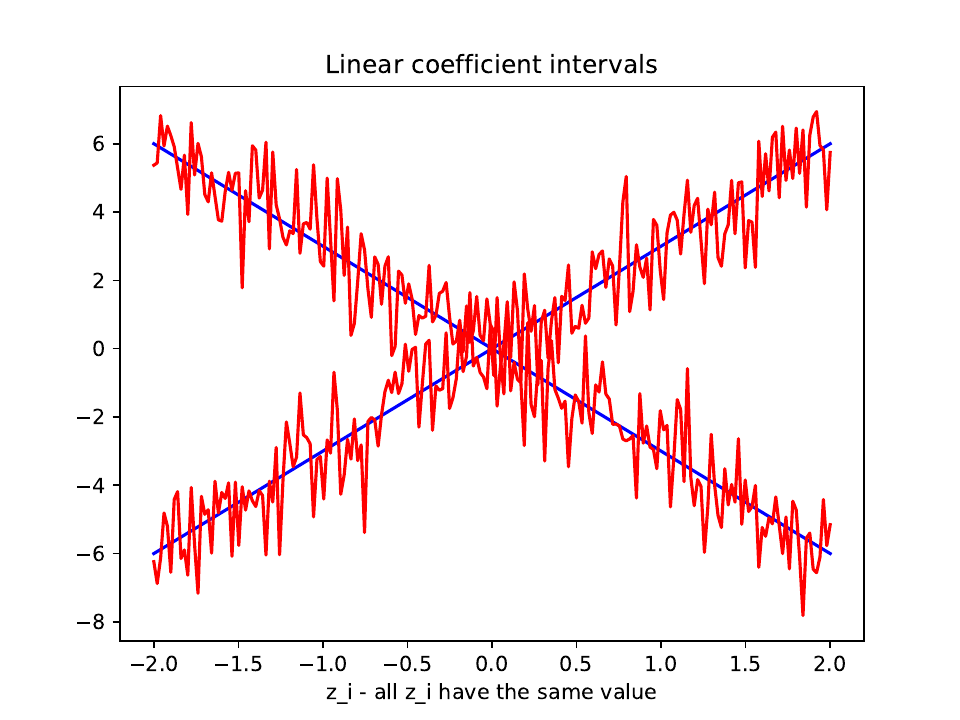}
\includegraphics[width=0.49\linewidth]{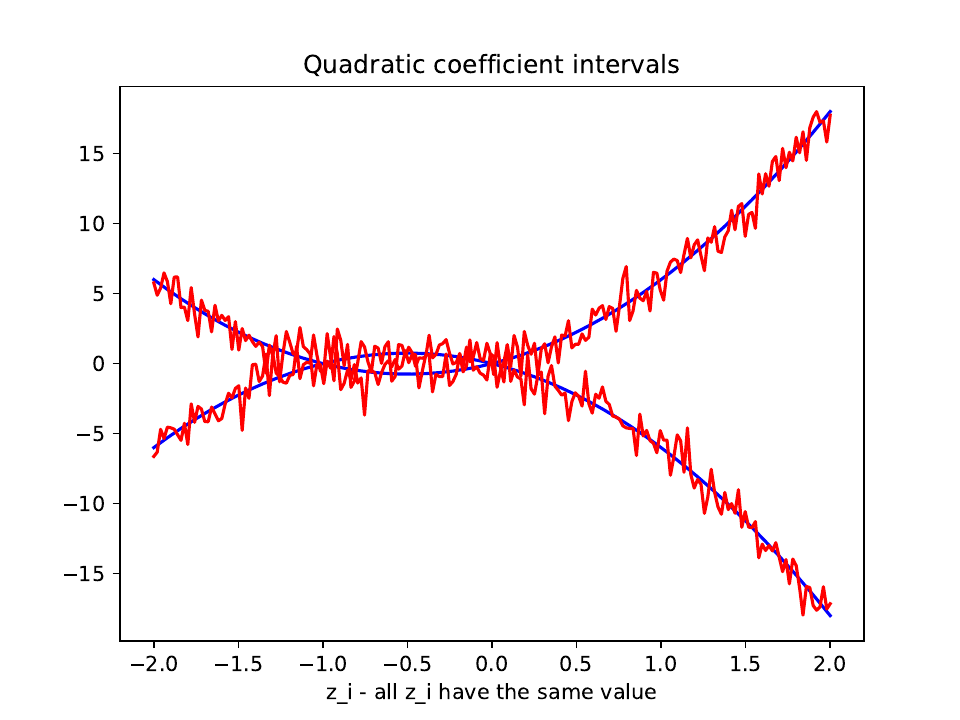}
\caption{Coefficients versus parameters for the linear case (left) and the quadratic case (right). All parameters have the same value which varies between $-2$ and $2$. For each case, there is an upper and lower bound for the coefficients obtained by taking all matrix entries to either have the minimum value $-1$ or the maximum value $1$. Both bounds are also shown as unperturbed (blue) and perturbed (red).}
\label{fig_nonlindata_poly_coeffs}
\end{figure}
We analyze data generated for the linear case with PCA and data generated for the quadratic case with both PCA and autoencoders. For the autoencoders we have used MLPs with 6 layers (5 hidden, 1 output) with the third layer being the latent layer. The latent layer width has been varied and the remaining hidden layer widths have all been fixed at 64. The activation function ELU has been applied to all layers except the last. The training has been performed with the Adam optimizer exactly as for the operator networks, i.e., $10^6$ iterations with a decreasing learning rate. The batch size has been 64. The hyperparameters of the autoencoders and training settings have, just as in the case of the operator networks, been obtained by trial and error. To measure well-trainedness, we have looked at the average mean squared reconstruction error over 1000 unperturbed samples generated in the same way as during training. The training time for a single autoencoder (fixed latent layer width) on an Apple M1 CPU has typically been in the range 240 -- 270 s.

The results from both the PCA and autoencoder analysis are presented in Figure~\ref{fig_nonlindata_poly_pcaAE}. The PCA results give a 3-dimensional latent space in the linear case and a 6-dimensional in the quadratic. This is evident from the number of significant singular values for the different cases. The autoencoder results suggest the existence of both a 3- and a 6-dimensional latent space in the quadratic case. This can be deduced from the two plateaus for the two perturbed cases: one at latent layer widths 3 -- 5 and one at 6 -- 8. The autoencoders thus manage to find the underlying 3-dimensional structure in the quadratic case whereas PCA does not.
\begin{figure}
\centering
\includegraphics[width=0.49\linewidth]{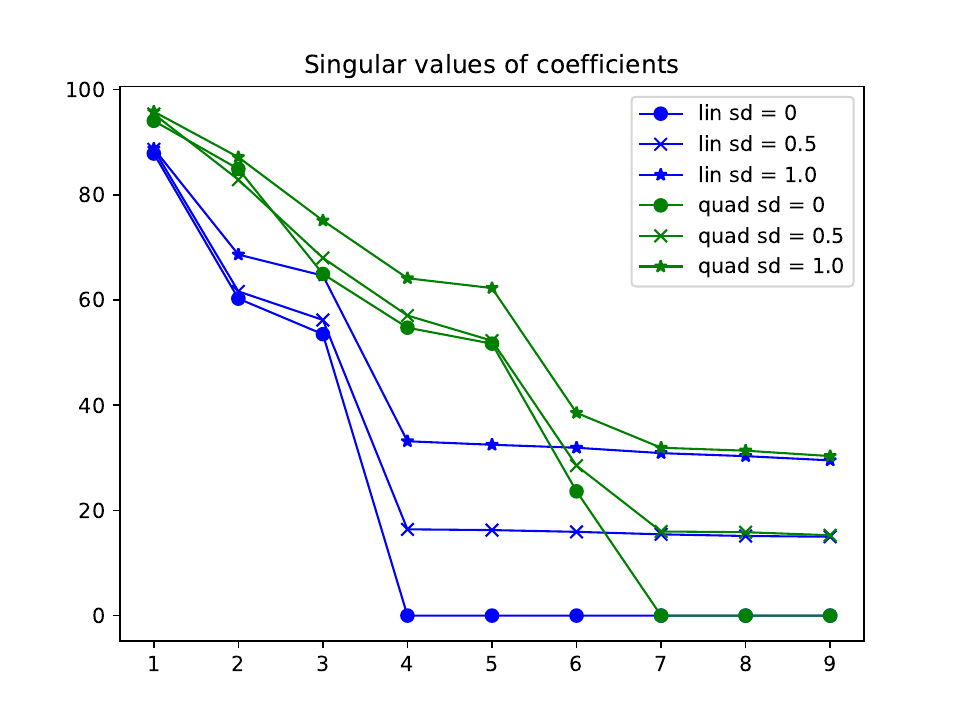}
\includegraphics[width=0.49\linewidth]{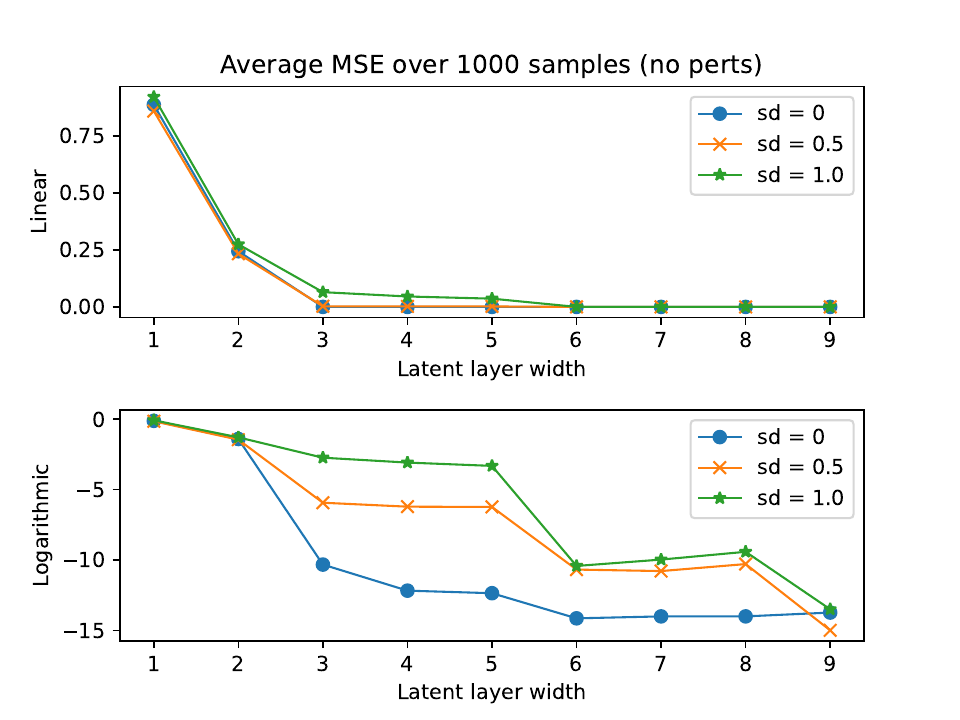}
\caption{\textbf{Left:} PCA results for both the linear and quadratic case. The plots show the singular values of the coefficients in decreasing order for three different perturbations: unperturbed and two perturbed (standard deviation = 0.5 and 1). \textbf{Right:} Autoencoder results for the quadratic case for the same three perturbations as PCA but used during training. The plots show the average mean squared reconstruction error over unperturbed test data for different latent layer widths both on a linear and logarithmic scale.}
\label{fig_nonlindata_poly_pcaAE}
\end{figure}

\paragraph{Gaussian Data}
We consider a more advanced nonlinear example where the coefficients $a \in \IR^{|J|}$ depend on the parameter variables $x \in \IR^{n_X}$ as 
\begin{align}
a_j = a_j(x) = \exp(-\gamma(x_k - x_{0, l} )^2) + \delta_j
\label{eq_gaussdata_coeffs}
\end{align}
Here we have $L$ number of equidistant Gaussian bell curves indexed by $l$ where each coefficient is assigned exactly one bell curve with midpoint $x_{0, l}$ and exactly one parameter $x_k$ according to $l = j \mod L$ and $k = j \mod n_X$, respectively. The perturbations $\delta \in \IR^{|J|}$ are sampled from a normal distribution.

For the numerical results we take $\gamma = 2$ and sample perturbations $\delta_j \sim \mcN(0, 0.0225)$ (standard deviation = 0.15). We consider four cases:
\begin{itemize}
\item $(n_X, L) = (2, 5)$ with $x_0 = (0, 4, 8, 12, 16)$ and $x_k \sim \mcU(-2, 18)$
\item $(n_X, L) = (3, 6)$ with $x_0 = (0, 2, 4, 6, 8, 10)$ and $x_k \sim \mcU(-2, 12)$
\item $(n_X, L) = (3, 7)$ with $x_0 = (0, 2, 4, 6, 8, 10, 12)$ and $x_k \sim \mcU(-2, 14)$
\item $(n_X, L) = (4, 8)$ with $x_0 = (0, 2, 4, 6, 8, 10, 12, 14)$ and $x_k \sim \mcU(-2, 16)$
\end{itemize}
We analyze data generated for these cases with both PCA and autoencoders. For the autoencoders we have used MLPs with 5 layers (4 hidden, 1 output) with the middle layer being the latent layer. The latent layer width has been varied and the remaining hidden layer widths have all been fixed at 64. The activation function ELU has been applied to all layers except the last. The training has been performed with the Adam optimizer exactly as for the operator networks, i.e., $10^6$ iterations with a decreasing learning rate. The batch size has been 64. Again, the hyperparameters of the autoencoders and training settings have been obtained by trial and error, where we have looked at the average mean squared reconstruction error over 1000 unperturbed samples generated in the same way as during training. The training time for a single autoencoder (fixed latent layer width) on an Apple M1 CPU has typically been in the range 210 -- 250 s.

The bell curves for the coefficients, PCA results and autoencoder results are presented in Figure~\ref{fig:gaussian-data-example}. The PCA results show something interesting. If the number of bell curves $L$ is divisible by the latent dimension $n_X$, PCA gives that the underlying structure has dimension $L$. If $L$ is \emph{not} divisible by $n_X$, PCA instead gives that this dimension is $n_X L$. For example, for $(n_X, L) = (2, 5)$ in Figure~\ref{fig_nonlindata_gauss_ldim2_gb5}, PCA gives latent dimension = 10, and for $(n_X, L) = (3, 6)$ in Figure~\ref{fig_nonlindata_gauss_ldim3_gb6}, PCA gives latent dimension = 6. This phenomenon is easily understood by the number of unique combinations of latent parameters $x_k$ and bell curves, characterized by $x_{0, l}$, in the construction of the coefficients given by \eqref{eq_gaussdata_coeffs}. The autoencoder results all suggest the existence of latent spaces of a lower dimension than given by PCA. This is most clearly seen from the existence of plateaus for the two perturbed cases (standard deviation = 0.075 and 0.15) on the logarithmic scale in all four cases. However, the suggested latent dimension does match the actual one as well as in the previous example with polynomial data, hinting at the higher complexity of the Gaussian data. This is especially true in the cases where $n_X$ does not divide $L$.

\begin{figure}
\begin{subfigure}[t]{\linewidth}
\centering
\includegraphics[width=0.32\linewidth]{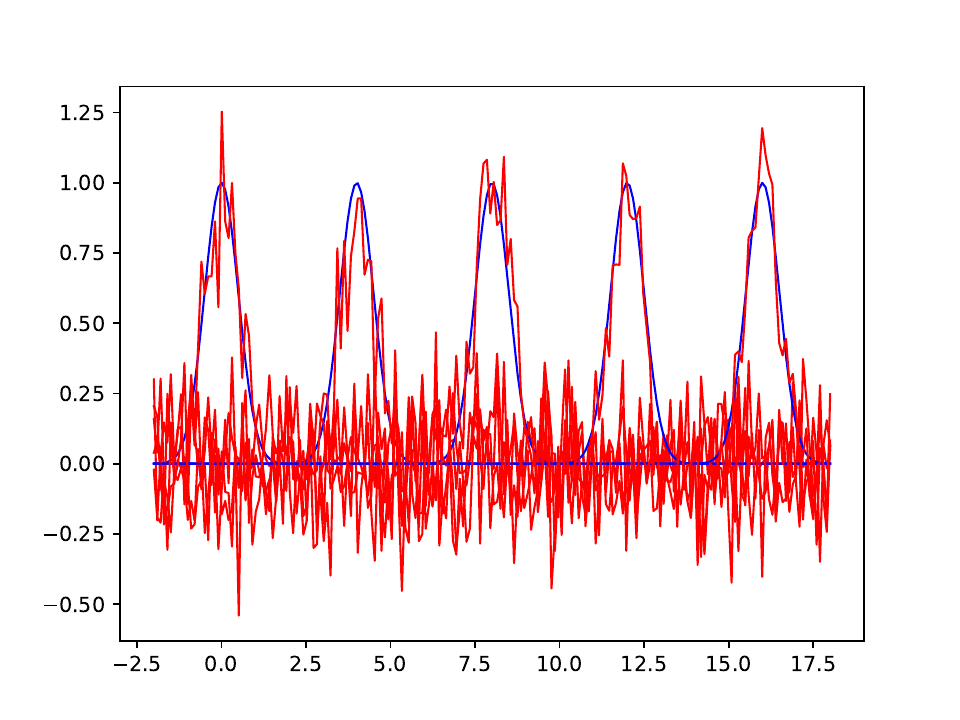}
\includegraphics[width=0.32\linewidth]{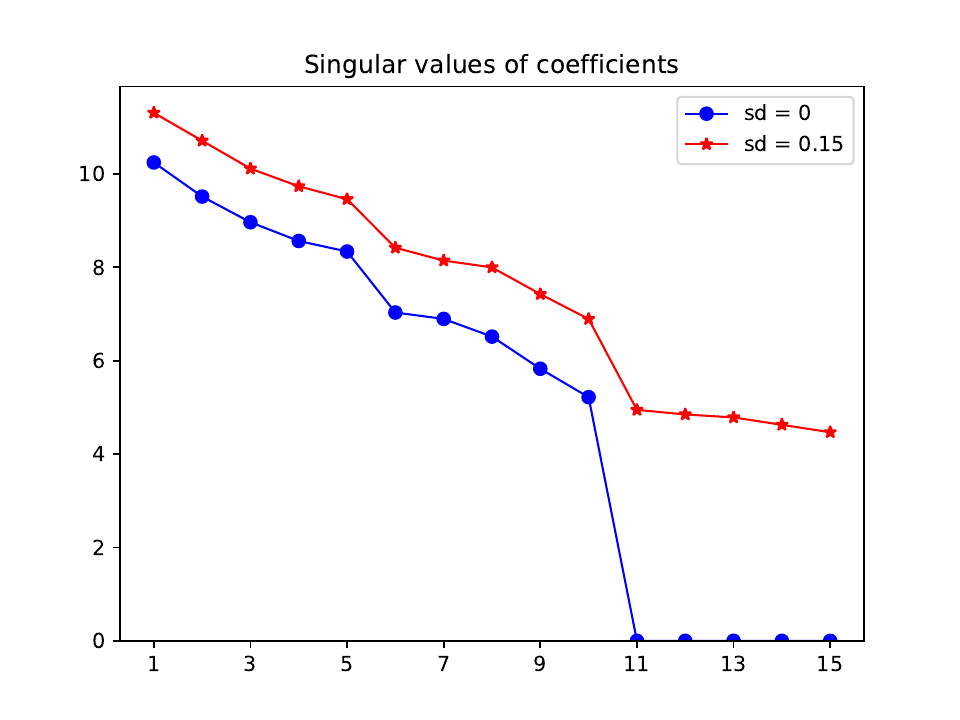}
\includegraphics[width=0.32\linewidth]{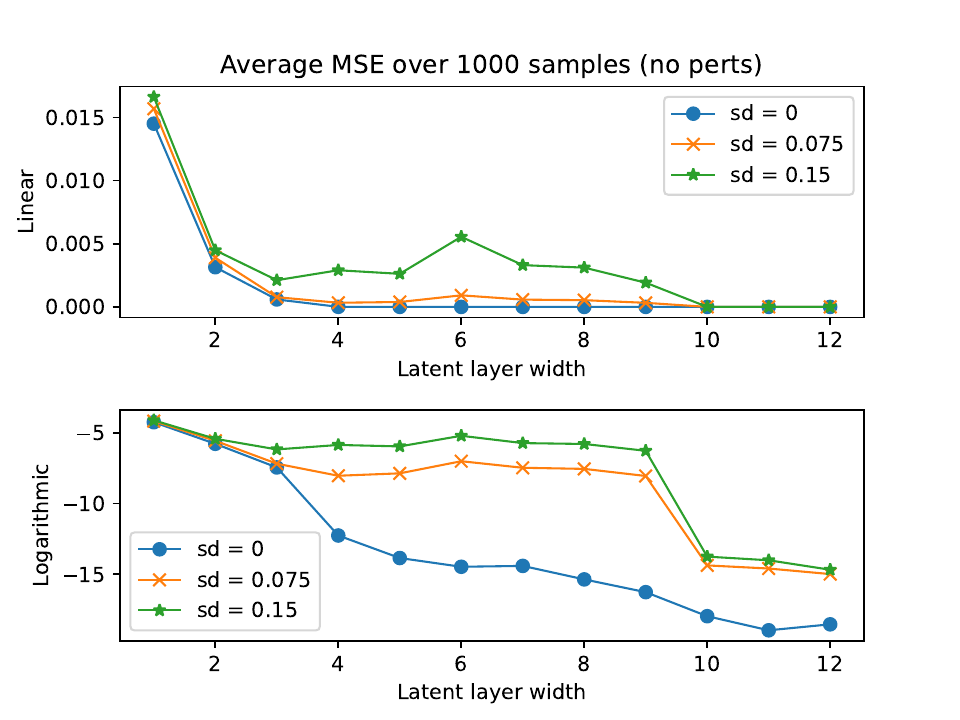}
\subcaption{$n_X = 2$, $L = 5$}
\label{fig_nonlindata_gauss_ldim2_gb5}
\end{subfigure}
\begin{subfigure}[t]{\linewidth}
\centering
\includegraphics[width=0.32\linewidth]{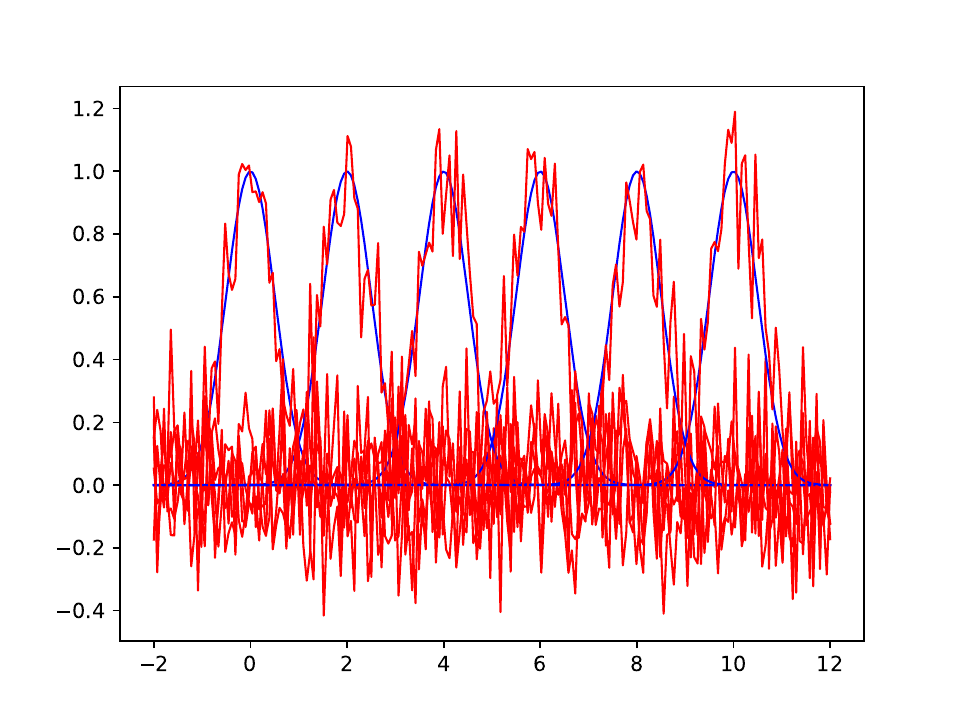}
\includegraphics[width=0.32\linewidth]{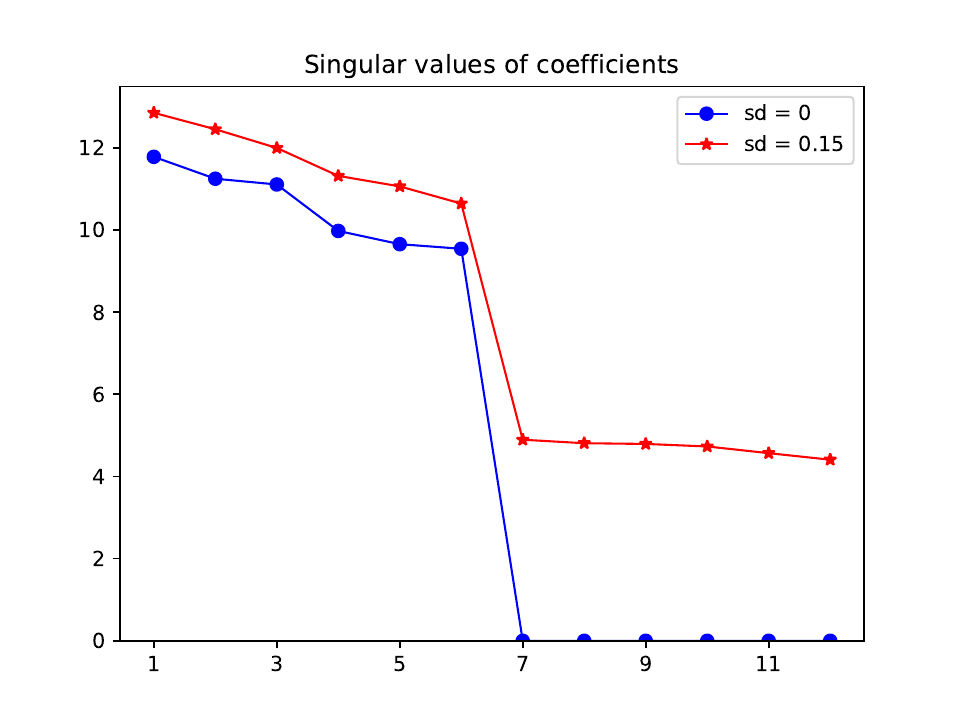}
\includegraphics[width=0.32\linewidth]{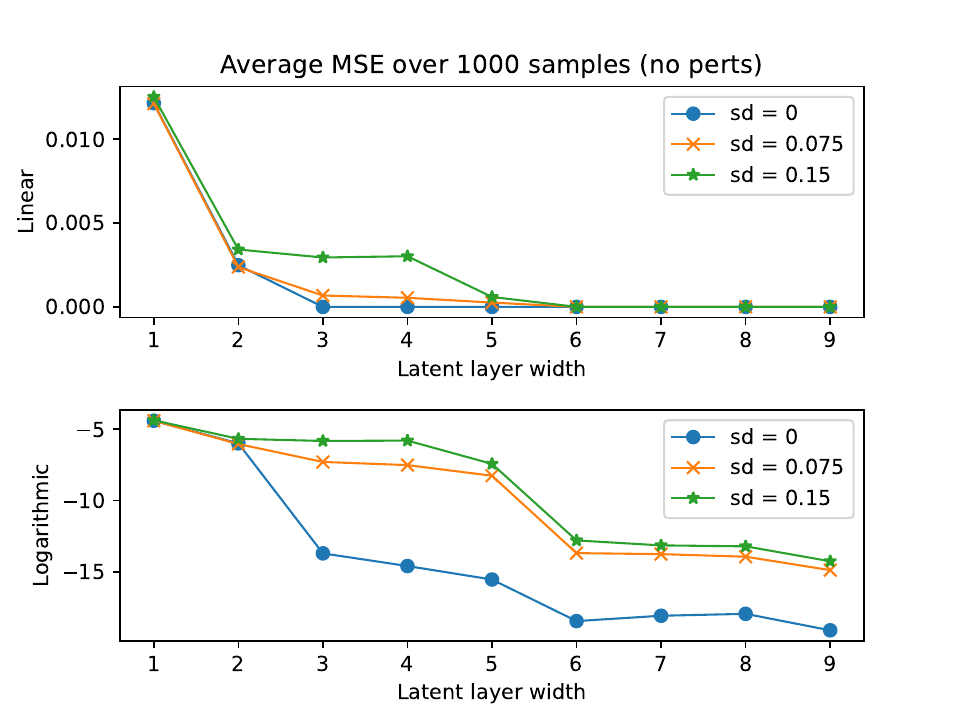}
\subcaption{$n_X = 3$, $L = 6$}
\label{fig_nonlindata_gauss_ldim3_gb6}
\end{subfigure}
\begin{subfigure}[t]{\linewidth}
\centering
\includegraphics[width=0.32\linewidth]{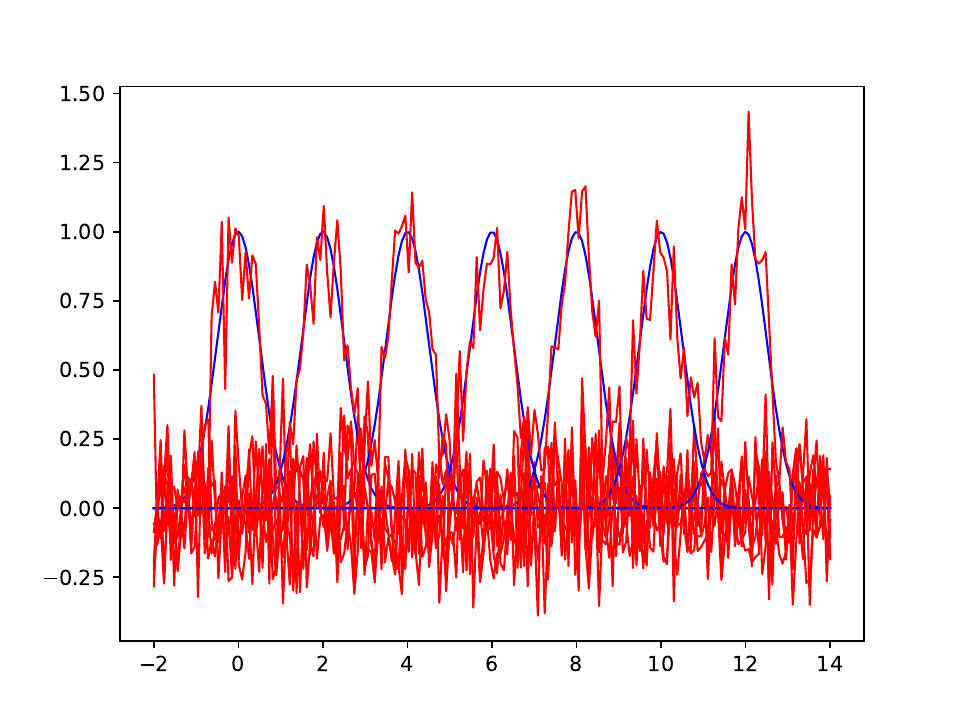}
\includegraphics[width=0.32\linewidth]{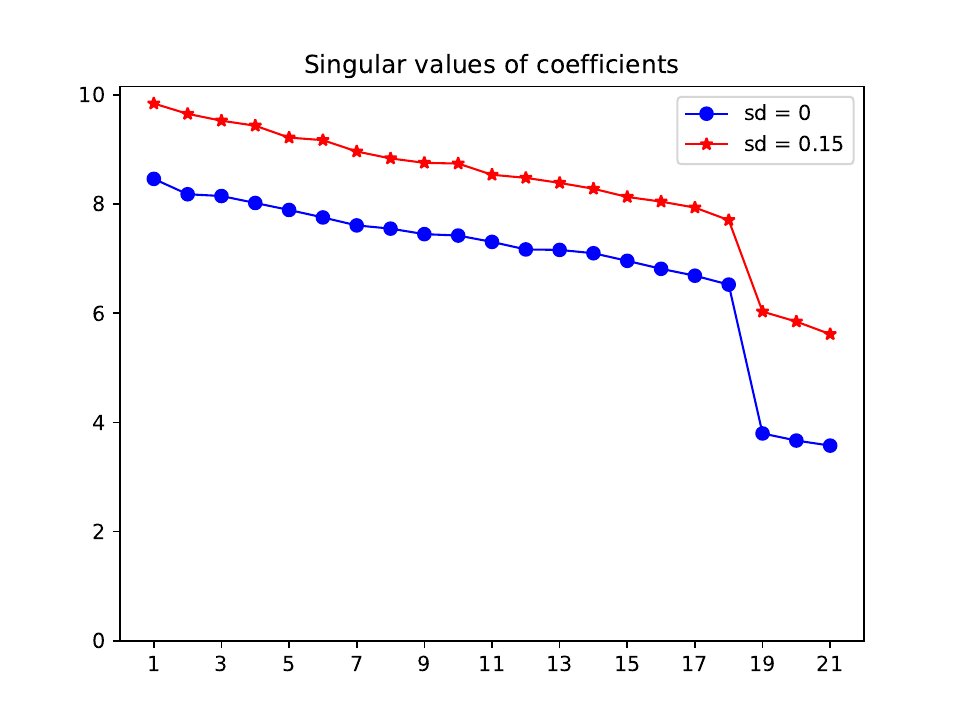}
\includegraphics[width=0.32\linewidth]{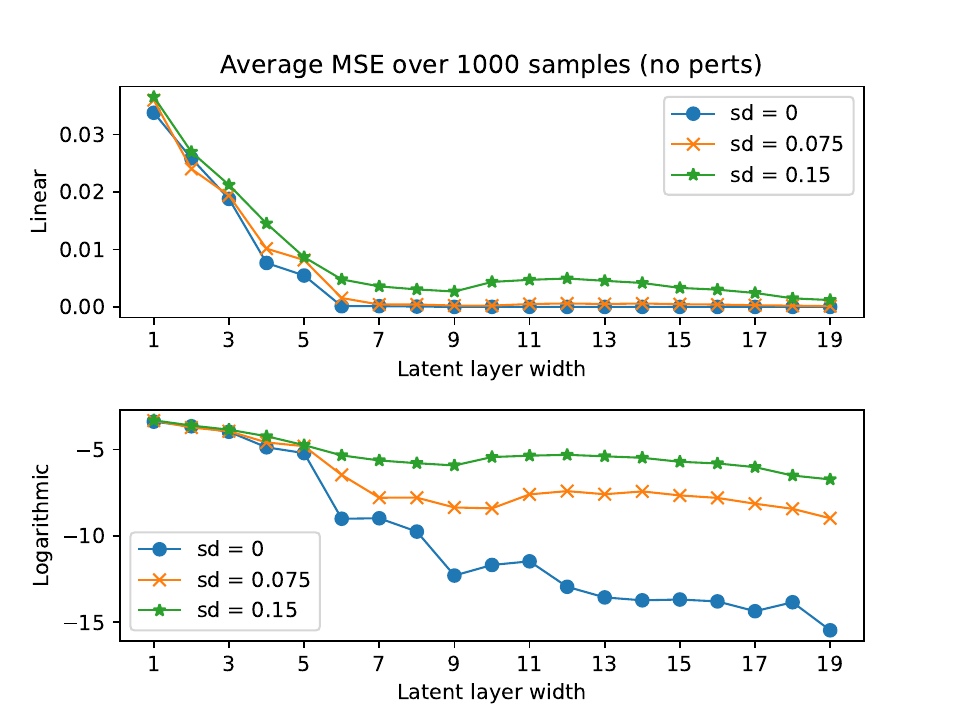}
\subcaption{$n_X = 3$, $L = 7$}
\label{fig_nonlindata_gauss_ldim3_gb7}
\end{subfigure}
\begin{subfigure}[t]{\linewidth}
\centering
\includegraphics[width=0.32\linewidth]{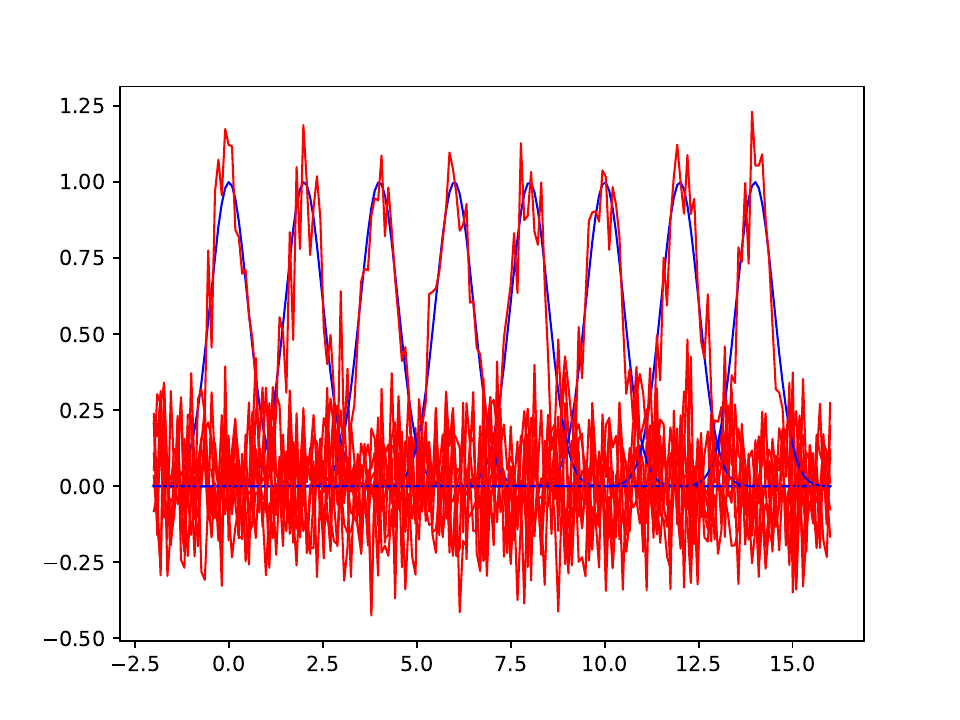}
\includegraphics[width=0.32\linewidth]{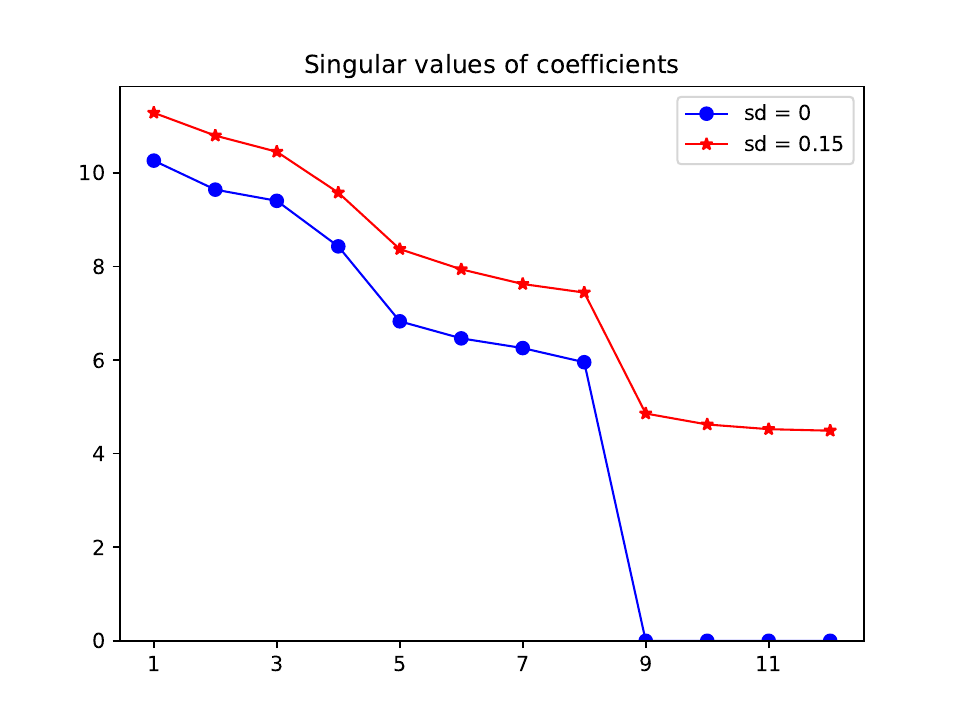}
\includegraphics[width=0.32\linewidth]{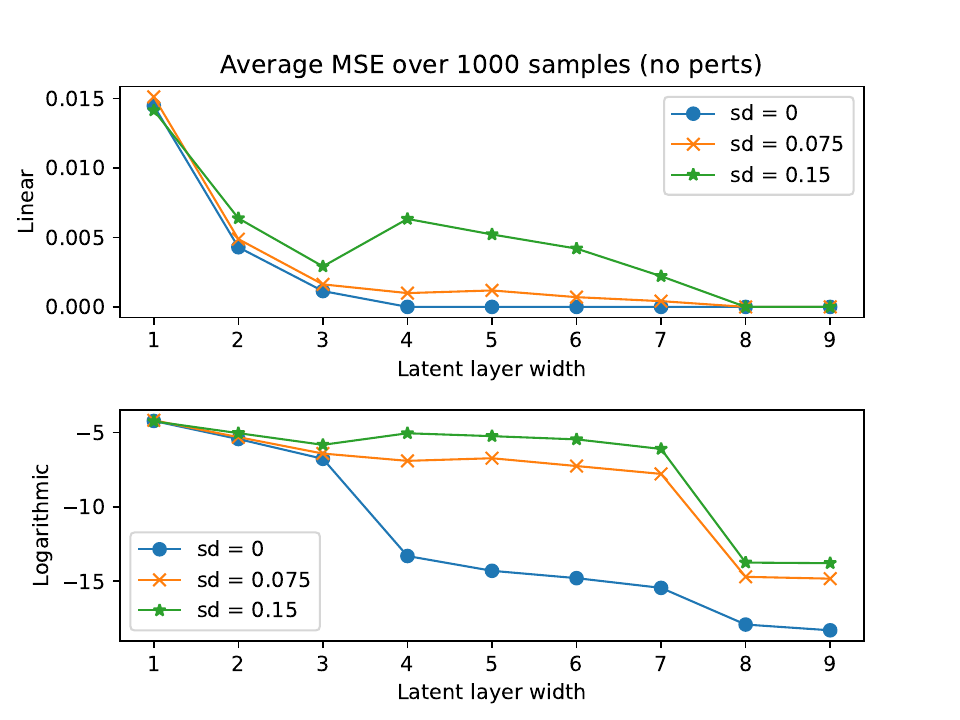}
\subcaption{$n_X = 4$, $L = 8$}
\label{fig_nonlindata_gauss_ldim4_gb8}
\end{subfigure}

\caption{Gaussian data examples. \textbf{Left:} Bell curves used for the coefficients, unperturbed (blue) and perturbed (red). \textbf{Middle:} PCA results for unperturbed data (blue) and perturbed (red, standard deviation = 0.15). The plots show the singular values of the coefficients in decreasing order. \textbf{Right:} Autoencoder results for three different perturbations used during training: unperturbed and two perturbed (standard deviation = 0.075 and 0.15). The plots show the average mean squared reconstruction error over unperturbed test data for different latent layer widths both on a linear and logarithmic scale.}
\label{fig:gaussian-data-example}
\end{figure}

\paragraph{Combining Operator Network with Decoder}
In the third Gaussian data example with results presented in Figure~\ref{fig_nonlindata_gauss_ldim3_gb7}, we have $(n_X, L) = (3, 7)$. Here the PCA suggests that the underlying dimension is 21 (number of significant singular values), whereas the corresponding autoencoder study suggests that a reduction down to 9 dimensions could provide the same improvement as a reduction down to 17 in the case of the autoencoders trained on perturbed data (9 and 17 give roughly the same error). In light of the above, we may take an autoencoder with latent layer width = 9 from this case and connect its decoder to the input of the operator network for the 244x244 mesh with 21 input coefficients. We may thus solve an inverse minimization problem over a 9-dimensional latent space instead of a 21-dimensional coefficient space. We present demonstrations of this process in Figures~\ref{fig_nonlinnonlin_optprocess_unpertdec_unpertomega} -- \ref{fig_nonlinnonlin_optprocess}. A summary of the optimization results for these three demonstrations and also the one in Figure~\ref{fig_nonlinlin_optprocess} is given in Table~\ref{table:nonlinear_opt_res}.

The main difference between the three demonstrations is the decoder used. First in Figures~\ref{fig_nonlinnonlin_optprocess_unpertdec_unpertomega} -- \ref{fig_nonlinnonlin_optprocess_unpertdec}, we use the decoder from the ``sd = 0'' autoencoder, meaning it was trained on unperturbed data. The first of these two demonstrations is for clean data, $u_0$ in $\omega$, and the second for noisy. We see that the two optimization processes are essentially the same but find it instructive to present both as the clean data case functions as a reference. Second, in Figure~\ref{fig_nonlinnonlin_optprocess}, we use the decoder from the ``sd = 0.15'' autoencoder, meaning it was trained on perturbed data with perturbations from $\mcN(0, 0.0225)$. From the logarithmic scale in the right frame in Figure~\ref{fig_nonlindata_gauss_ldim3_gb7} we see that the reconstruction errors of the two autoencoders differ substantially, by several orders of magnitude. Comparing the corresponding optimization processes, we also see that using the ``sd = 0'' decoder (Figure~\ref{fig_nonlinnonlin_optprocess_unpertdec}) produces a much more accurate reconstruction compared to the ``sd = 0.15'' decoder (Figure~\ref{fig_nonlinnonlin_optprocess}) that fails to do so.

The reconstructions in all three decoder cases, and especially the last, are less accurate compared to the case with only the operator network presented in Figure~\ref{fig_nonlinlin_optprocess}, as can be seen from both the figures and the MSE's in Table~\ref{table:nonlinear_opt_res}. This is reasonable since the reference solution in all four cases is the same network output corresponding to a specific \emph{coefficient} input and in the case with only the operator network we optimize in this coefficient space whereas in the decoder cases in some latent space. It is simply not guaranteed that the decoders may attain this specific coefficient input when mapping from the latent space. One reason being that a single change in any of the 9 latent variables can affect all the 21 coefficients. Comparing the MSE's on the different subdomains in Table~\ref{table:nonlinear_opt_res}, we see that in all four cases it is smaller on the convex hull of $\omega$ than on the complement as expected. This is also true for the fully linear case (corresponding results are presented in the caption of Figure~\ref{fig_linlin_optprocess}). The average iteration times presented in Table~\ref{table:nonlinear_opt_res} are essentially the same for the four cases. Something that is positive for using decoders, but maybe not so surprising considering how much smaller the decoder MLP's are in comparison to the operator MLP. In summary, autoencoders may be used to reduce the dimension of the optimization space (latent instead of coefficient space), but to really gain from such a reduction and to maintain accuracy, care needs to be taken in how the reduction mapping is constructed. We point out that the MLP approach considered here is rather simple and that we believe there is substantial room for improvement by considering more sophisticated methods.

As final remarks we point out that taking some output of the method under consideration as the reference solution, as is done here, is typically not a proper choice since it is too idealized. However, here we make this choice to put more focus on the effects of latent space optimization. We also point out that all the optimization processes involving neural networks presented here have been for the rougher networks: the operator network in the last row of Table~\ref{tab_ONtraininginfo_coeffs21} and the autoencoders in Figure~\ref{fig_nonlindata_gauss_ldim3_gb7} have alternatives with better measures of well-trainedness. The idea behind this being that if the concept works to some degree in the harder cases, it should work even better in the easier ones.

\begin{figure}
\centering
\includegraphics[width=0.3\linewidth]{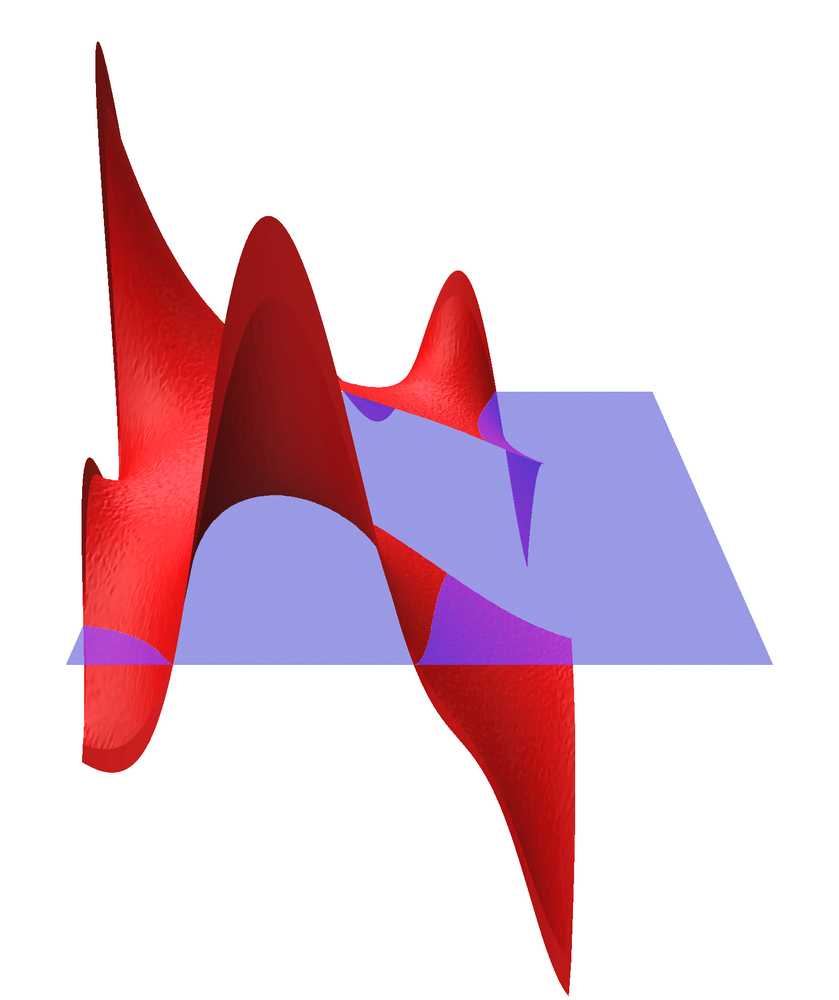}
\includegraphics[width=0.3\linewidth]{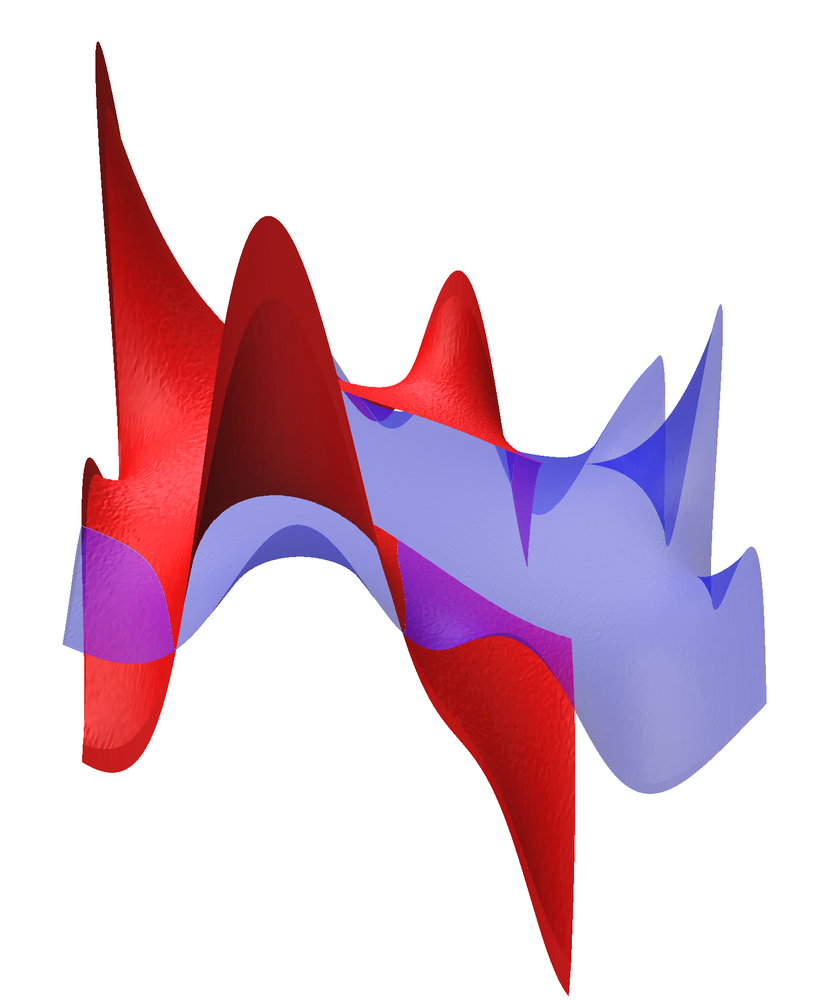}
\includegraphics[width=0.3\linewidth]{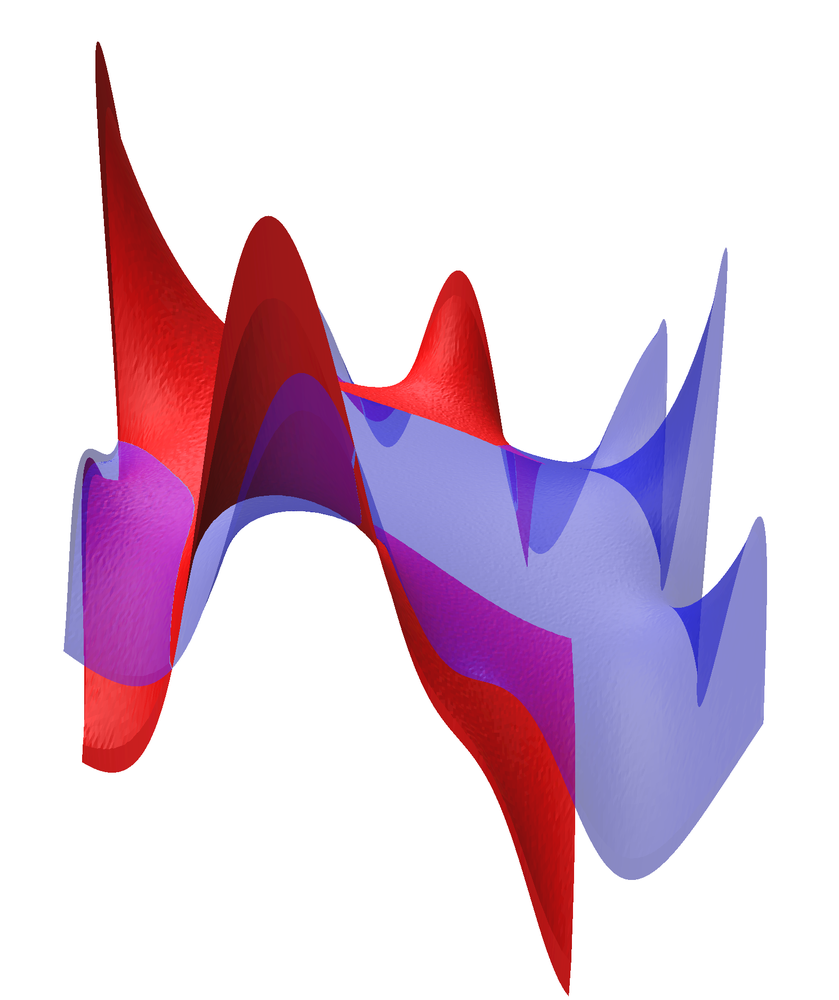}
\includegraphics[width=0.3\linewidth]{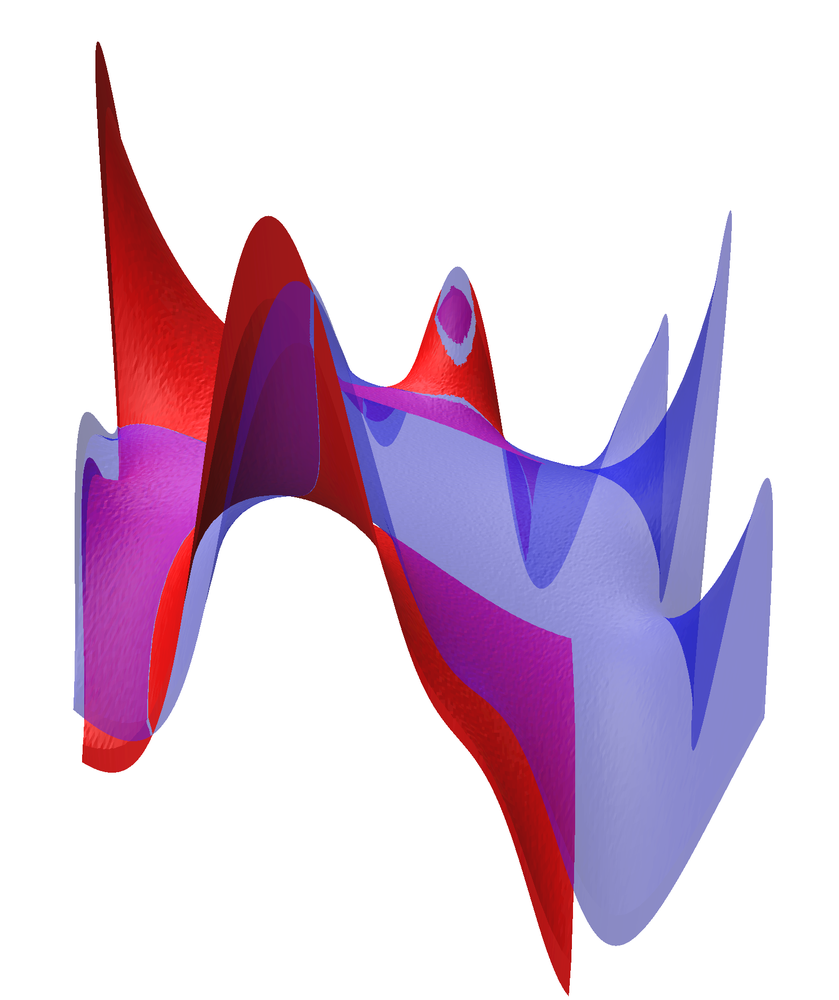}
\includegraphics[width=0.3\linewidth]{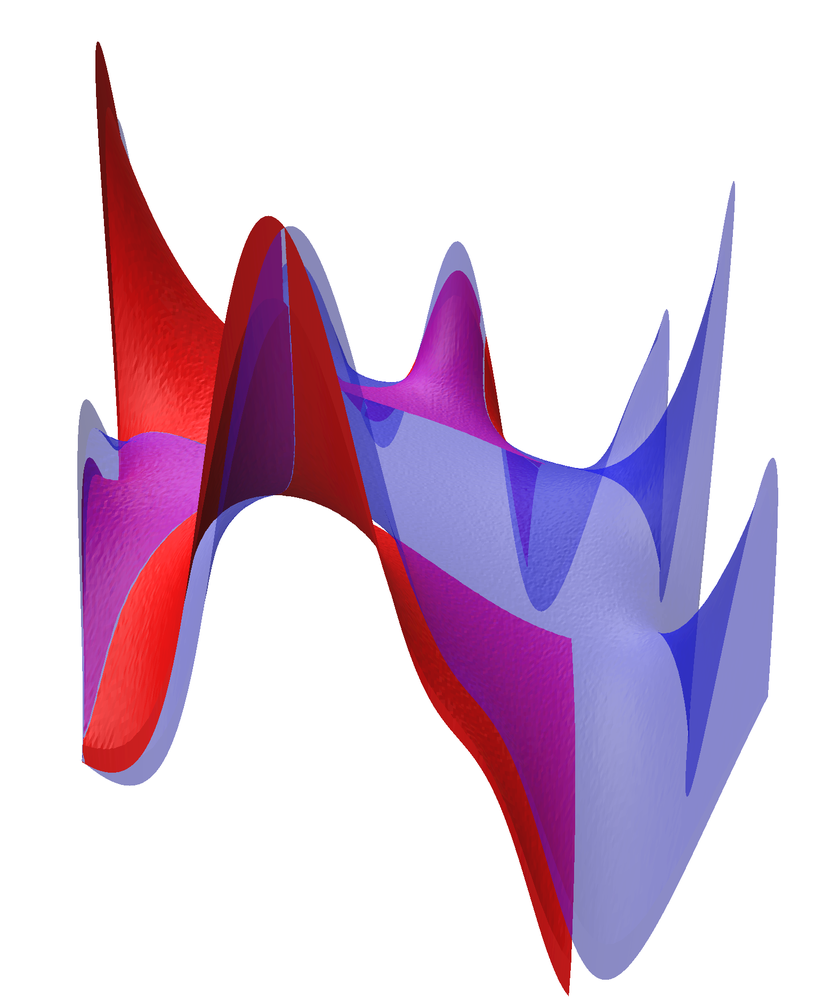}
\includegraphics[width=0.3\linewidth]{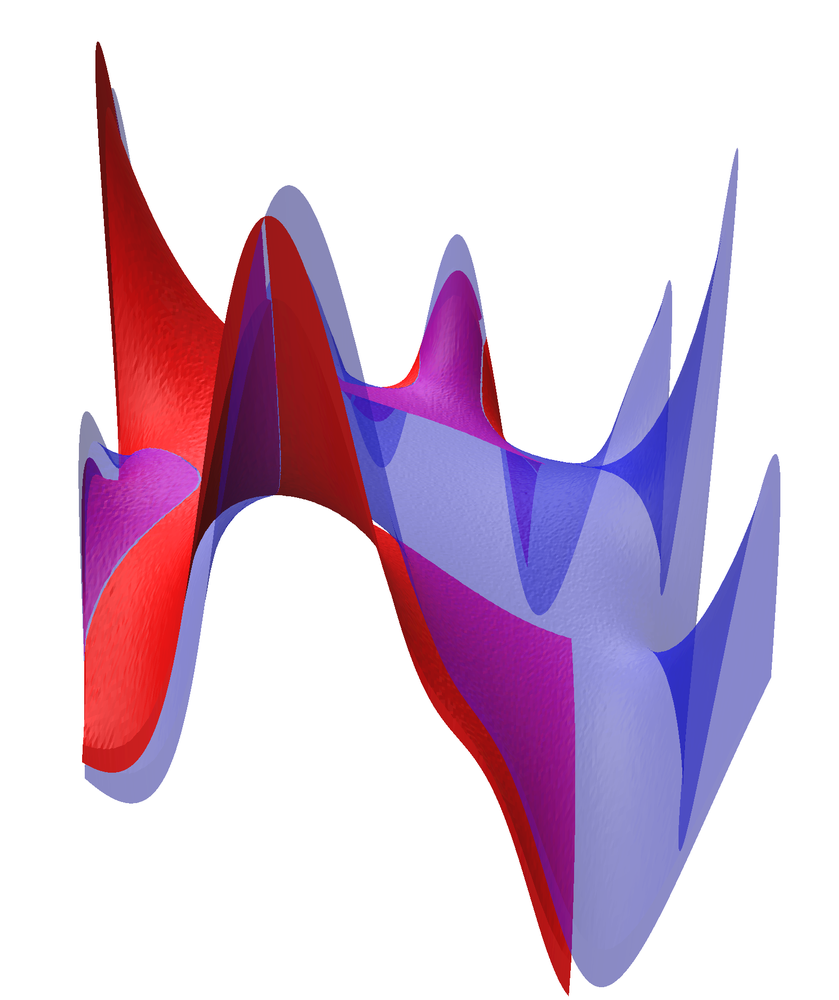}
\includegraphics[width=0.3\linewidth]{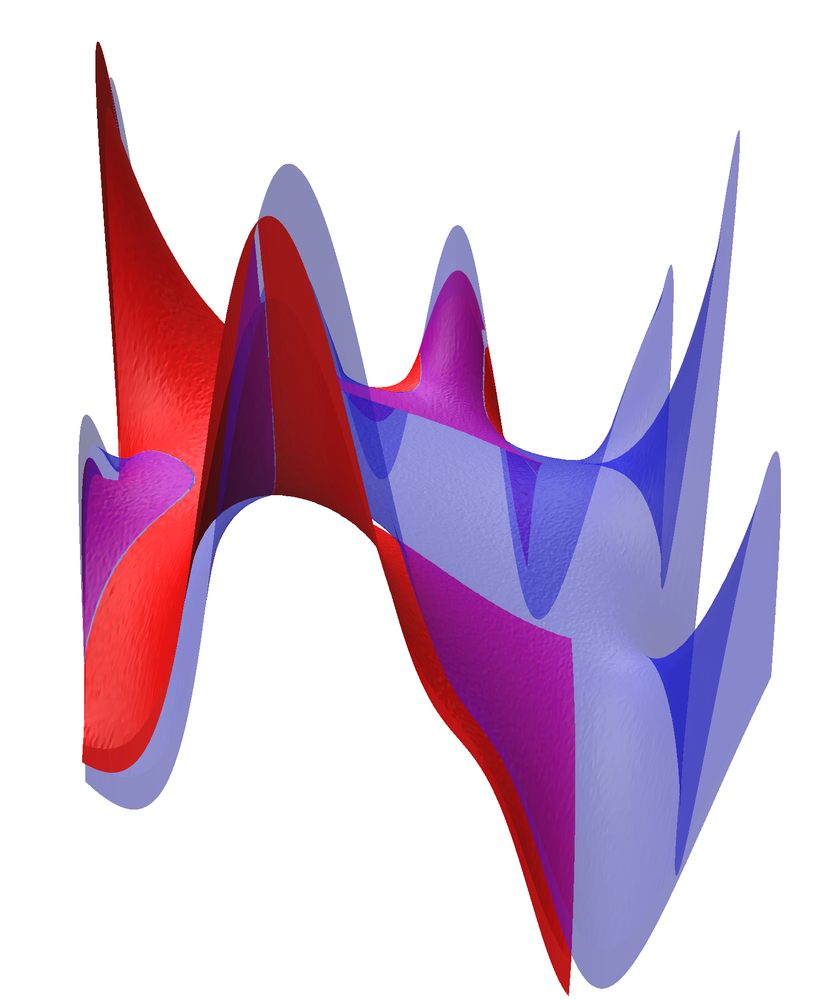}
\includegraphics[width=0.3\linewidth]{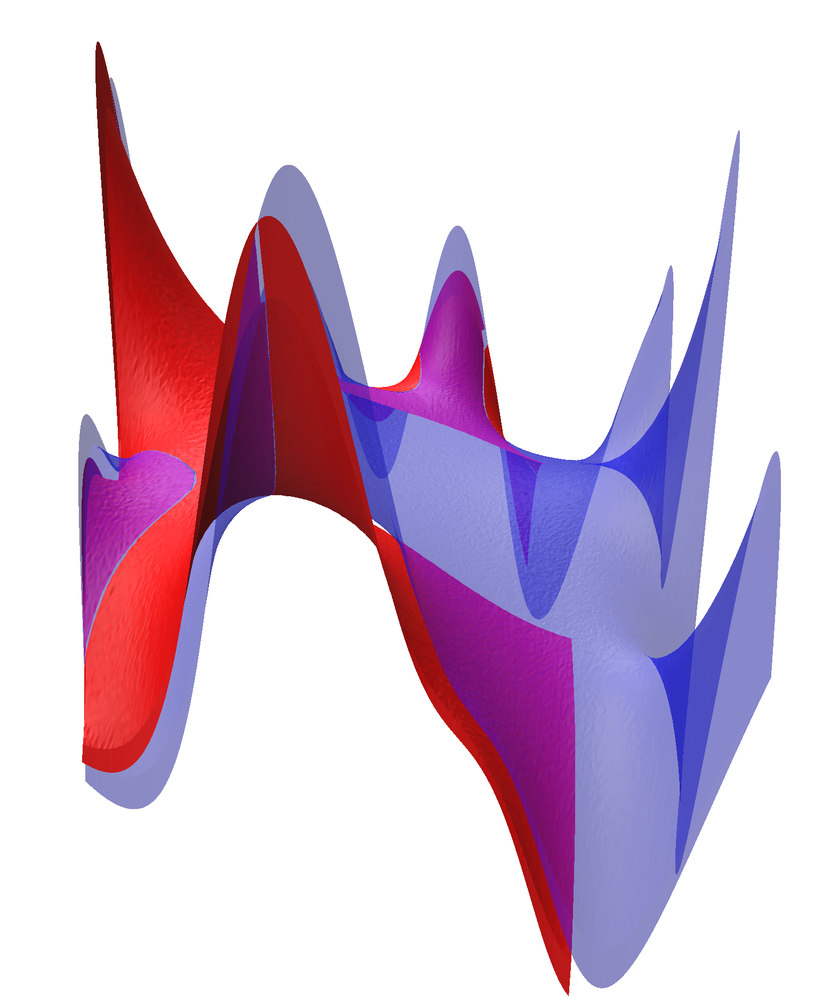}
\includegraphics[width=0.3\linewidth]{figures/invprobs_referencesolutions/nonlin_21coeffs_coeff13val10}
\caption{Optimization process over a 9-dimensional \emph{latent space} for a \emph{nonlinear} inverse problem with \emph{clean} data. Again, the operator network in the last row of Table~\ref{tab_ONtraininginfo_coeffs21} (21 input coefficients, 59049 output DoFs) is used, but here together with the ``sd = 0'' decoder from the right frame in Figure~\ref{fig_nonlindata_gauss_ldim3_gb7}. The decoder maps from a 9-dimensional latent space to a 21-dimensional coefficient space. The last frame shows $\omega$ and the reference solution used for the data which was obtained by taking $p_{14} = 10$ and all other $p_n$'s = 0. The penultimate frame shows the optimization's MSE-converged reconstruction of the reference solution. The MSE converged after 1481 iterations with the Adam optimizer with a step size = 0.1. This took 68.9 s on an Apple M1 CPU. The MSE's between the reference solution and the converged reconstruction are: on $\omega$ (used in optimization), $\text{MSE}_{\omega}$ = 2.22e-3; on the convex hull of $\omega$, $\text{MSE}_{\text{co}(\omega)}$ = 1.52e-3; and on its complement, $\text{MSE}_{\text{co}(\omega)^c}$ = 4.83e-3.}
\label{fig_nonlinnonlin_optprocess_unpertdec_unpertomega}
\end{figure}
\begin{figure}
\centering
\includegraphics[width=0.3\linewidth]{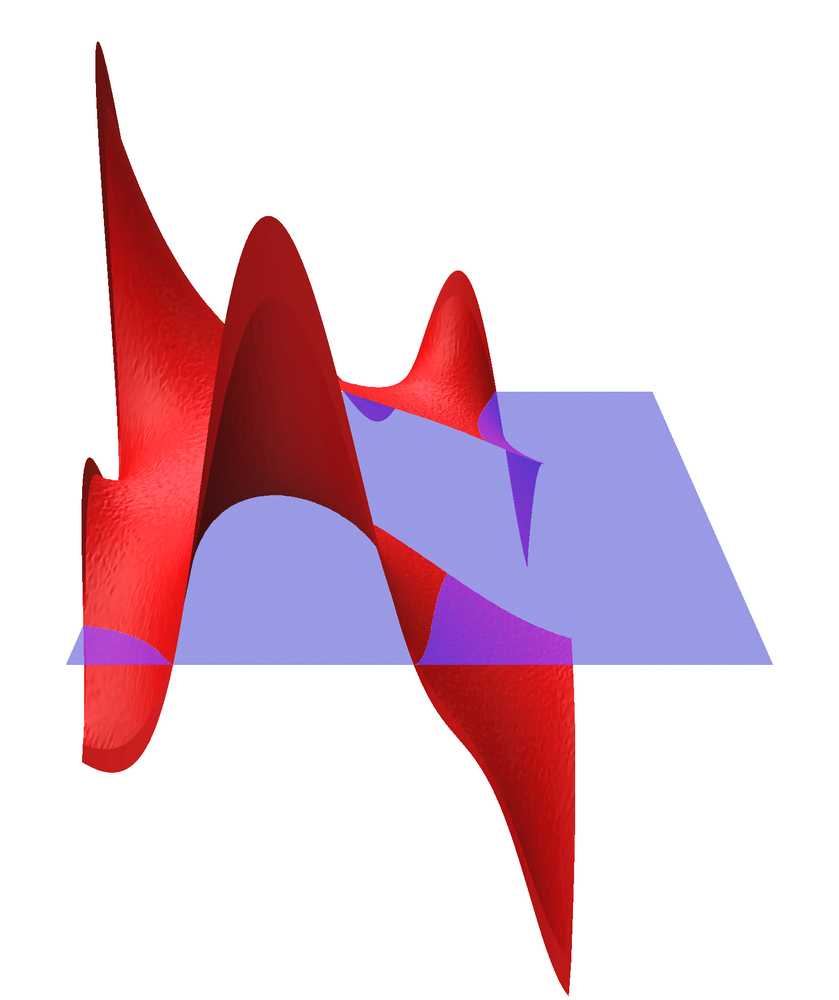}
\includegraphics[width=0.3\linewidth]{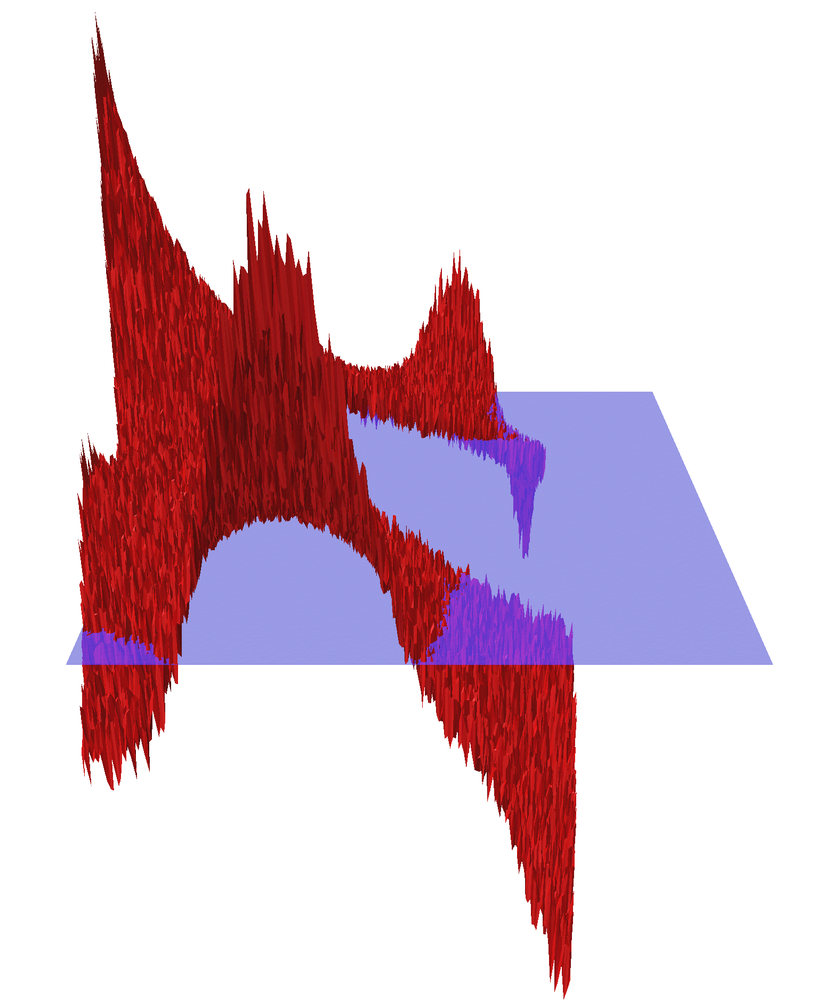}
\includegraphics[width=0.3\linewidth]{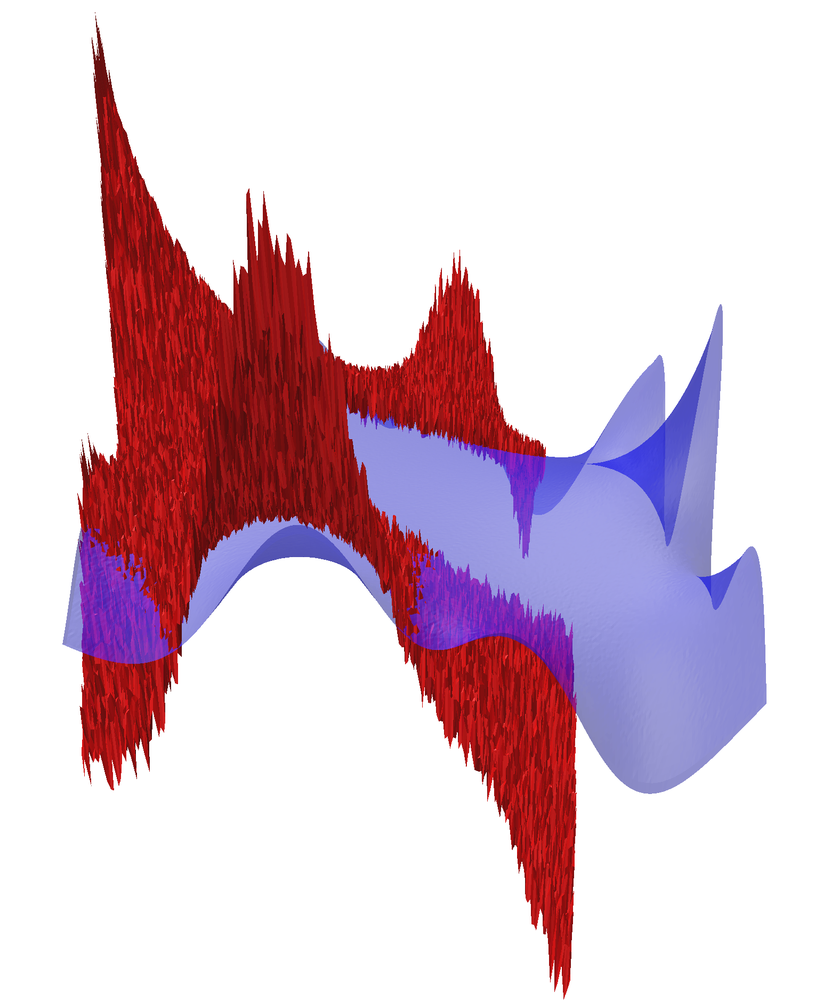}
\includegraphics[width=0.3\linewidth]{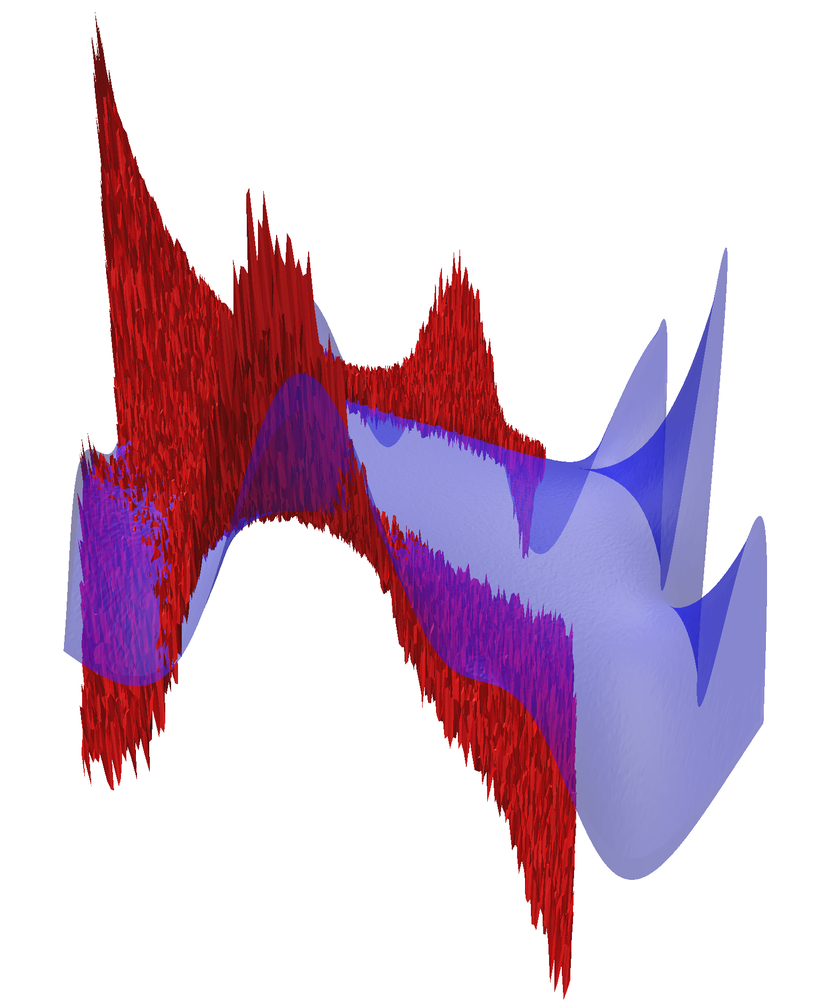}
\includegraphics[width=0.3\linewidth]{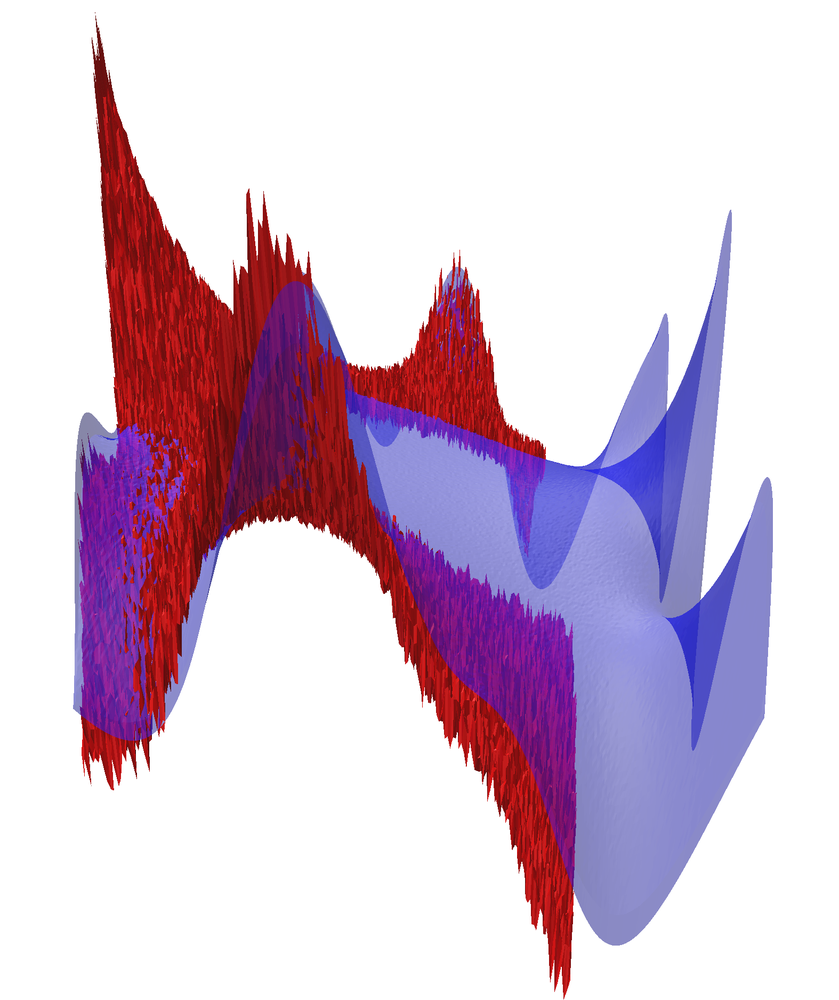}
\includegraphics[width=0.3\linewidth]{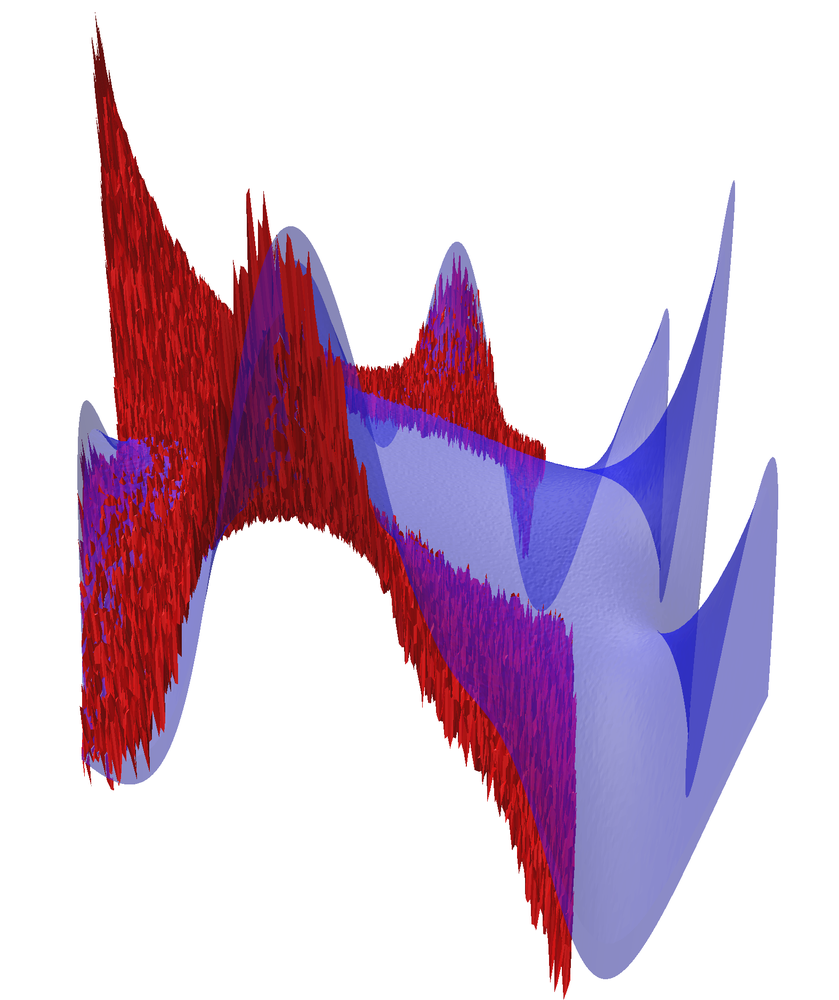}
\includegraphics[width=0.3\linewidth]{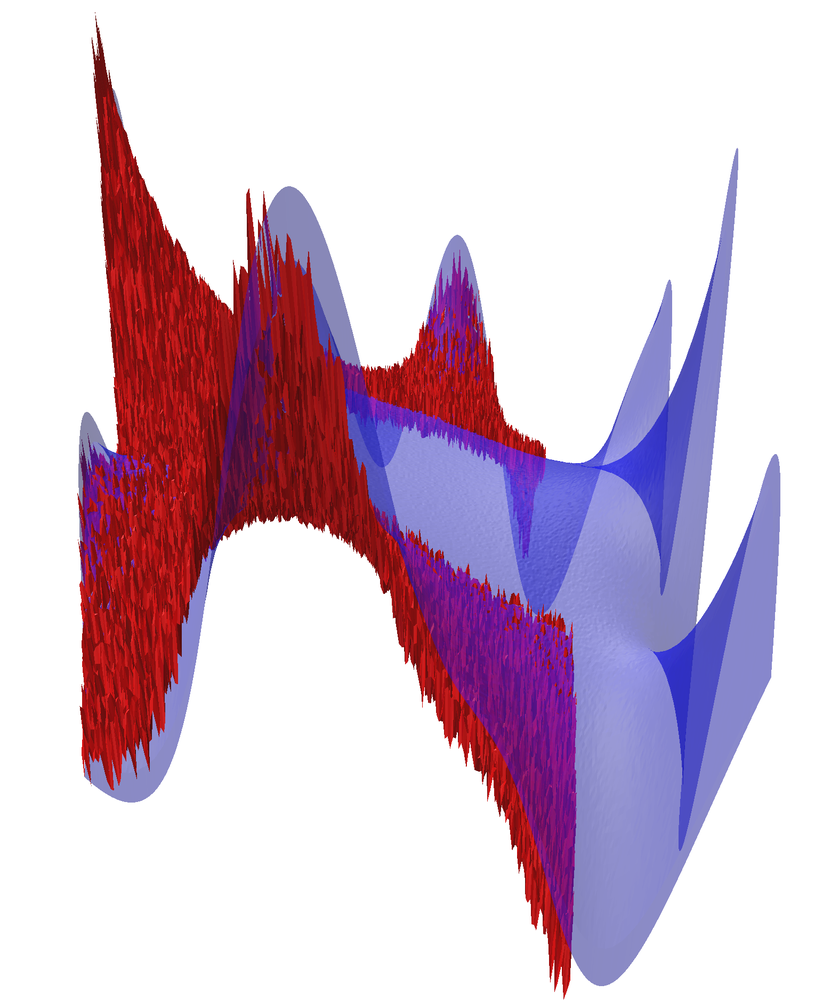}
\includegraphics[width=0.3\linewidth]{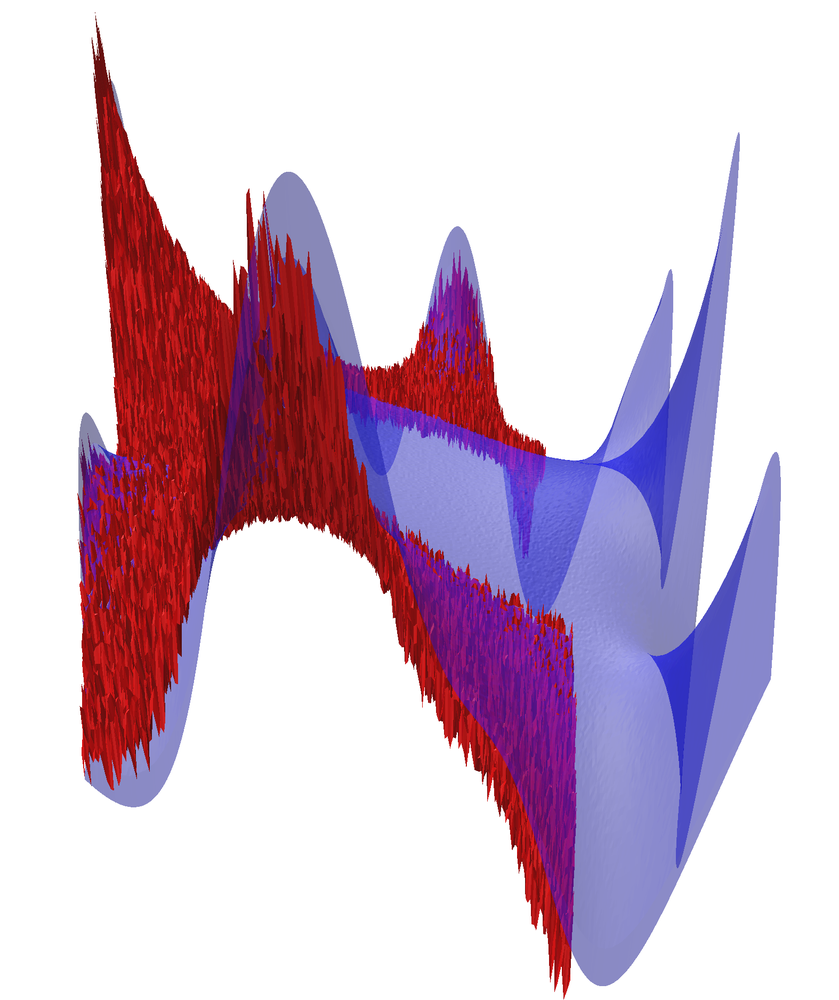}
\includegraphics[width=0.3\linewidth]{figures/invprobs_referencesolutions/nonlin_21coeffs_coeff13val10}
\caption{Optimization process over a 9-dimensional \emph{latent space} for a \emph{nonlinear} inverse problem with noisy data. Again, the operator network in the last row of Table~\ref{tab_ONtraininginfo_coeffs21} (21 input coefficients, 59049 output DoFs) is used together with the ``sd = 0'' decoder from the right frame in Figure~\ref{fig_nonlindata_gauss_ldim3_gb7}. The decoder maps from a 9-dimensional latent space to a 21-dimensional coefficient space. The unperturbed data is shown in the first frame. The second frame is the same as the first but with added noise sampled from $\mcU(-0.05, 0.05)$. The last frame shows $\omega$ and the reference solution used for the data which was obtained by taking $p_{14} = 10$ and all other $p_n$'s = 0. The penultimate frame shows the optimization's MSE-converged reconstruction of the reference solution. The MSE converged after 1018 iterations with the Adam optimizer with a step size = 0.1. This took 47.5 s on an Apple M1 CPU. The MSE's between the reference solution and the converged reconstruction are: on $\omega$ (used in optimization), $\text{MSE}_{\omega}$ = 3.07e-3; on the convex hull of $\omega$, $\text{MSE}_{\text{co}(\omega)}$ = 2.37e-3; and on its complement, $\text{MSE}_{\text{co}(\omega)^c}$ = 5.25e-3.}
\label{fig_nonlinnonlin_optprocess_unpertdec}
\end{figure}
\begin{figure}
\centering
\includegraphics[width=0.3\linewidth]{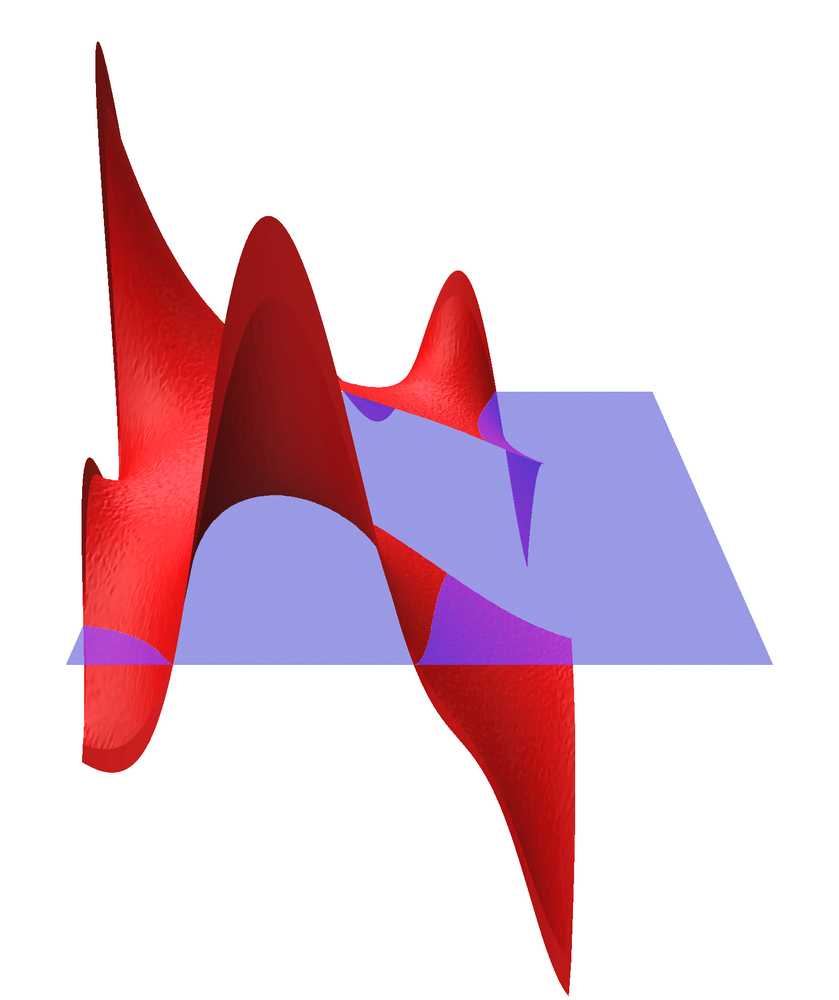}
\includegraphics[width=0.3\linewidth]{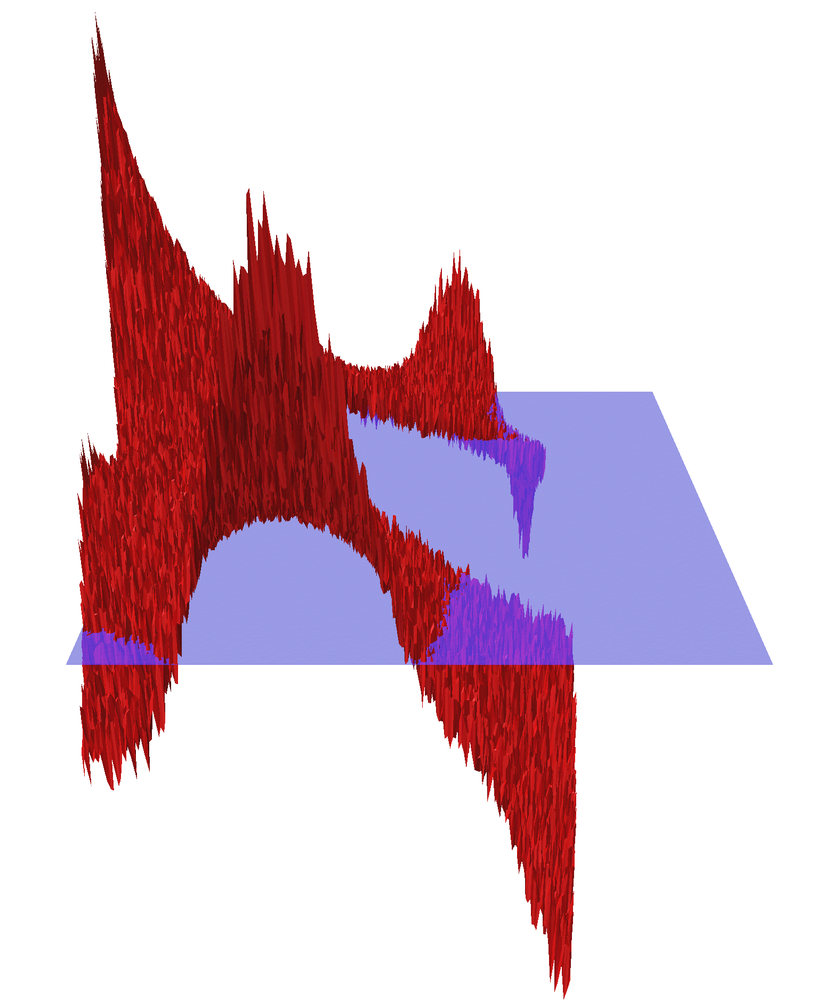}
\includegraphics[width=0.3\linewidth]{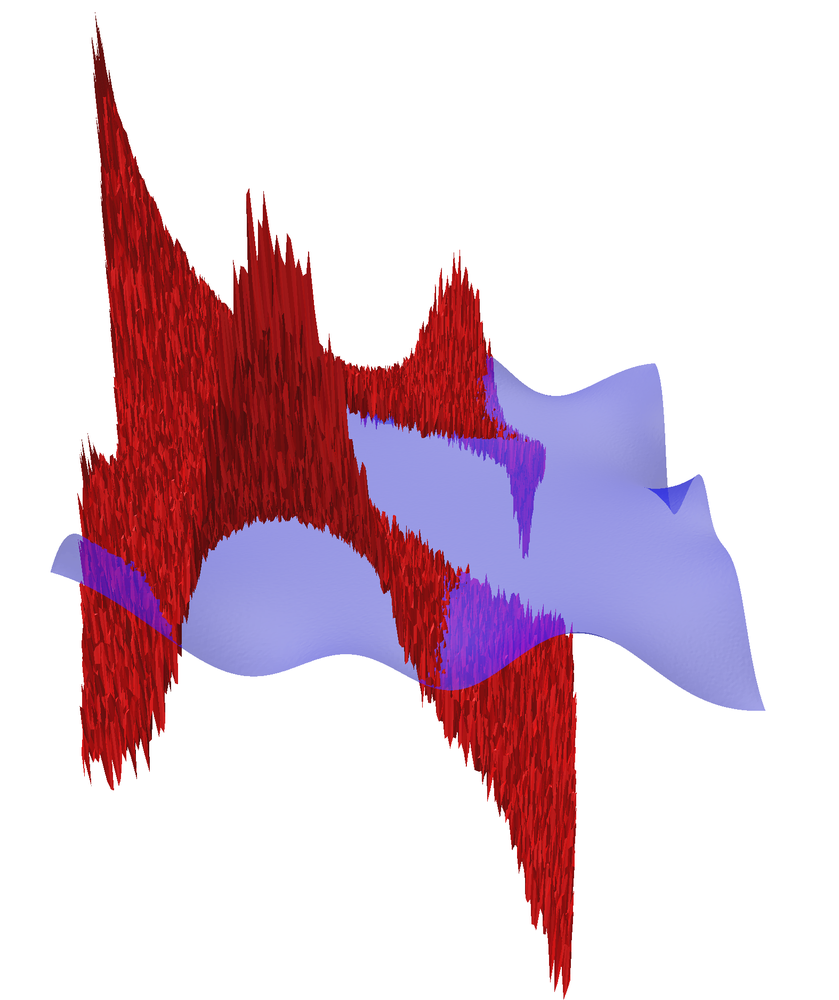}
\includegraphics[width=0.3\linewidth]{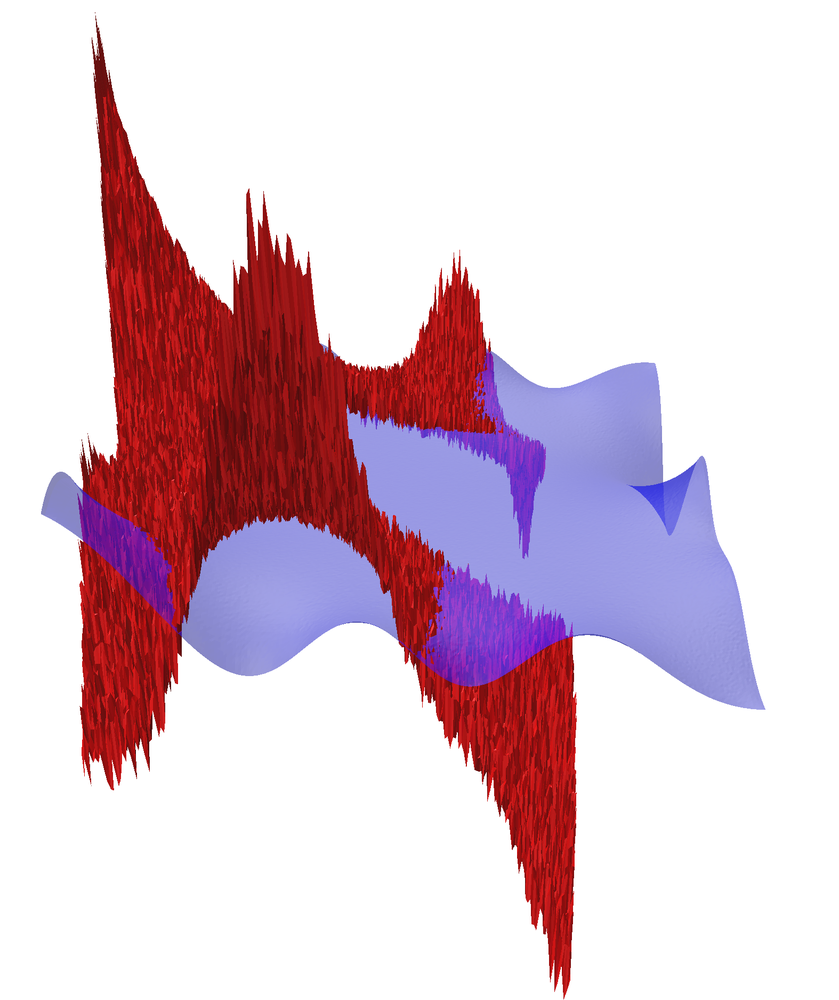}
\includegraphics[width=0.3\linewidth]{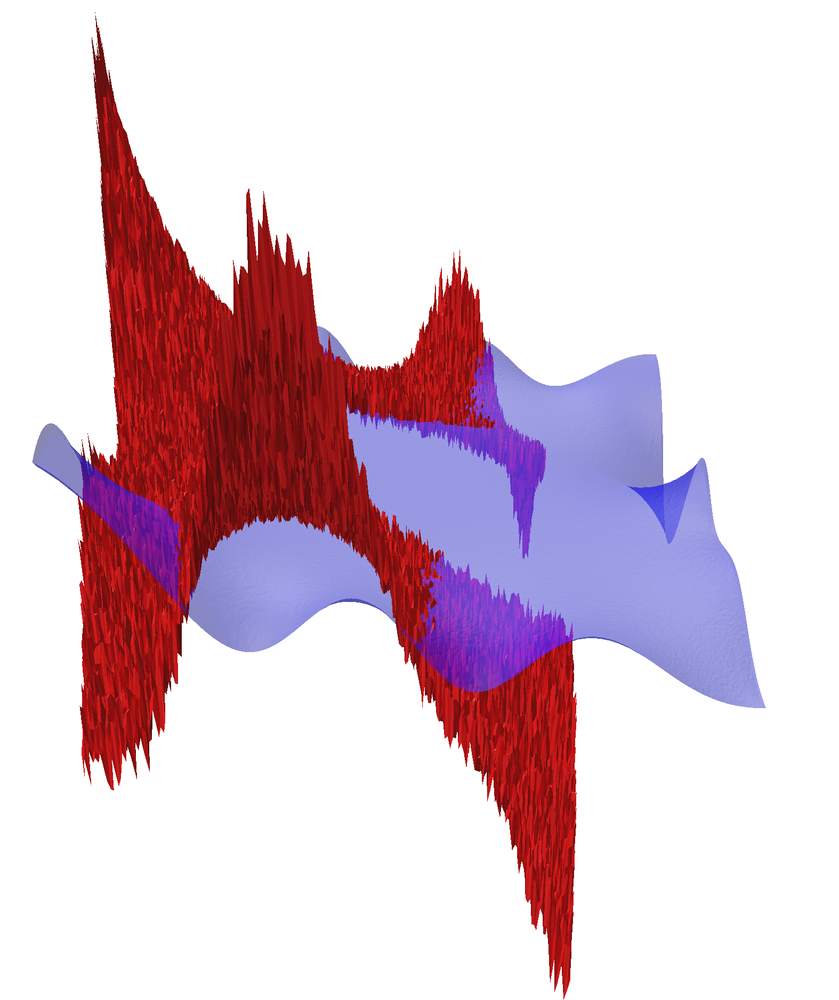}
\includegraphics[width=0.3\linewidth]{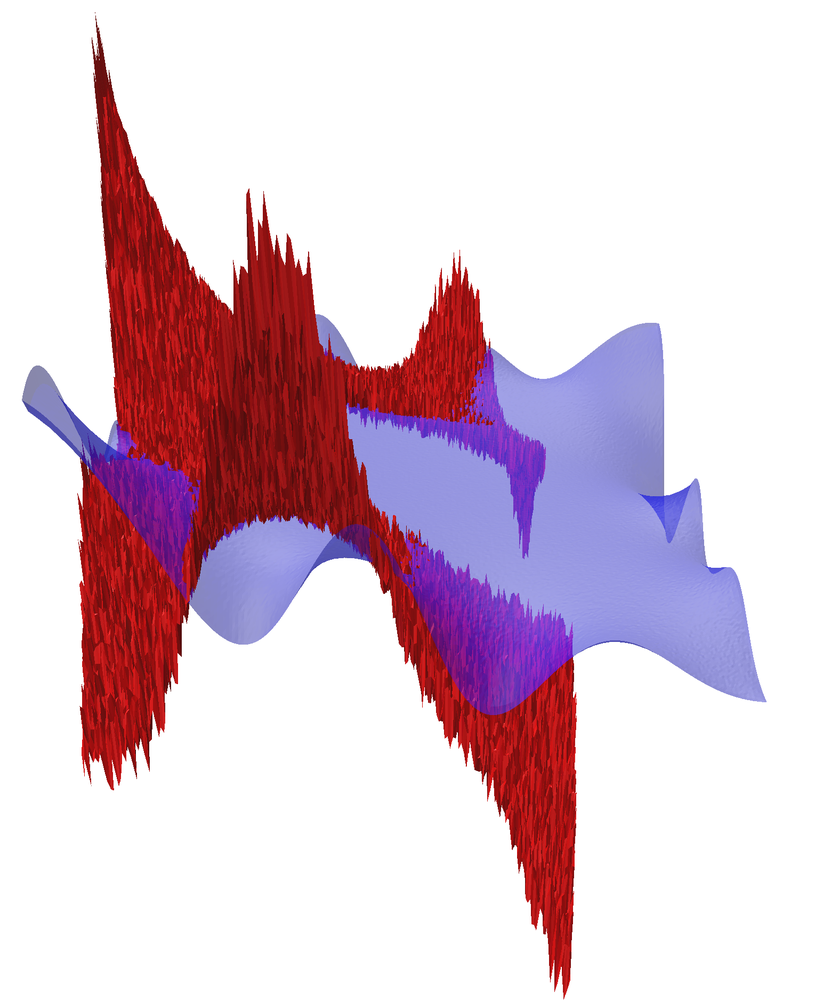}
\includegraphics[width=0.3\linewidth]{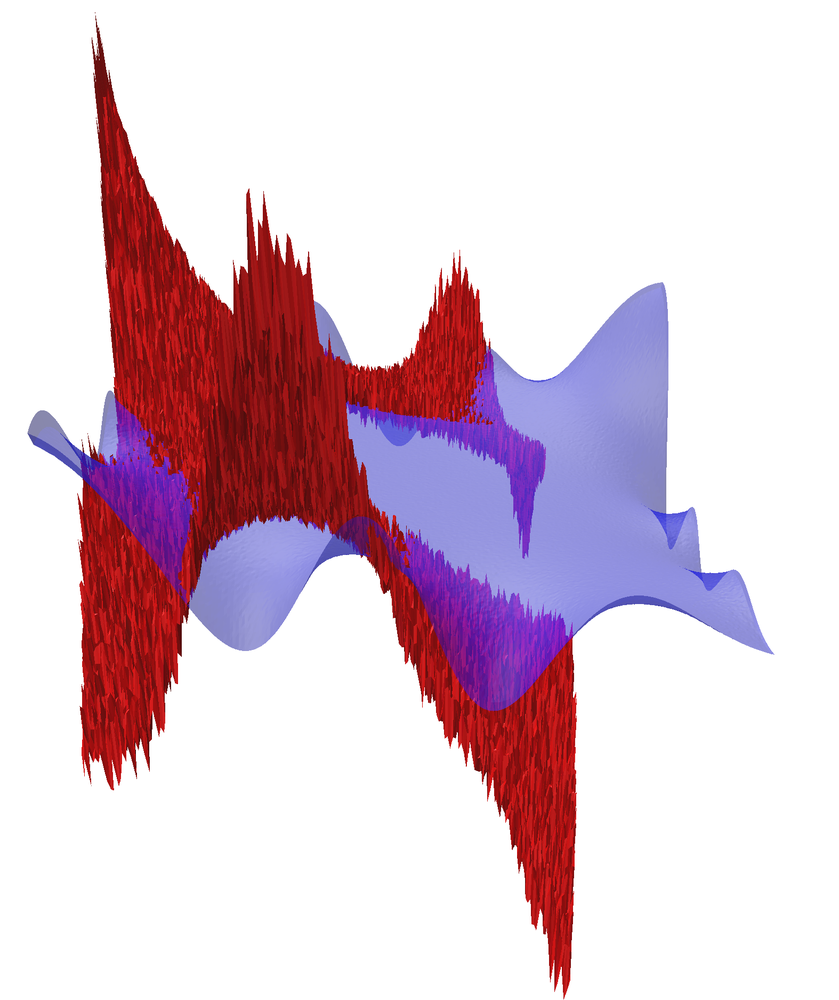}
\includegraphics[width=0.3\linewidth]{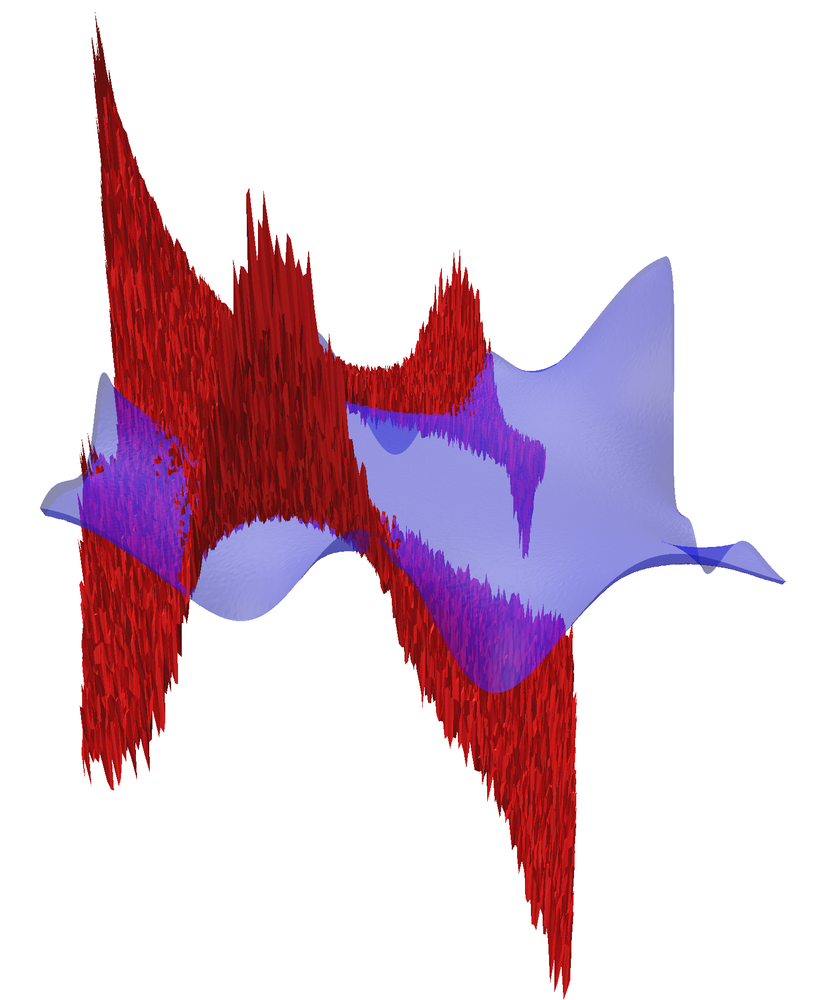}
\includegraphics[width=0.3\linewidth]{figures/invprobs_referencesolutions/nonlin_21coeffs_coeff13val10}
\caption{Optimization process over a 9-dimensional \emph{latent space} for a \emph{nonlinear} inverse problem with noisy data. Again, the operator network in the last row of Table~\ref{tab_ONtraininginfo_coeffs21} (21 input coefficients, 59049 output DoFs) is used, but here together with the ``sd = 0.15'' decoder from the right frame in Figure~\ref{fig_nonlindata_gauss_ldim3_gb7}. The decoder maps from a 9-dimensional latent space to a 21-dimensional coefficient space. The unperturbed data is shown in the first frame. The second frame is the same as the first but with added noise sampled from $\mcU(-0.05, 0.05)$. The last frame shows $\omega$ and the reference solution used for the data which was obtained by taking $p_{14} = 10$ and all other $p_n$'s = 0. The penultimate frame shows the optimization's MSE-converged reconstruction of the reference solution. The MSE converged after 212 iterations with the Adam optimizer with a step size = 0.1. This took 11.0 s on an Apple M1 CPU. The MSE's between the reference solution and the converged reconstruction are: on $\omega$ (used in optimization), $\text{MSE}_{\omega}$ = 1.99e-2; on the convex hull of $\omega$, $\text{MSE}_{\text{co}(\omega)}$ = 1.44e-2; and on its complement, $\text{MSE}_{\text{co}(\omega)^c}$ = 5.51e-2.}
\label{fig_nonlinnonlin_optprocess}
\end{figure}

\begin{table}
  \caption{Summary of optimization results for all \emph{nonlinear} inverse problems using an operator network. All problems have the same reference solution and use the operator network in the last row of Table~\ref{tab_ONtraininginfo_coeffs21} (21 input coefficients, 59049 output DoFs). The optimization processes for the problems are presented in Figures \ref{fig_nonlinlin_optprocess}, \ref{fig_nonlinnonlin_optprocess_unpertdec_unpertomega} -- \ref{fig_nonlinnonlin_optprocess}, respectively. In the table, ``Op'' means the operator network, ``dec'' means decoder, ``sd = x'' means what perturbation was added to the training data for the decoder, and ``$\delta_\omega$'' means that noisy data was used for the inverse problem.}
  \label{table:nonlinear_opt_res}
    \centering
    \begin{tabular}{lccccc}
    \toprule
    Configuration & Iterations & Avg iter time & $\text{MSE}_{\omega}$ & $\text{MSE}_{\text{co}(\omega)}$ & $\text{MSE}_{\text{co}(\omega)^c}$ \\
    \midrule
    Op, $\delta_{\omega}$ & 2843 & 4.93e-2 s & 8.36e-4 & 8.38e-4 & 1.91e-3 \\
    Op + dec ``sd = 0'' & 1481 & 4.65e-2 s & 2.22e-3 & 1.52e-3 & 4.83e-3 \\ 
    Op + dec ``sd = 0'', $\delta_{\omega}$ & 1018 & 4.67e-2 s & 3.07e-3 & 2.37e-3 & 5.25e-3 \\ 
    Op + dec ``sd = 0.15'', $\delta_{\omega}$ & 212 & 5.19e-2 s & 1.99e-2 & 1.44e-2 & 5.51e-2 \\ 
    \bottomrule
  \end{tabular}%
\end{table}

\section{Conclusions}
The regularization of severely ill-posed inverse problems using large data sets and stabilized finite element methods was considered and shown to be feasible both for linear and nonlinear problems. In the linear case, a fairly complete theory for the approach exists, and herein, we complemented previous work with the design and analysis of a reduced-order model. In the linear case, a combination of POD for the data reduction and reduced model method for the PDE-solution was shown to be a rigorous and robust approach that effectively can improve stability from logarithmic to linear in the case where the data is drawn from some finite-dimensional space of moderate dimension. To extend the ideas to nonlinear problems we introduced a machine learning framework, both for the data compression and the reduced model. After successful training, this resulted in a very efficient method for the solution of the nonlinear inverse problem. The main observations were the following:

\begin{enumerate}
\item The combination of analysis of the inverse problem, numerical analysis of finite element reconstruction methods, and data compression techniques allows for the design of robust and accurate methods in the linear case.
\item Measured data can be used to improve stability, provided a latent data set of moderate size can be extracted from the data cloud.
\item Machine learning can be used to leverage the observations in the linear case to nonlinear inverse problems and data assimilation and results in fast and stable reconstruction methods.
\end{enumerate}

The main open questions are related to how the accuracy of the machine learning approach can be assessed and controlled through network design and training. For recent work in this direction, we refer to \cite{MR4716388}.
\bigskip
\paragraph{Acknowledgements.} This research was supported in part by the Swedish Research
Council Grant, No. \ 2021-04925, and the Swedish
Research Programme Essence. EB acknowledges funding from EPSRC grants EP/T033126/1 and EP/V050400/1.

The GPU computations were enabled by resources provided by the National Academic Infrastructure for Supercomputing in Sweden (NAISS), partially funded by the Swedish Research Council through grant agreement No.\ 2022-06725.

\bibliographystyle{habbrv}
\footnotesize{
\bibliography{inverse}
}

\bigskip
\bigskip
\noindent
\footnotesize {\bf Authors' addresses:}

\smallskip
\noindent
Erik Burman,  \quad \hfill \addressuclshort\\
{\tt e.burman@ucl.ac.uk}

\smallskip
\noindent
Mats G. Larson,  \quad \hfill \addressumushort\\
{\tt mats.larson@umu.se}

\smallskip
\noindent
Karl Larsson,  \quad \hfill \addressumushort\\
{\tt karl.larsson@umu.se}

\smallskip
\noindent
Carl Lundholm, \quad \hfill \addressumushort\\
{\tt carl.lundholm@umu.se}

\end{document}